\documentclass[11pt]{article}
\usepackage[letterpaper]{geometry}
\usepackage{fullpage}
\usepackage{amsmath,amsthm,amssymb,bbm,fullpage,tikz}
\usepackage{mathtools}
\usepackage{cases}
\usepackage{graphicx}
\usepackage{tabularx}
\usepackage{microtype}
\usepackage{enumitem}
\usepackage{authblk}
\usepackage{url}
\usepackage[colorlinks,citecolor=blue,urlcolor=blue,linkcolor=blue,linktocpage=true]{hyperref}
\usepackage{sectsty}
\allsectionsfont{\bfseries\sffamily}

\usepackage[colorinlistoftodos]{todonotes}

\usepackage{cleveref}
\crefformat{equation}{(#2#1#3)}
\crefrangeformat{equation}{(#3#1#4) to~(#5#2#6)}
\crefname{equation}{}{}
\Crefname{equation}{}{}

\usepackage[authoryear,round]{natbib}

\let\hat\widehat
\let\tilde\widetilde

\newtheorem{theorem}{Theorem}
\newtheorem{lemma}[theorem]{Lemma}
\newtheorem{corollary}[theorem]{Corollary}

\newtheorem{proposition}[theorem]{Proposition}
\newtheorem{definition}{Definition}
\newtheorem*{definition*}{Definition}
\newtheorem*{remark*}{Remark}
\crefname{definition}{\textbf{definition}}{definitions}
\Crefname{definition}{Definition}{Definitions}
\crefname{assumption}{\textbf{assumption}}{assumptions}
\Crefname{assumption}{Assumption}{Assumptions}

\DeclareMathOperator*{\argmax}{argmax}
\DeclareMathOperator*{\argmin}{argmin}

\newcommand{\wS}{\widehat{S}}

\newenvironment{enum}{
\begin{enumerate}
  \setlength{\itemsep}{1pt}
  \setlength{\parskip}{0pt}
  \setlength{\parsep}{0pt}
}{\end{enumerate}}

\begin{document}

\begin{center}
\textsf{\textbf{\Large Bootstrapping and Sample Splitting}}\\
\textsf{\textbf{\Large For High-Dimensional, Assumption-Free Inference}}\\
\textsf{\textbf{Alessandro Rinaldo, Larry Wasserman, Max G'Sell and Jing Lei}}\\
\vspace{.11pt}
\textsf{\textbf{Carnegie Mellon University}}\\
\vspace{.11pt}
\textsf{\textbf{April 1 2018}}
\end{center}

\begin{quote}
{\em Several new methods have been recently proposed 
for performing valid inference after model selection.
An older method is sampling splitting:
use part of the data for model selection and the rest for inference.
In this paper we revisit sample splitting combined with the bootstrap (or the Normal approximation).
We show that this leads to a simple, 
assumption-free approach to inference and
we establish results on the accuracy of the method.
In fact, we find new bounds on the accuracy of the
bootstrap and the Normal approximation for general nonlinear parameters
with increasing dimension which we then use to assess the accuracy of regression inference.
We define new parameters that measure variable importance
and
that can be inferred with greater accuracy than the usual regression coefficients.
Finally, we elucidate an inference-prediction trade-off:
splitting
increases the accuracy and robustness of inference but
can decrease the accuracy of the predictions.}
\end{quote}

``Investigators who use [regression] are not paying adequate attention
to the connection - if any - between the models and the phenomena they
are studying. ...  By the time the models are deployed, the scientific
position is nearly hopeless. Reliance on models in such cases is
Panglossian ...''

		---David Freedman

\section{Introduction}

We consider the problem of carrying out assumption-free statistical inference
after model selection for high-dimensional linear regression.  This is now a
large topic and a variety of approaches have been considered under different
settings
-- an overview of a subset of these can be found in
\cite{dezeure2015high}.  We defer a detailed discussion of the literature and
list of references until Section \ref{sec:related}.

In this paper, we will use linear models but we do not assume that the true regression function is linear.
We show the following:
\begin{enumerate}
\item Inference based on sample splitting followed by the bootstrap (or Normal approximation)
gives assumption-free, robust confidence intervals under very weak
assumptions.
No other known method gives the same inferential guarantees.
\item The usual regression parameters are not the best choice of parameter to estimate
in the weak assumption case. We 
propose new parameters, called LOCO (Leave-Out-COvariates) parameters,
that are interpretable, general and can be estimated accurately.
\item There is a trade-off between prediction accuracy and inferential accuracy.
\item We provide new bounds on the accuracy of the Normal approximation and the bootstrap
to the distribution of the projection parameter 
(the best linear predictor)
when the dimension increases and the model is wrong.
We need these bounds since we will use Normal approximations or the bootstrap after choosing the model.
In fact, we provide new general bounds on Normal approximations for
nonlinear parameters with increasing dimension.
This gives new insights on the accuracy of inference in high-dimensional situations.
In particular, the accuracy of the Normal approximation for the standard regression parameters is very
poor while the approximation is very good for LOCO parameters.
\item The accuracy of the bootstrap can be improved by using an alternative version that
we call the image bootstrap. However, this version is computationally expensive.
The image bootstrap is discussed in the appendix.
\item We show that the law of the projection parameter cannot be consistently estimated without
sample splitting.
\end{enumerate}

\vspace{11pt}

We want to emphasize that we do not claim that the LOCO parameter is optimal in any sense.
We just aim to show that there exist alternatives to the usual parameters that,
when the linear model is not true,
(i) are more interpretable and (ii) can be inferred more accurately.

\subsubsection*{Problem Setup and Four (Random) Parameters that Measure Variable Importance}

We consider a distribution-free regression framework, where the random pair
$Z = (X,Y) \in \mathbb{R}^d \times \mathbb{R} $ of $d$-dimensional covariates
and response variable has an unknown distribution $P$ belonging to a large non-parametric class $\mathcal{Q}_n$ of
probability distributions on $\mathbb{R}^{d+1}$. 
We make no assumptions on the regression function $x \in \mathbb{R}^d \mapsto \mu(x) = 
\mathbb{E}\left[ Y | X = x \right]$ describing the relationship  between the
vector of covariates and the expected value of the response variable. In particular, we do not require it to be linear.

We observe $\mathcal{D}_n = (Z_1,\ldots, Z_n)$, an i.i.d. sample of size $n$
from some $P \in \mathcal{Q}_n$,
where $Z_i = (X_i,Y_i)$, for $i = 1,\ldots,n$. 
We apply to the data a procedure $w_n$, which
returns both a subset  of the coordinates and an estimator of the regression
function over the selected coordinates. 
Formally,
\[
    \mathcal{D}_n \mapsto w_n(\mathcal{D}_n) = \left(\widehat{S},
    \widehat{\mu}_{\widehat{S}}\right),
\]
where $\widehat{S}$, the selected model, is a random, nonempty  subset of
$\{1,\ldots,d\}$ and $\widehat{\mu}_{\widehat{S}}$ is an estimator of the regression function 
$x \in \mathbb{R}^d \mapsto
\mathbb{E}\left[ Y | X_{\widehat{S}} = x_{\widehat{S}} \right]$
restricted to the selected covariates $\widehat{S}$, where for $x \in
\mathbb{R}^d$, $x_{\wS} = (x_j, j \in \wS)$ and $(X,Y) \sim P$, independent of $\mathcal{D}_n$.

The model selection and estimation steps comprising the procedure $w_n$ need not be related to
each other, and can each be accomplished by any appropriate method. The only
assumption we
impose on $w_n$ is that the size of the selected model be under our control; that is, $
0 < |\widehat{S}| \leq k $, for a pre-defined positive
integer $k \leq d$
where $k$ and $d$ can both increase with sample size. 
For example, $\widehat{S}$ may be defined as the set of
$k$ covariates with the highest linear correlations with the response and
$\hat{\mu}_{\widehat{S}}$ may be any non-parametric estimator of the regression
function over the coordinates in $\widehat{S}$ with bounded range.
Although our framework allows for arbitrary estimators of the regression function,
we will be focussing on linear estimators:
$\widehat{\mu}_{\widehat{S}}(x) = \widehat{\beta}_{\widehat{S}}^\top
x_{\widehat{S}} $, where $\widehat{\beta}_{\widehat{S}}$ is any estimator of the of the linear regression
coefficients for the selected variables -- such as ordinary least squares on the
variables in $\widehat{S}$. In particular, $\widehat{\beta}_{\widehat{S}}$
may arise from fitting a sparse linear model, such as the lasso or
stepwise-forward regression, in which case estimation of the
regression parameters and model selection can be accomplished simultaneously
with one procedure.

It is important to emphasize that, since we impose minimal assumptions  on the
class $\mathcal{Q}_n$ of data generating distribution and allow for arbitrary model selection and estimation procedures
$w_n$, we  will not assume anything about the quality of the output returned by the procedure
$w_n$. In particular, the selected model $\widehat{S}$ needs not be a good approximation of any
optimal model, however optimality may be defined. Similarly, $\hat{\mu}_{\widehat{S}}$ may
not be a consistent estimator of the regression function restricted to
$\widehat{S}$.  Instead, our concern is to provide statistical guarantees for 
various criteria of significance for the selected model
$\widehat{S}$, uniformly over the choice of $w_n$ and over all the
distributions $P \in \mathcal{Q}_n$. We will accomplish this goal by producing
confidence sets for four {\it random}
parameters in $\mathbb{R}^{ \widehat{S}}$, each providing a
different assessment of the level of statistical significance of the variables
in $\widehat{S}$ from a purely {\it predictive} standpoint. 
All of the random parameters under consideration are functions of the data generating distribution $P$, of the
sample $\mathcal{D}_n$ and, therefor, of its size $n$ and, importantly, of the model selection and estimation procedure
$w_n$. 
Below, $(X,Y)$ denotes a draw from $P$, independent of the sample
$\mathcal{D}_n$. Thus the distribution of $(X,Y)$ is the same as their
conditional distribution given $\mathcal{D}_n$.

\begin{itemize}
    \item {\bf The projection parameter $\beta_{\widehat{S}}$.} The linear projection parameter
	$\beta_{\widehat{S}}$ is defined to be the vector of coefficients of the
	best linear predictor of $Y$  using $X_{\widehat{S}}$: 
\[
    \beta_{\widehat{S}} = \argmin_{\beta \in \mathbb{R}^{\widehat{S}}}
    \mathbb{E}_{X,Y} \left[ (Y- \beta^\top X_{\widehat{S}})^2 \right],
\]
where $\mathbb{E}_{(X,Y)}$ denote the expectation with respect to the
distribution of $(X,Y)$. The terminology projection parameters refers to the
fact  that $X^\top \beta_{\wS}$ is the projection of $Y$ into
the linear space of all random variables that can be obtained as linear
functions of $X_{\wS}$. For a through discussion and an analysis of the properties of such parameters see \cite{buja2015models}.
More generally, this type of quantities are also studied in 
\cite{lee2016exact, taylor2014exact, berk2013valid, wasserman2014}.
Note that the projection parameter is well-defined
even though the true regression function  $\mu$ is not linear.
Indeed, it is immediate that
\begin{equation}\label{eq::projection-parameter}
    \beta_{\wS} = \Sigma_{\wS}^{-1}\alpha_{\wS}
\end{equation}
where
$\alpha_{\wS} = (\alpha_{\wS}(j):\ j\in \wS)$,
$\alpha_{\wS}(j) = \mathbb{E}[Y X_{\wS}(j)]$
and $\Sigma_{\wS} = \mathbb{E}[X_{\wS} X_{\wS}^\top]$.
We remark that the regression projection parameter depends only on the selected model $\wS$,
and not any estimate $\hat{\mu}_{\wS}$ of the regression function on the
coordinates in $\wS$ that may be implemented in $w_n$. 
\item {\bf The LOCO parameters $\gamma_{\wS}$ and $\phi_{\wS}$.}
    Often, statisticians are interested in $\beta_{\wS}$ 
as a measure of the importance of the selected covariates.
But, of course, there are other ways to measure variable importance.
We now define two such parameters, which we refer to as {\em Leave Out COvariate Inference -- or LOCO -- parameters}, which were originally defined in
\cite{lei2016distribution} and are similar to the variable importance measures used in random forests.
The first LOCO parameter is
$\gamma_{\wS} = (\gamma_{\wS}(j):\ j\in \wS)$,
where
\begin{equation}\label{eq:gamma.j}
\gamma_{\wS}(j) = \mathbb{E}_{X,Y}\Biggl[|Y-\hat\beta_{\wS(j)}^\top X_{\wS(j)}|-
|Y-\hat\beta_{\wS}^\top X_{\wS}| \Big| \mathcal{D}_n \Biggr].
\end{equation}
In the above expression, $\hat\beta_{\wS}$ is {\it any} estimator of the projection parameter $\beta_{\wS}$ and  $\wS(j)$ and
$\hat\beta_{\wS(j)}$ are obtained 
by re-running the model selection and estimation procedure after removing the
$j^{\mathrm{th}}$ covariate from the data $\mathcal{D}_n$. To be clear, for each $j \in wS$, $\wS(j)$ is a subset of size $k$ of $\{1,\ldots,d\} \setminus \{j\}$.
Notice that the selected model can be different when covariate $X_j$ is
held out from the data, so that the intersection between $\wS(j)$ and $\wS$ can be quite smaller than $k-1$.
The interpretation of $\gamma_{\wS}(j)$ is simple:
it is the increase in prediction error by not having access to $X(j)$ (in both
the model selection and estimation steps).
Of course, it is possible to extend the definition of this parameter by leaving out several variables
from $\wS$ at once without additional conceptual difficulties.\\
The parameter $\gamma_{\wS}$ has several advantages over the projection
parameter $\beta_{\wS}$: it is more interpretable since it refers directly to prediction error and
we shall see that the accuracy of the Normal approximation
and the bootstrap is much higher.
Indeed, we believe that the widespread focus on $\beta_{\wS}$
is mainly due to the fact that statisticians are used to
thinking in terms of cases where the linear model is assumed to be correct.\\
The second type of LOCO parameters that we consider are the median LOCO parameters
$\phi_{\wS} = (\phi_{\wS}(j):\ j\in {\wS})$
with
\begin{equation}\label{eq:median.LOCO}
    \phi_{\wS}(j) = {\rm median}\Biggl[|Y-\hat\beta_{\wS(j)}^\top X_{\wS}|-
    |Y-\hat\beta_{\wS}^\top X_{\wS}|\,\Biggr],
\end{equation}
where the median is over the conditional distribution of $(X,Y)$ given
$\mathcal{D}_n$. Though one may simply regard $\phi_{\wS}$ as a robust version of
$\gamma_{\wS}$, we find that inference for $\phi_{\wS}$ will remain valid under weaker
assumptions that the ones needed for $\gamma_{\wS}$. 
Of course,
as with $\gamma_{\wS}$, we may leave out multiple covariate at the same time.
\item {\bf The prediction parameter $\rho_{\wS}$}.
It is also of interest to obtain an omnibus parameter that
measures how well the selected model will predict
future observations.
To this end, we define the future predictive error as
\begin{equation}
    \rho_{\wS} = \mathbb{E}_{X,Y}\Bigl[| Y - \hat\beta_{\wS}^\top X_{\wS}|\,   \Bigr],
\end{equation}
where $\widehat{\beta}_{\wS}$ is any estimator the projection parameters $\beta_{\wS}$.
\end{itemize}

{\bf Remarks.}
\begin{enumerate}
	\item 
The LOCO and prediction parameters do not require linear estimators.
For example we can define
$$
\gamma_{\wS}(j) = \mathbb{E}_{X,Y}\Biggl[ |Y-\hat\mu_{\wS(j)}(X_{{\wS}(j)})| -
|Y-\hat\mu_{\wS}(X_{\wS})|\ \Biggr], \quad j \in \wS,
$$
where $\hat\mu_{\wS}$ is any regression estimator 
restricted to the coordinates in $\wS$ and
$\hat\mu_{\wS(j)}$ is the estimator obtained after performing a new model
selection process and then refitting without covariate $j
\in \wS$. Similarly, we could have
\[
    \rho_{\wS} = \mathbb{E}_{X,Y}\Bigl[| Y - \hat{\mu}_{\wS}(X_{\wS})|\,
    \Bigr],
\]
for an arbitrary estimator $\hat{\mu}_{\wS}$.
For simplicity, we will focus on linear estimators, although our results about the
LOCO and prediction parameters hold even in this more general setting.
\item It is worth reiterating that the projection and LOCO parameters are only defined over the coordinates in $\wS$,  the set of variables that
are chosen in the model selection phase.
If a variable is not selected then the corresponding parameter is set to be identically zero and is not the target of any inference.
\end{enumerate}

There is another version of the projection parameter defined as follows.
For the moment, suppose that $d < n$ and that there is no model selection.
Let 
$\beta_n = (\mathbb{X}^\top \mathbb{X})^{-1}\mathbb{X}^\top \mu_n$
where $\mathbb{X}$ is the $n\times d$ design matrix, whose columns are the $n$ vector of covariates $X_1,\ldots, X_n$, 
and $\mu_n = (\mu_n(1),\ldots, \mu_n(n))^\top$,
with $\mu_n(i) = \mathbb{E}[Y_i | X_1,\ldots, X_n]$.
This is just the conditional mean of the least squares estimator
given $X_1,\ldots, X_n$.
We call this the {\em conditional projection parameter}.
The meaning of this parameter when the linear model
is false is not clear.
It is a data dependent parameter, even in the absence of
model selection.
\cite{buja2015models} 
have devoted a whole paper to this issue.
Quoting from their paper:
\begin{quote}
{\it When fitted models are approximations, conditioning on the regressor is no longer
permitted ... Two effects occur:
(1) parameters become dependent on the regressor distribution;
(2) the sampling variability of the parameter estimates no longer derives
from the conditional  distribution  of the response alone.
Additional sampling variability arises when the nonlinearity 
conspires with the randomness of the regressors to generate a 
$1/\sqrt{n}$ contribution to the standard errors.}
\end{quote}

Moreover, it is not possible to estimate the distribution of the
conditional projection parameter estimate
in the distribution free framework.
To see that, note that the least squares
estimator can be written as
$\hat\beta(j) = \sum_{i=1}^n w_i Y_i$
for weights $w_i$ that depend on the design matrix.
Then
$\sqrt{n}(\hat\beta(j) - \beta(j)) = \sum_{i=1}^n w_i \epsilon_i$
where
$\epsilon_i = Y_i - \mu_n(i)$.
Thus, for each $j \in \{1,\ldots,d\}$ we have that
$\sqrt{n}(\hat\beta(j) - \beta(j))$ is approximately $\approx N(0,\tau^2)$,
where
$\tau^2 = \sum_i w_i^2 \sigma_i^2$,
with
$\sigma_i^2 = {\rm Var}(\epsilon_i | X_1,\ldots, X_n)$.
The problem is that there is no consistent estimator of
$\tau^2$ under the nonparametric models we are considering.
Even if we assume that $\sigma_i^2$ is constant
(an assumption we avoid in this paper), we still have that
$\tau^2 =\sigma^2 \sum_i w_i^2$
which cannot be consistently estimated without assuming that the linear model is correct.
Again, we refer the reader to \cite{buja2015models} for more discussion.
In contrast,
the projection parameter
$\beta = \Sigma^{-1}\alpha$ is a fixed functional 
of the data generating distribution $P$
and is estimable.
For these reasons, we focus in this paper on the projection parameter
rather than the conditional projection parameter.

\subsection*{Goals and Assumptions}

Our main goal is to provide statistical guarantees for each of the four
random parameters of variable significance introduced above, under our
distribution free framework. For notational convenience, in this section we let
$\theta_{\wS}$ be any of the parameters of interest: $\beta_{\wS}$,
$\gamma_{\wS}$,
$\phi_{\wS}$ or $\rho_{\wS}$.

We will rely on sample splitting: assuming for notational convenience that the
sample size is
$2n$, we randomly split the data $\mathcal{D}_{2n}$
into two halves, $\mathcal{D}_{1,n}$ and $\mathcal{D}_{2,n}$. Next, we run the model
selection and estimation procedure $w_{n}$ on $\mathcal{D}_{1,n}$, obtaining
both $\wS$ and $\hat{\mu}_{\wS}$ (as remarked above, if we are concerned with the
projection parameters, then we will only need $\wS$). 
We then use the
second half of the sample $\mathcal{D}_{2,n}$  to construct
an estimator $\hat \theta_{\wS}$ and
a confidence hyper-rectangle $\hat{C}_{\wS}$ for $\theta_{\wS}$ satisfying the
following properties:

\begin{align}
 {\rm Concentration}: \phantom{xxxxxxxxx}&\  \displaystyle \limsup_{n
 \rightarrow \infty}  \sup_{w_{n}\in {\cal W}_{n}} \sup_{P\in {\cal Q}_n}
 \mathbb{P}(||\hat\theta_{\wS}-\theta_{\wS}||_\infty > r_n) \to 0  \label{eq:concentration}\\
 \vspace{.11pt} \nonumber\\
 {\rm Coverage\   validity\  (honesty)}: &\  
 \displaystyle\liminf_{n\to\infty}\inf_{w_{n}\in {\cal W}_{n}}  \inf_{P\in
     {\cal Q}_{n}}
\mathbb{P}(\theta_{\wS}\in \hat{C}_{\wS})\geq 1-\alpha \label{eq::honest}\\
\vspace{.11pt} \nonumber \\
{\rm Accuracy}: \phantom{xxxxxxxx}&\  
\displaystyle  \limsup_{n
 \rightarrow \infty}  \sup_{w_{n}\in {\cal W}_{n}} \sup_{P\in {\cal
 Q}_n}\mathbb{P}(\nu(\hat{C}_{\wS})> \epsilon_n)\to 0 \label{eq::accuracy}
\end{align}
where $\alpha \in (0,1)$ is a pre-specified level of significance,  
$\mathcal{W}_n$ is the set of all the model selection and estimation procedures on samples of size $n$, 
$r_n$  and $\epsilon_n$ both vanish as $n \rightarrow \infty$ and
$\nu$ is the size of the set
(length of the sides of the rectangle)
where we recall that $k = |\wS|$ is
non-random.
The probability statements above take into account both the randomness in the sample
$\mathcal{D}_{n}$
and the randomness associated to splitting it into halves.

{\bf Remark.}
The property  that  the coverage of $\hat{C}_{\wS}$ is guaranteed uniformly over the entire class
$\mathcal{Q}_n$ is known as (asymptotic) honesty \citep{li1989honest}. 
Note that the confidence intervals are for random parameters (based on half the data)
but the uniform coverage, accuracy and concentration guarantees hold marginally.

 The statistical guarantees listed above assure that both
    $\hat{\theta}_{\wS}$ and $\hat{C}_{\wS}$ are {\it
robust} with respect to the choice of $w_n$.
We seek validity over all model selection and estimation rules
 because, in realistic data analysis, the procedure $w_n$
can be very complex.
In particular, the choice of model can involve: plotting, outlier removal,
transformations, choosing among various competing models, etc.. Thus, unless
we have validity over all $w_n$,
there will be room for unconscious biases to enter. Note that
sample splitting is key in yielding
uniform coverage and robustness.

The confidence sets we construct will be hyper-rectangles.
The reason for such choice is two-fold.
First, once we have a rectangular confidence set for a vector parameter,
we immediately have simultaneous confidence intervals for the
components of the vector.
Secondly, recent results on high dimensional normal approximation of normalized sums by \cite{cherno1,cherno2} have shown that central limit theorems 
for hyper-rectangles have only a logarithmic dependence on the dimension.

Depending on the target parameter, the class $\mathcal{Q}_n$ of data generating
distributions on $\mathbb{R}^{d+1}$ for the pair $(X,Y)$ will be different. We
will provide details on each such case separately. However, it is worth noting
that  inference for the projection parameters calls for a far more restricted
class of distributions than the other parameters. In particular, we find it
necessary to impose uniform bounds on the largest and smallest eigenvalues of
the covariance matrices of all $k$ marginals of the $d$ covariates, as well as
bounds on the higher moments of $X$ and on the mixed moments of $X$ and $Y$.
We will further assume, in most cases, that the distribution of the pair
$(X,Y)$ in $[-A,A]^{d+1}$, for some fixed $A>0$. Such compactness assumptions
are stronger than necessary but allow us to keep the statement of the results
and their proofs simpler. In particular, they may be replaced with appropriate
tail or moment bounds and not much will change in our analysis and results.

Although we have formulated the guarantees of honest validity, accuracy and
concentration in asymptotic terms, all of our results are in fact obtained as
finite sample bounds. This allow us to derive consistency rates in $n$ with all
the relevant quantities, such as the dimension $d$, the size of the selected
model $k$, and the variance and eigenvalue bounds needed for the projection
parameters accounted for in the constants (with the exception of $A$, which we
keep fixed). As a result, our results remain valid and are in fact most
interesting  when all these quantities are allowed to change with $n$.

\subsection{Related Work}\label{sec:related}

The problem of inference after model selection has
received much attention lately.  Much of the work falls broadly into three
categories: inference uniformly over selection procedure, inference with
regard to a particular debiased or desparsified model, and inference conditional on model
selection.  A summary of some of the various methods is in
Table \ref{table::compare}.
We discuss these approaches in more detail in Section \ref{section::comments}.

The uniform approach includes POSI \citep{berk2013valid},
which constructs valid inferential procedures regardless of the model selection
procedure by maximizing over all possible model selections.  This method
assumes Normality and a fixed, known variance, as well as being computationally
expensive.  The idea is built upon by later work 
\citep{bachoc2, bachoc}, which extend the ideas to other parameters of interest and
which allow for heteroskedasticity, non-normality, and model misspecification.

Most
other approaches focus on a particular model selection procedure and conduct
inference for selections made by that procedure. 
This includes the literature on debiased or desparsified regularized models,
for example
\cite{buhlmann2013statistical}, \cite{zhang2014confidence},
\cite{javanmard2014confidence}, \cite{peter.sarah.2015}, \cite{test},
\cite{zhang2017simultaneous}, \cite{vandegeer2014asymptotically}, \cite{nickl2013confidence}.  This work
constructs confidence intervals for parameters in high dimensional regression.
These can be used for the selected model
if a Bonferroni correction is applied.
However, these methods tend to assume that the linear model
is correct as well as a number of other assumptions on the design matrix and
the distribution of errors.

A separate literature on selective inference has focused on inference with
respect to the selected model, conditional on the event of that model's selection.  This
began with \cite{lockhart2014significance}, but was developed more fully in
\cite{lee2016exact}, \cite{carving}, and \cite{taylor2014exact}.
Further works in this area
include \cite{tibshirani2015uniform}, \cite{randomization},
\cite{loftus2015selective}, \cite{bootstrap.john}, \cite{tibshirani2016exact}, \cite{loco.john}.
In the simplest version, the distribution of
$\sqrt{n}(\hat\beta(j) - \beta(j))$ conditional
on the selected model has a truncated Gaussian distribution,
if the errors are Normal and the covariates are fixed.
The cdf of the truncated Gaussian is used as a pivot to
get tests and confidence intervals.
This approach requires Normality, and a fixed, known variance.
While the approach has broadened in later work, the methods still tend to assume
fixed design and a known, parametric structure to the outcome.

There have been several additional approaches to this problem that don't fall
in any of these broad categories.  While this is a larger literature than can
be addressed completely here, it includes early work on model selection
\cite{hurvich1990impact} and model averaging interpretations
\cite{hjort2003frequentist}; the impossibility results of \cite{leeb2008can},
\cite{buja2015models} on random $X$ and model misspecification; methods based
on resampling or sample splitting
\citep{CL:11,CL:13,Efron:14,wasserman2009high,meinshausen2009pvalues};
stability selection \citep{meinshausen2010stability, shah2013variable};  the
conformal inference approach of \cite{lei2016distribution}; goodness-of-fit
tests of \cite{shah2018goodness};
moment-constraint-based uniform confidence sets \citep{andrews2009hybrid};
\cite{meinshausen2015group} on inference
about groups of variables under general designs; \cite{belloni2011inference} in
the instrumental variable setting; \cite{belloni2015uniform} on post-selection
inference for $Z$-estimators, and the knockoffs approach of
\cite{barber2015controlling} and later \cite{candes2016panning}.  Although they
are not directed at linear models,\cite{wager2014confidence} and
\cite{JMLR:v17:14-168} address similar problems for random forests.

\begin{table}
\begin{center}
\begin{tabular}{llllll}
Method           & Parameter  & Assumptions & Accuracy          & Computation & Robust\\ \hline
Debiasing        & True $\beta$    & Very Strong & $1/\sqrt{n}$       & Easy        & No\\
Conditional      & Projection & Strong      & Not known                 & Easy        & No\\
Uniform          & Projection & Strong        & $\sqrt{k/n}$      & NP hard     & Yes\\
Sample Splitting & Projection & Weak        & $\sqrt{k^{5/2}\log k\sqrt{\log n}/n}$    & Easy        & Yes\\
Sample Splitting & LOCO       & None        & $\sqrt{\log (kn)/n}$ & Easy        & Yes\\
\end{tabular}
\end{center}
\caption{\em Different inferential methods. 
`accuracy' refers to the size of sides of the confidence set.
`robust' refers to robustness to model assumptions.
The term `Very Strong' means that the linear model is assumed to be correct and that there are incoherence
assumptions on the design matrix. `Strong' means constant variance and Normality are assumed. `Weak' means
only iid and invertible covariance matrix (for the selected variables).
`None' means only iid or iid plus a moment assumption.}
\label{table::compare}
\end{table}

{\bf Sample Splitting.}
The oldest method for inference after model selection
is sample splitting: half the data ${\cal D}_1$ are used for model fitting and the
other half ${\cal D}_2$ are used for inference.\footnote{
For simplicity, we assume that the data are split into two parts of equal size.
The problem of determining the optimal size of the split is not considered in this paper.
Some results on this issue are contained in \cite{shao1993linear}.}

Thus $S = w_{n}({\cal D}_1)$.
The earliest references for sample splitting that we know of are
\cite{Barnard},
\cite{cox1975note},
\cite{faraway1995data}
\cite{hartigan1969using},
page 13 of \cite{miller2002subset}
\cite{moran1973dividing},
page 37 of \cite{mosteller1977data} and
\cite{picard1990data}.
To quote Barnard:
`` ... the simple idea of splitting a sample in two and then
developing the hypothesis on the basis of one part and testing
it on the remainder may perhaps be said to be one of the most
seriously neglected ideas in statistics ...''

To the best of our knowledge
there are only two methods that achieve
asymptotically honest coverage:
sample splitting and uniform inference.
Uniform inference is based on estimating the distribution
of the parameter estimates over all possible model selections.
In general, this is infeasible.
But we compare sample splitting and uniform inference
in a restricted model in Section \ref{section::splitornot}.

\subsection{Outline}

In Section
\ref{section::splitting}
we introduce the basic sample splitting strategies.
In Section \ref{section::splitornot}
we compare sample splitting to non-splitting strategies.
Section \ref{section::comments}
contains some comments 
on other methods.
In Section \ref{section::simulation}
we report some numerical examples.
In Section
\ref{section::berry}
we establish a Berry-Esseen bound for regression
with possibly increasing dimension and no assumption of linearity
on the regression function.
Section \ref{section::conclusion}
contains concluding remarks.
Extra results, proofs 
and a discussion of another version of the bootstrap, are relegated to the Appendices.

\subsection{Notation}

Let $Z=(X,Y)\sim P$ 
where $Y\in\mathbb{R}$ and
$X\in \mathbb{R}^d$.
We write
$X = (X(1),\ldots, X(d))$
to denote the components of the vector $X$.
Define
$\Sigma = \mathbb{E}[X X^\top]$ and
$\alpha = (\alpha(1),\ldots,\alpha(d))$
where $\alpha(j) = \mathbb{E}[Y X(j)]$.
Let $\sigma = {\rm vec}(\Sigma)$ and
$\psi \equiv \psi(P) = (\sigma,\alpha)$.
The regression function is
$\mu(x) = \mathbb{E}[Y|X=x]$.
We use $\nu$ to denote Lebesgue measure.
We write
$a_n \preceq b_n$ to mean that
there exists a constant $C>0$ such that
$a_n \leq C b_n$ for all large $n$.
For a non-empty subset $S\subset \{1,\ldots, d\}$ of the covariates
 $X_S$ or $X(S)$ denotes the corresponding elements of $X$: $(X(j):\ j\in S)$ 
Similarly,
$\Sigma_S = \mathbb{E}[X_S X_S^\top]$ and $\alpha_S = \mathbb{E}[Y X_S]$.
We write
$\Omega = \Sigma^{-1}$ and
$\omega = {\rm vec}(\Omega)$
where ${\rm vec}$ is the operator that stacks a matrix into one large vector.
Also, ${\rm vech}$ is the half-vectorization operator that takes a symmetric matrix
and stacks the elements on and below the diagonal into a matrix.
$A\otimes B$ denotes the Kronecker product of matrices.
The commutation matrix $K_{m,n}$ is the
$mn \times mn$ matrix defined by
$K_{m,n} {\rm vec}(A) = {\rm vec}(A^\top)$.
For any $k\times k$ matrix $A$.
$\mathrm{vech}(A)$ denotes the
column vector of dimension $k(k+1)/2$ obtained by vectorizing only
the lower triangular part of $k\times k$ matrix $A$.

\section{Main Results}
\label{section::splitting}

We now describe how to construct estimators of the random parameters defined earlier.
Recall that we rely on data splitting: we randomly split the $2n$ data into two halves
${\cal D}_{1,n}$ and ${\cal D}_{2,n}$.  Then, for a given choice of  the model
selection and estimation rule $w_n$, we use ${\cal D}_{1,n}$ to select a non-empty set of variables $\wS \subset \{
1,\ldots,d\}$ where
$k =|\wS| < n$. For the LOCO and prediction parameters, based on $\mathcal{D}_{1,n}$, we also compute $\widehat{\beta}_{\wS}$,
any estimator of the projection parameters restricted to $\wS$. In addition, 
for each $j \in \wS$,  we further compute, still using $\mathcal{D}_{1,n}$ and the rule $w_n$, $\widehat{\beta}_{\wS(j)}$, the estimator of the
projection parameters over the set $\widehat{S}(j)$.   
Also, for $l=1,2$, we denote with $\mathcal{I}_{l,n}$ random subset of $\{1,\ldots, 2n\}$ containing the indexes for the data points in $\mathcal{D}_{l,n}$.

\subsection{Projection Parameters}
\label{sec:projection}
In his section we will derive various statistical guarantees for the projection parameters, defined in 
\eqref{eq::projection-parameter}. We will first define the class of data generating distributions on $\mathbb{R}^{d+1}$  for which our results hold. In the definition below, $S$ denotes a non-empty subset of $\{1,\ldots,d\}$  and  $W_S = ({\rm vech}(X_S X_S^\top), X_SY)$. 
\begin{definition}\label{def:Pdagger}
	Let ${\cal P}_n^{\mathrm{OLS}} $ be the set of all probability distributions $P$ on $\mathbb{R}^{d+1}$ with zero mean, Lebesgue density and such that, for some positive quantities $A, a, u, U , v$ and  $\overline{v}$,
	\begin{enumerate}
		\item the support of $P$ is contained in $[-A,A]^{d+1}$;
		\item $\min_{ \{ S \colon |S| \leq k \} } \lambda_{\rm min}(\Sigma_S) \geq u$ and $\max_{ \{ S \colon |S| \leq k\} } \lambda_{\rm max}(\Sigma_S) \leq U$, where  $\Sigma_S = \mathbb{E}_P[X_S X_S^\top]$;
				\item $\min_{ \{S \colon |S| \leq k \} } \lambda_{\rm min}({\rm Var}_P(W_S))\geq v$ and $\max_{ \{S \colon |S| \leq k\} }
\lambda_{\rm max}({\rm Var}_P(W_S))\leq \overline{v}$.
\item $\min\{ U, \overline{v} \} \geq \eta$, for a fixed $\eta>0$.
	\end{enumerate}
\end{definition}

The first compactness assumption can be easily modified by assuming
instead that $Y$ and $X$ are sub-Gaussian, without any technical
difficulty. We make such boundedness assumption to simplify
our results.  The bound on the smallest eigenvalue of $\Sigma_S$,
uniformly over all subsets $S$ is natural: the projection parameter is
only well defined provided that $\Sigma_S$ is invertible for all $S$,
and the closer $\Sigma_S$ is to being singular the higher the
uncertainty. The uniform condition on the largest eigenvalue of
$\Sigma_S$ in part 2. is used to obtain sharper bounds than the ones
stemming from the crude bound $U \leq A k$ implied by the assumption
of a compact support (see e.g. \Cref{thm:beta.accuracy2} below).  The
quantities $v$ and $\overline{v}$ in part 3. are akin to 4th moment
conditions. In particular, one can always take $\overline{v} \leq A^2
k^2$ in the very worst case.  Finally, the assumption of zero mean is
imposed out of convenience and to simplify our derivations, so that we
need not to be concerned with an intercept term.  As remarked above,
in all of our results we have kept track of the dependence on the
constants $a, u, U , v$ and $\overline{v}$, so that we may in fact
allow all these quantities to change with $n$ (but we do treat $A$ as
fixed and therefore have incorporate it in the constants).  Finally,
the assumption that t$U$ and $\overline{v}$ are bounded from zero is
extremely mild. In particular, the parameter $\eta$ is kept fixed and
its value affect the constants in Theorems \ref{thm:beta.accuracy2},
\ref{thm::big-theorem} and \ref{theorem::beta.boot} (the matrix
Bernstein inequality (see \Cref{lem:operator})).

{\bf Remark.} Although our assumptions imply that the individual coordinates of $X$ are sub-Gaussians, we do not require $X$ itself to be a sub-Gaussian vector, in the usual sense that, for each $d$-dimensional unit vector  $\theta$, the random variable $\theta^\top X$ is sub-Gaussian with variance parameter independent of $\theta$ and $d$.

Recall that the projection
parameters defined in \eqref{eq::projection-parameter} are
\begin{equation}\label{eq:betahat}
    \beta_{\wS} = \Sigma_{\wS}^{-1}\alpha_{\wS},
\end{equation}
where $\wS$ is the model selected based on $\mathcal{D}_{1,n}$ (of size no larger than $k$) and
\begin{equation}\label{eq:sigma.alpha}
    \alpha_{\wS} = \mathbb{E}[Y X(\wS)] \quad \text{and} \quad 
\Sigma_{\wS} = \mathbb{E}[X(\wS) X(\wS)^\top].
\end{equation}
We will be studying the ordinary least squares
estimator $\hat{\beta}_{\wS}$ of $\beta_{\wS}$  
computed using the sub-sample $\mathcal{D}_{2,n}$ and restricted to the coordinates
$\wS$. That is, 
\begin{equation}\label{eq:least.squares}
    \hat{\beta}_{\wS} = \widehat{\Sigma}_{\wS}^{-1} \widehat{\alpha}_{\wS} 
\end{equation}
where, for any non-empty subset $S$ of $\{1,\ldots,d\}$, 
\begin{equation}\label{eq:alpha.beta.hat}
\widehat{\alpha}_{S} = \frac{1}{n} \sum_{i \in \mathcal{I}_{2,n} } Y_i X_i(S)
\quad  \text{and} \quad 
\widehat{\Sigma}_{S} = \frac{1}{n} \sum_{i \in \mathcal{I}_{2,n}} X_i(\wS) X_i(S)^\top.
\end{equation}
Since each $P \in \mathcal{P}_n^{\mathrm{OLS}}$ has a Lebesgue density,
$\hat{\Sigma}_{\wS}$ is invertible almost surely as long as $n \geq k \geq |\wS|$.
Notice that $\hat{\beta}_{\wS}$ is not an unbiased estimator of
$\beta_{\wS}$ , conditionally or unconditionally on $\mathcal{D}_{2,n}$.

In order to relate $\hat{\beta}_{\wS}$ to $\beta_{\wS}$, it will first be convenient
to condition on $\wS$ and thus regard $\beta_{\wS}$ as a $k$-dimensional
deterministic vector
of parameters (recall that, for simplicity,  we assume that $|\wS| \leq k$), which
depends on some unknown $P \in \mathcal{P}_n^{\mathrm{OLS}}$. Then,
$\hat{\beta}_{\wS}$ is an estimator of a fixed parameter $\beta_{\wS} =
\beta_{\wS}(P)$ computed using
an i.i.d. sample $\mathcal{D}_{2,n}$ from  the same distribution $P \in
\mathcal{P}_n^{\mathrm{OLS}}$. Since all our bounds depend on $\wS$ only through its
size $k$, those bounds will hold also unconditionally.

For each $P \in \mathcal{P}_n^{\mathrm{OLS}}$, we can represent the
parameters $\Sigma_{\wS} = \Sigma_{\wS}(P)$ and $\alpha_{\wS} = \alpha_{\wS}(P)$
in \eqref{eq:sigma.alpha} in vectorized form as
\begin{equation}\label{eq:psi.beta}
\psi = \psi_{\wS} = \psi(\wS,P)=  
\left[ \begin{array}{c} 
	\mathrm{vech}(\Sigma_{\wS})\\
	\alpha_{\wS}\\ 
       \end{array} \right]  \in \mathbb{R}^{b},
       \end{equation}
where $b = \frac{ k^2 + 3k}{2} $. 
Similarly, based on the sub-sample $\mathcal{D}_{2,n}$ we define the $n$ random vectors 
$$
W_i = 
\left[ \begin{array}{c}
 \mathrm{vech}(X_i(\wS) X_i(\wS)^\top)\\
  Y_i \cdot X_i(\wS) \\ 
\end{array} 
\right] \in \mathbb{R}^b, \quad i \in \mathcal{I}_{2,n},
$$ 
and their average
\begin{equation}\label{eq:hat.psi.beta}
    \hat{\psi} = \hat{\psi}_{\wS} = \frac{1}{n} \sum_{i \in \mathcal{I}_{2,n}} W_i.
\end{equation}
It is immediate to see that $\mathbb{E}_P[\hat{\psi}] = \psi$, uniformly over all
$P \in \mathcal{P}_n^{\mathrm{OLS}}$.

We express both the projection parameter $\beta_{\wS}$ and the least square
estimator 
$\hat{\beta}_{\wS}$ as non-linear functions of $\psi$ and
$\hat{\psi}$, respectively, in the following way.  
Let $g \colon \mathbb{R}^b
\rightarrow  \mathbb{R}^k$ be given by
\begin{equation}\label{eq:g.beta}
    x = \left[ \begin{array}{c}  
    	    x_1\\
    	    x_2\\
	\end{array}
    \right] \mapsto \left( \mathrm{math}(x_1) \right)^{-1} x_2,
\end{equation}
where $x_1$ and $x_2$ correspond to the first $k(k+1)/2$ and the last $k$
coordinates of  $x$, respectively, and $\mathrm{math}$ is the inverse mapping of
$\mathrm{vech}$, i.e.  $\mathrm{math}(x) = A$ if and only if $\mathrm{vech}(A) =
x$. Notice that  $g$ is well-defined over the convex set 
\[
\left\{     \left[ \begin{array}{c}  
    \mathrm{vech}(\Sigma)\\
    x
\end{array}
\right] \colon \Sigma \in \mathcal{C}^+_{k}, x \in \mathbb{R}^k \right\}
\]
where $\mathcal{C}^+_k$ is the cone of positive definite matrices of dimension
$k$.
It follows from our assumptions that, for each $P \in \mathcal{P}_n^{\mathrm{OLS}}$,
$\psi$ is in the domain of $g$ and, as long as $n \geq d$, so is $\hat{\psi}$,
almost surely.
Thus, we may write
$$
\beta_{\wS} = g(\psi_{\wS}) \quad \text{and} \quad \hat{\beta}_{\wS} =
g(\hat{\psi}_{\wS}).
$$

This formulation of $\beta_{\wS}$ and $\hat{\beta}_{\wS}$ is convenient because, 
by expanding each coordinate of $g(\hat{\psi})$ separately through a first-order
Taylor series expansion  around $\psi$,
it allows us to re-write $\hat{\beta}_{\wS} -
\beta_{\wS}$ as a linear  transformation of
$\hat{\psi} - \psi$ given by the Jacobian of $g$ at $\psi$, plus a stochastic
reminder term. Since $\hat{\psi} - \psi$ is an average, such 
approximation is simpler to analyze that the original quantity   $\hat{\beta}_{\wS} -
\beta_{\wS}$ and, provided that the reminder term of the Taylor
expansion be small,  also sufficiently accurate. 
This program is carried out in detail and greater generality in a later 
Section \ref{section::berry}, where we derive 
finite sample Berry-Esseen bounds for non-linear statistics of sums of
independent random vectors. The results in
this section are  direct, albeit non-trivial, 
applications of those bounds.

\subsubsection*{Concentration of $\hat{\beta}_{\wS}$}

We begin by deriving high probability concentration bonds for
$\hat{\beta}_{\wS}$ around $\beta_{\wS}$. 
When there is no model selection nor sample splitting -- so that $\wS$ is deterministic and equal to
$\{1,\ldots,d$) -- our results yield consistency rates for the ordinary least squares
    estimator of the projection parameters, under increasing dimensions and a
    misspecified model. An analogous result was established in \cite{hsu14}, where the
approximation error $\mu(x) - x^\top \beta$ is accounted for explicitly.

\begin{theorem}\label{thm:beta.accuracy2}
Let
\[
B_n  = 
    \frac{ k}{u^2}
    \sqrt{ U \frac{ \log k +
    \log n}{n}} 
\]
and assume that $\max\{ B_n, u B_n \} \rightarrow 0$ as $n \rightarrow \infty$. Then, there exists
a constant $C>0$, dependent on $A$ and $\eta$ only, such that, for all $n$ large enough, 
\begin{equation}\label{eq::beta2}
\sup_{w_n \in \mathcal{W}_n} \sup_{P \in \mathcal{P}_n^{\mathrm{OLS}}} \|\hat\beta_{\wS} - \beta_{\wS} \| \leq C
B_n,
\end{equation}
with probability at last $1 - \frac{2}{n}$.
\end{theorem}

\noindent {\bf Remarks.} 
\begin{enumerate}
\item  It is worth recalling that, in the result above as well as in all the result of the paper, the probability is with respect to joint distribution of
the entire sample and of the splitting process.
	\item For simplicity, we have phrased the bound in \Cref{thm:beta.accuracy2} in an asymptotic manner. The result can be trivially turned into a finite sample statement by appropriately adjusting the value of the constant $C$ depending on how rapidly $\max\{ B_n, u B_n \} \rightarrow 0$ vanishes. 
	\ 	\item The proof of the above theorem relies namely an inequality for matrix norms and the vector and matrix
Bernstein
concentration inequalities (see \Cref{lem:operator} below).
\item 
Theorems \ref{thm:beta.accuracy} and \ref{thm:beta.accuracy2} 
can be easily generalized to cover the case in which the model selection and the computation of the projection parameters are performed on the entire dataset and not on separate, independent splits. In this situation, it is necessary to obtain a high probability bound for the quantity
\[
\max_{S} \| \beta_S - \hat{\beta}_S \|
\]
where the maximum is over all non-empty subsets of $\{1,\ldots,d\}$ of
size at most $k$ and $\hat{\beta}_S =
\hat{\Sigma}_{S}^{-1}\hat{\alpha}_{S}$ (see Equation
\ref{eq:alpha.beta.hat}). Since there are less than $ \left( \frac{e
  d}{k} \right)^k $ such subsets, an additional union bound argument
in each application of the matrix and vector Bernstein's inequalities
(see Lemma \ref{lem:operator}) within the proofs of both Theorem
\ref{thm:beta.accuracy} and \ref{thm:beta.accuracy2} will give the
desired result. The rates so obtained will be then worse than the ones
from Theorems \ref{thm:beta.accuracy} and \ref{thm:beta.accuracy2}
which, because of the sample splitting do not require a union bound.
In particular, the scaling of $k$ with respect to $n$ will be worse by
a factor of $k \log \frac{d}{k}$. This immediately gives a rate of
consistency for the projection parameter under arbitrary model
selection rules without relying on sample splitting.  We omit 
the details.
\end{enumerate}

\subsubsection*{Confidence sets for the projection parameters: Normal
Approximations}
We will now derive confidence intervals for the projection
parameters using on a high-dimensional Normal approximation  
to $\hat{\beta}_{\wS}$. 
The construction of such confidence sets entails approximating the dominant linear term in the Taylor
series expansion of $\hat{\beta}_{\wS} - \beta_{\wS}$ by a centered
Gaussian vector in $\mathbb{R}^{\wS}$ with the same covariance matrix
$\Gamma_{\wS}$ (see (\ref{eq:Gamma}) in  \Cref{section::berry}). The coverage
properties of the resulting confidence
sets depend crucially on the ability to estimate such covariance. For that
purpose, we use a plug-in estimator, 
given by
\begin{equation}
\label{eq::Ga}
    \hat\Gamma_{\wS} = \hat{G}_{\wS}\hat V_{\wS} \hat{G}_{\wS}^\top
\end{equation}
where  $\hat V_{\wS} = \frac{1}{n}\sum_{i=1}^n [ (W_i - \hat\psi) (W_i -
\hat\psi)^\top]$
is the $b \times b$ empirical covariance matrix of the $W_i$'s and the $k \times
b$ matrix 
 $\hat{G}_{\wS}$ is the Jacobian of the mapping $g$, given explicitly below
in \Cref{eq::GG}, evaluated at
$\hat{\psi}$.

The first confidence set for the projection parameter based on the Normal
approximation that we propose is an
$L_\infty$ ball of appropriate radius centered at $\hat{\beta}_{\wS}$:
\begin{equation}\label{eq::beta.conf-rectangle}
    \hat{C}_{\wS}  = \Bigl\{ \beta \in \mathbb{R}^k:\
    ||\beta-\hat\beta_{\wS}||_\infty \leq
    \frac{\hat{t}_\alpha}{\sqrt{n}}\Bigr\},
\end{equation}
where $\hat{t}_\alpha$ is a random radius (dependent on $\mathcal{D}_{2,n}$
)  such that 
\begin{equation}\label{eq:t.akpha}
    \mathbb{P}\left( \| \hat{\Gamma}^{1/2}_{\hat{S}} Q \|_\infty \leq \hat{t}_\alpha
    \right) = \alpha,
\end{equation}
with $Q$  a random vector having the $k$-dimensional standard Gaussian
distribution  and independent of the data. 

In addition to the $L_\infty$ ball given in  \eqref{eq::beta.conf-rectangle},
we also construct a confidence set for
$\beta_{\wS}$ to be a hyper-rectangle, with sides of different lengths in
order to account for different variances in the covariates. This can be
done using the set
\begin{equation}\label{eq:beta.hyper:CI}
    \tilde C_{\wS} = \bigotimes_{j\in \wS} \tilde{C}(j),
\end{equation}
where
\[
    \tilde{C}(j) = \left[ \hat\beta_{\wS}(j) - z_{\alpha/(2k)}
    \sqrt{\frac{ \hat\Gamma_{\wS}(j,j)}{n}},  \hat\beta_{\wS}(j) + z_{\alpha/(2k)}
    \sqrt{\frac{ \hat\Gamma_{\wS}(j,j)}{n}}\right],
\]
with $\hat\Gamma_{\wS}$ given by (\ref{eq::Ga}) and $z_{\alpha/(2k)}$ the upper $1 -
\alpha/(2k)$ quantile of a standard Normal variate. Notice that we use a Bonferroni
correction to guarantee a nominal coverage of $1-\alpha$.

\begin{theorem}\label{thm::big-theorem}
Let $\hat{C}_{\wS}$ and $\tilde{C}_{\wS}$ the confidence sets defined 
in \eqref{eq::beta.conf-rectangle} and \eqref{eq:beta.hyper:CI}, respectively.
Let 
\begin{equation}\label{eq:un}
u_n = u -K_{2,n},
\end{equation}
where
\[     
    K_{2,n} = C A \sqrt{ k U \frac{\log k + \log n}{n} },
\]
with $C = C(\eta)>0$ the universal constant in \eqref{eq:matrix.bernstein.simple.2}.
Assume, in addition, that $n$ is large enough so that
$ u_n $ is positive.
Then, for a $C >0$ dependent on $A$ only,
\begin{equation}
    \label{eq:big-theorem.Linfty}
\inf_{w_n \in \mathcal{W}_n}    \inf_{P\in {\cal P}_n^{\mathrm{OLS}}}\mathbb{P}(\beta \in \hat{C}_{\wS}) \geq 1-\alpha -
C \Big(\Delta_{n,1} + \Delta_{n,2}+\Delta_{n,3}  \Big)
\end{equation}
and
\begin{equation}
    \label{eq:big-theorem.hyper}
\inf_{w_n \in \mathcal{W}_n}    \inf_{P\in {\cal P}_n^{\mathrm{OLS}}}\mathbb{P} (\beta \in \tilde{C}_{\wS}) \geq 1-\alpha -
C\Big(\Delta_{n,1} + \Delta_{n,2}+\tilde{\Delta}_{n,3}  \Big),
\end{equation}
where
\[
\Delta_{n,1} =  \frac{1}{\sqrt{v}}\left(  \frac{
    \overline{v}^2 k^2 (\log kn)^7)}{n}\right)^{1/6}  ,  \quad  
    \Delta_{n,2} = \frac{ U }{ \sqrt{v}}  \sqrt{ \frac{k^4 \overline{v} \log^2n \log k}{n\,u_n^6}
},  
\]
\[
\Delta_{n,3} = \left( \frac{ U^2 }{
    v }\right)^{1/3} \left( \overline{v}^2 \frac{k^{5}}{u_n^{6} u^4}
    \frac{ \log n}{n} \log ^4 k\right)^{1/6}  \quad \text{and} \quad 
\tilde{\Delta}_{n,3} = \min \left\{ \Delta_{n,3}, \frac{U^2}{v} \overline{v} \frac{ k^{5/2}}{u_n^3 u^2} \frac{ 
 \log n}{n} \log k \right\}.
\]
\end{theorem}

A few remarks are in order.
\begin{enum}
\item The coverage probability is affected by three factors:  
the  term  $\Delta_{n,1}$, which bounds the approximation error stemming from the high dimensional Berry-Esseen theorem (see \Cref{thm:high.dim.clt}); 
the term $\Delta_{n,2}$, which  is a high probability bound on the size of the reminder term in the Taylor series expansion of $\beta_{\wS}$ around $\widehat{\beta}_{\wS}$ and can therefore be thought of as the price  for the non-linearity of the projection parameter, and
the terms  $\Delta_{n,3}$ and $\tilde{\Delta}_{n,3}$, which are due to the fact that the covariance of the estimator is unknown and needs to be also estimated, leading to another source of error (the bootstrap procedure, described below, implicitly estimates this covariance).
\item In terms of dependence of $k$ on $n$, all other things being equal, the covariance term $\Delta_{3,n}$ exhibit the worst rate, as it constrain $k$ to be of smaller order than $n^{1/5}$ in order to guarantee asymptotic coverage of $\hat{C}_{\wS}$. This same term also contains the worst dependence on $u$, the uniform bound on the smallest eigenvalue of all covariance matrices of the form $\Sigma_S$, for $S \subset \{1,\ldots,d\}$ with $0 < S \leq k$.  Thus, the dependence of the rates on the dimension and on the minimal eigenvalue is overall quite poor. While this is, to an extent, unavoidable, we do not know whether our upper bounds are  sharp. 
\item The reasons for replacing $u$ by the smaller term $u_n$ given in \eqref{eq:un}  are somewhat technical, but are explained in the proof of the theorem. Assuming a scaling in $n$ that guarantees that the error terms $\Delta_{1,n}$, $\Delta_{2,n}$ and $\Delta_{3,n}$ are vanishing, such modification is inconsequential and does not affect the rates.
\item The coverage rates  obtained for the LOCO and prediction parameters below in \Cref{sec:loco.parameters} are significantly faster then the ones for the projection parameters, and hold under less restrictions on the class of data generating distributions. We regard this as another reason to prefer the LOCO parameters.
\item 
Interesting,  the covariance error term $\tilde{\Delta}_{3,n}$  for confidence
set $\tilde{C}_{\wS}$ is no worse than the corresponding term for the set
$\hat{C}_{\wS}$, suggesting that using hyper-rectangles in stead of hyper-cubes  may be a better choice. 
\item The quantity
$\overline{v}$ can of be order $k^2$ in the worst case, further inflating the terms $\Delta_{3,n}$ and $\tilde{\Delta}_{3,n}$.

\item As a function of sample size,
there is a term of order $n^{-1/6}$ in $\Delta_{1,n}$ and $\Delta_{3,n}$.
The exponent $1/6$ comes from the Berry-Esseen bound in Section 3.
\cite{cherno2}
conjecture that this rate is optimal for high-dimensional central limit theorems.
Their conjecture is based on the lower bound result in 
\cite{bentkus1985lower}.
If their conjecture is true, then this is best rate that can be hoped for
in general.
\item 
The rates are slower than the rate obtained in
the central limit theorem given in \cite{portnoy1987central} for robust
regression estimators. 
A reason for such discrepancy is that \cite{portnoy1987central} assumes, among
the other things,  that
the linear model is correct.
In this case, the least squares estimators is conditionally unbiased.
Without the assumption of model correctness there is a substantial bias.
\item If we assume that the covariates are independent then the situation gets dramatically better.
For example, the term $\Delta_{n,2}$ is then $O(1/\sqrt{n})$.
But the goal of this paper is to avoid adding such assumptions.
\end{enum}

We now consider the accuracy of the confidence set given by the hyper-rectangle
$\tilde{C}_{\wS}$  from Equation  \eqref{eq:beta.hyper:CI} by deriving an upper
bound on the length of the largest side of $\max_{j \in \wS}  \tilde{C}(j)$.
Similar rates can be obtained for length of the sides of the hyper-cube
confidence set $\hat{C}_{\wS}$ given in \eqref{eq::beta.conf-rectangle}.

\begin{corollary}\label{cor:accuracy.beta}
	With probability at least $ 1- \frac{2}{n}$, the maximal length of the sides
  of the hyper-rectangle $\tilde{C}_{\wS}$ is bounded by
\[
C \sqrt{ \frac{\log k}{n}   \left( \frac{k^{5/2}}{u_n^3 u^2} \overline{v} \sqrt{ \frac{\log n}{n}} + \frac{k }{u^4} \overline{v}\right) }, 
\]
for a constant $C>0$ depending on $A$ only, uniformly over all $P \in  \mathcal{P}_n^{\mathrm{OLS}}$.
 
\end{corollary}

\subsubsection*{Confidence sets for the projection parameters: The Bootstrap}

The  confidence set in \eqref{eq::beta.conf-rectangle} based on the Normal approximation require the
evaluation of both the matrix
$\hat{\Gamma}_{\wS}$ and the quantile $\hat{t}_\alpha$ in \eqref{eq:t.akpha},
which may be computationally inconvenient. Similarly the  hyper-rectangle 
\eqref{eq:beta.hyper:CI} requires computing the diagonal entries in
$\hat{\Gamma}_{\wS}$. Below we show that the paired bootstrap can be deployed
to construct analogous confidence sets,  centered at
$\hat{\beta}_{\wS}$, without knowledge of
$\hat{\Gamma}_{\wS}$.

Throughout, by the bootstrap distribution we mean the empirical
probability measure associated to the sub-sample $\mathcal{D}_{2,n}$ and
conditionally on $\mathcal{D}_{1,n}$ and the outcome of the sample splitting
procedure.

We let $\hat{\beta}^*_{\wS}$ denote the estimator of the projection parameters
$\beta_{\wS}$ of the form \eqref{eq:betahat} and arising from an i.i.d. sample of size $n$ drawn from the bootstrap
distribution. It is important to point out that $\hat{\beta}^*_{\wS}$ is
well-defined only provided that the bootstrap realization of the covariates
$(X_1^*,\ldots,X_n^*)$ is such that the corresponding $k$-dimensional empirical covariance
matrix
\[
    \frac{1}{n} \sum_{i \in \mathcal{I}_{2,n}} X_i^*(\wS) (X_i^*(\wS))^\top 
\]
is invertible. Since the data distribution is assumed to have a $d$-dimensional
Lebesgue density, this occurs almost surely with respect to the distribution of 
the full sample $\mathcal{D}_n$ if the bootstrap sample contains more than $k$
distinct values. Thus, the bootstrap guarantees given below only holds on such
event. Luckily, this is a matter of little consequence, since under our
assumptions the probability
that such event does not occur is exponentially small in $n$ (see 
\Cref{eq:lem.occupancy} below).

For a given $\alpha \in (0,1)$, let $\hat{t}^*_\alpha$ be the smallest
    positive number such that 
    \[
	\mathbb{P}\left( \sqrt{n} \| \hat{\beta}^*_{\wS} - \hat{\beta}_{\wS}\|
	\leq \hat{t}^*_\alpha \Big| \mathcal{D}_{2,n} \right) \geq 1 - \alpha.
    \]
    Next, let $(\tilde{t}^*_j, j \in \wS)$ be such that
    \[
	\mathbb{P}\left( \sqrt{n} | \hat{\beta}^*_{\wS}(j) - \hat{\beta}_{\wS}
	(j) \leq \tilde{t}^*_j, \forall j \Big| \mathcal{D}_{2,n} \right) \geq 1 - \alpha.
    \]
By the union bound, each $\tilde{t}^*_j$ can be chosen to
    be the largest positive number such that
    \[
	\mathbb{P}\left( \sqrt{n} | \hat{\beta}^*_{\wS}(j) - \hat{\beta}_{\wS}
	(j) > \tilde{t}^*_j, \Big|  \mathcal{D}_{2,n} \right)  \leq \frac{\alpha}{k}.
    \]
Consider the following two bootstrap confidence sets: 
\begin{equation}\label{eq:ci.boot.beta}
	    \hat{C}^*_{\wS} = \left\{ \beta \in \mathbb{R}^{\wS} \colon \|
	    \beta - \hat{\beta}_{\wS}
	    \|_\infty \leq \frac{ \hat{t}^*_{\alpha}}{\sqrt{n}} \right\} \quad
	    \text{and} \quad 
	    \tilde{C}^*_{\wS} = \left\{ \beta \in \mathbb{R}^{\wS} \colon |
	    \beta(j) - \hat{\beta}_{\wS}(j)
	    | \leq \frac{ \tilde{t}^*_{j}}{\sqrt{n}}, \forall j \in \wS \right\}
	\end{equation}

It is immediate to see that $\hat{C}^*_{\wS}$ and $\tilde{C}^*_{\wS}$ are just the
	bootstrap equivalent of the confidence sets of \eqref{eq::beta.conf-rectangle} and \eqref{eq:beta.hyper:CI}, respectively.

\begin{theorem}
\label{theorem::beta.boot}
Let 
\[
v_n  = v - K_{1,n}, \quad \overline{v}_n  = \overline{v} + K_{1,n}, \quad u_n = u -  K_{2,n} \quad \text{and} \quad U_n = U + K_{2,n},
\]
where
\[
    K_{1,n} = C A^2 \sqrt{ b \overline{v}\frac{\log b + \log n}{n} } \quad \text{and} \quad     K_{2,n} = C A \sqrt{ k U \frac{\log k + \log n}{n} }, 
\]
with $C = C(\eta)>0$  the constant in \eqref{eq:matrix.bernstein.simple.2}.
Assume that $n$ is large enough so that
$v_n  = v - K_{1,n}$ and $u_n = u -K_{2,n}$ are both positive.
Then, for a constant  $C = C(A)>0$,
\begin{equation}
 \inf_{w_n \in \mathcal{W}_n}   \inf_{P\in {\cal
 P}^{\mathrm{OLS}}_n}\mathbb{P}(\beta_{\wS} \in C^*_{\wS}) \geq 1-\alpha -
C\left(\Delta^*_{n,1} + \Delta^*_{n,2} + \Delta_{n,3} \right),
\end{equation}
where $C^*_{\wS}$ is either one of the bootstrap confidence sets in
\eqref{eq:ci.boot.beta}, 
\[
\Delta^*_{n,1}  =  \frac{1}{\sqrt{v_n}}\left(  \frac{
    k^2 \overline{v}_n^2 (\log kn)^7)}{n}\right)^{1/6} , \quad   \Delta^*_{n,2}
    = \frac{ U_n }{ \sqrt{v_n}}  \sqrt{ \frac{k^4 \overline{v}_n \log^2n \log
k}{n\,u_n^6}}
\]
and $\Delta_{n,3}$ is  as in \Cref{thm::big-theorem}.
\end{theorem}

{\bf Remark.} The term $\Delta_{n,3}$ remains unchanged from the
Normal approximating case since it
arises from the Gaussian comparison part, which does not depend on the bootstrap
distribution. 

{\bf Remark.} 
It is important that we use the pairs bootstrap ---
where each pair $Z_i=(X_i,Y_i)$, $i=\mathcal{I}_{2,n}$, is treated as one observation ---
rather than a residual based bootstrap.
In fact, the validity of the residual bootstrap requires the underlying
regression function to be linear, which we do not assume. 
See 
\cite{buja2015models}
for more discussion on this point.
In both cases,
the Berry-Esseen theorem for simple convex sets
(polyhedra with a limited number of faces)
with increasing dimension due to
\cite{cherno1, cherno2}
justifies the method.
In the case of $\beta_{\wS}$
we also need a Taylor approximation
followed by an application of the
Gaussian anti-concentration result
from the same reference.

The coverage rates from \Cref{theorem::beta.boot} are of course no better than the ones obtained in \Cref{thm::big-theorem}, and are consistent with the results of \cite{el2015can}
who found that,
even when the linear model is correct,  the bootstrap does poorly
when $k$ increases.

The coverage accuracy can also be improved by changing
the bootstrap procedure; see Section \ref{section::improving}.

{\bf Remark.}
Our results concern the bootstrap distribution and assume the ability to determine  the quantities $\hat{t}^*_\alpha$ and $(\tilde{t}^*_j, j \in \wS)$ in Equation \Cref{eq:ci.boot.beta}. Of course, they can
	be approximated to an arbitrary level of precision by drawing a large
	enough number $B$ of bootstrap
	samples and then by computing the appropriate empirical quantiles from
	those samples. This will result in an additional approximation error, which can be easily quantified using the DKW inequality (and, for the set $\tilde{C}^*_{\wS}$, also the union bound) and which is, for large $B$,
	negligible compared to the size of the error bounds  obtained above. For simplicity, we do not
	provide these details. Similar considerations apply to all subsequent bootstrap results.

\subsubsection*{The Sparse Case}

Now we briefly discuss the case of sparse fitting where $k = O(1)$
so that the size of the selected model is not allowed to increase
with $n$.
In this case, things simplify considerably.
The standard central limit theorem shows that
$$
\sqrt{n}(\hat\beta - \beta)\rightsquigarrow N(0,\Gamma)
$$
where
$\Gamma  = \Sigma^{-1} \mathbb{E}[(Y-\beta^\top X)^2] \Sigma^{-1}$.
Furthermore,
$\Gamma$ can be consistently estimated by the sandwich estimator
$\hat\Gamma = \hat\Sigma^{-1} A \hat\Sigma^{-1}$
where
$A = n^{-1}\mathbb{X}^\top R \mathbb{X}$,
$\mathbb{X}_{ij} = X_i(j)$,
$R$ is  the $k\times k$ diagonal matrix
with $R_{ii} = (Y_i - X_i^\top \hat\beta)^2$.
By Slutsky's theorem,
valid asymptotic confidence sets can be based on the
Normal distribution with $\hat\Gamma$ in place of $\Gamma$
(\cite{buja2015models}).

However, if $k$ is non-trivial relative to $n$,
then fixed $k$ asymptotics may be misleading.
In this case, the results of the previous section may be more
appropriate.
In particular, replacing $\Gamma$ with an estimate
then has a non-trivial effect on the coverage accuracy.
Furthermore, 
the accuracy depends on
$1/u$ where
$u = \lambda_{\rm min}(\Sigma)$.
But when we apply the results after sample splitting (as is our goal),
we need to define $u$ as
$u = \min_{|S|\leq k} \lambda_{\rm min}(\Sigma_S)$.
As $d$ increases, $u$ can get smaller and smaller even with fixed $k$.
Hence, the usual fixed $k$
asymptotics may be misleading.

{\bf Remark:}
We only ever report inferences for the selected parameters.
The bootstrap provides uniform coverage over all parameters in $S$.
There is no need for a Bonferroni correction. This is because the bootstrap is applied to
$||\hat\beta_{\wS}^* - \hat\beta_{\wS}||_\infty$. However, we also show that
univariate Normal approximations together with Bonferroni adjustments
leads valid hyper-rectangular regions;
see Theorem \ref{thm::bonf}.

\subsection{LOCO Parameters}\label{sec:loco.parameters}

Now we turn to the LOCO parameter $\gamma_{\wS} \in \mathbb{R}^{\wS}$, where $\wS$ is the model selected on the first half of the data. Recall that $j^{\mathrm{th}}$ coordinate
of this parameter is 
\[
\gamma_{\wS}(j) = \mathbb{E}_{X,Y}\Biggl[|Y-\hat\beta_{\wS(j)}^\top X_{\wS(j)}|-
|Y-\hat\beta_{\wS}^\top X_{\wS}| \Big| \mathcal{D}_{1,n} \Biggr],
\]
where $\hat\beta_{\wS} \in \mathbb{R}^{\wS}$ is any estimator of $\beta_{\wS}$, and $\hat\beta_{\wS(j)}$ is obtained 
by re-computing the same estimator on the set of covariates $\wS(j)$
resulting from re-running the same model selection procedure after removing covariate
$X_j$. The model selections $\wS$ and $\wS(j)$ and the estimators $\hat{\beta}_{\wS}$ and
$\hat{\beta}_{\wS}(j)$ are all computed using half of the sample,
$\mathcal{D}_{1,n}$.

In order to derive confidence sets for $\gamma_{\wS}$ we will assume that the data generating distribution belongs to the class ${\cal P}_n'$ of all distributions on $\mathbb{R}^{d+1}$ supported on $[-A,A]^{d+1}$, for some fixed constant $A>0$. Clearly the class $\mathcal{P}_n^{\mathrm{LOCO}}$ is significantly larger then the class $\mathcal{P}_n^{\mathrm{OLS}}$ considered for the projection parameters.

A natural unbiased estimator of $\gamma_{\wS}$ -- conditionally on
$\mathcal{D}_{1,n}$ --  is 
\[\hat{\gamma}_{\wS} = \frac{1}{n} \sum_{i \in \mathcal{I}_{2,n}} \delta_i,
\]
with $(\delta_i,i \in \mathcal{I}_{2,n})$ independent and identically
distributed random vectors in
$\mathbb{R}^{\wS}$ such that, for any $i \in \mathcal{I}_{2,n}$ and  $j \in
\wS$,
\begin{equation}\label{eq:delta.i}
    \delta_i(j) =  \Big| Y_i-\hat\beta_{\wS(j)}^\top X_i(\wS(j)) \Big|-
\Big|Y_i-\hat\beta_{\wS}^\top X_{i}(\wS) \Big|.
\end{equation}
$X_i$ obtained by considering only the coordinates in $S$.

To derive a CLT for $\hat\gamma_{\wS}$
we face  two technical problem. 
First, we require some control on the minimal variance of the coordinates of
the $\delta_i$'s. Since we allow for increasing $k$ and we impose minimal
assumptions on the class of data generating distributions, it is possible that any one variable might have a tiny influence
on the predictions. As a result, we cannot rule out the possibility that the
variance of some coordinate of the $\delta_i$' vanishes. In this case the rate
of convergence in
high-dimensional central limit theorems would be negatively impacted, in ways
that are difficulty to assess. To prevent
such issue we simply redefine $\gamma_{\wS}$ by adding a small amount of
noise with non-vanishing variance.
Secondly, we also need an upper bound on the third moments
of the coordinates of the $\delta_i$'s. In order to keep the presentation
simple,  we will truncate the estimator of
the regression function by hard-thresholding so that it has bounded range
$[-\tau,\tau]$ for a given $\tau>0$. Since both $Y$ and the coordinates of $X$ are
uniformly bounded in absolute value by $A$, this assumption is reasonable.

Thus, we  re-define the vector of LOCO parameters $\gamma_{\wS}$ so that its $j^{\mathrm{th}}$ coordinate is
\begin{equation}\label{eq:new.gamma}
    \gamma_{\wS}(j) = \mathbb{E}_{X,Y, \xi_j}\Biggl[ \left|Y- t_{\tau}\left(
    	\hat\beta_{\wS(j)}^\top X_{\wS(j)} \right) \right|-
    \left| Y-t_{\tau}\left(  \hat\beta_{\wS}^\top X_{\wS} \right) \right| +
\epsilon \xi(j)  \Biggr],
\end{equation}
where $\epsilon > 0$ is a pre-specified small number,  $\xi = (\xi(j), j \in \wS)$ is a random vector comprised of independent $\mathrm{Uniform}(-1,1)$, independent of the
data, and $t_{\tau}$ is the hard-threshold function: for any $x \in \mathbb{R}$,
$t_{\tau}(x)$ is $x$ if $|x| \leq \tau$ and $\mathrm{sign}(x) \tau$ otherwise.
Accordingly, we re-define the estimator $\hat{\gamma}_{\wS}$ of this modified
LOCO parameters as
\begin{equation}\label{eq:new.delta}
    \hat{\gamma}_{\wS}  = \frac{1}{n} \sum_{i \in \mathcal{I}_{2,n}} \delta_i,
\end{equation}
where the $\delta_i$'s are random vector in $\mathbb{R}^{\wS}$ such that the $j^{\mathrm{th}}$ coordinate
of $\delta_i$  is 
\[
     \Big| Y_i- t_{\tau}\left(  \hat\beta_{\wS(j)}^\top
    X_i(\wS(j)) \right) \Big|-
    \Big| Y_i - t_{\tau} \left(  Y_i-\hat\beta_{\wS}^\top X_{i}(\wS) \right)\Big|
     +
    \epsilon \xi_i(j), \quad j \in \wS.
\]

{\bf Remark.}
Introducing additional noise has the effect of making the inference conservative: the confidence
intervals will be slightly wider. 
For small $\epsilon$ and any non-trivial value of $\gamma_{\wS}(j)$ this will presumably have a negligible effect.
For our proofs, adding some additional noise and thresholding the regression
function are advantageous because the first choice will guarantee that  the
empirical covariance matrix of the $\delta_i$'s is non-singular, and the second
choice will imply that the coordinates of $\hat{\gamma}_{\wS}$ are bounded. 
It is possible to let $\epsilon\to 0$ and $\tau \rightarrow \infty$ as $n\to\infty$ at the expense of slower
concentration and Berry-Esseen rates.
For simplicity, we take $\epsilon$ and $\tau$ to be fixed but we will keep explicit track of these quantities in the constants.

Since each coordinate of $\hat{\gamma}_{\wS}$ is an average of random variables
that are bounded in absolute value by $2(A+\tau) + \epsilon$, and $\mathbb{E}\left[
\hat{\gamma}_{\wS} | \mathcal{D}_{1,n}\right] = \gamma_{\wS}$, a standard
bound for the maxima of $k$ bounded (and, therefore, sub-Gaussian) random  variables yields the following
concentration result. As usual, the probability is with respect to the
randomness in the full sample and in the splitting.

\vspace{11pt}
\begin{lemma}
    \[
	\sup_{w_n \in \mathcal{W}_n} \sup_{P  \in \mathcal{P}_n^{\mathrm{LOCO}}}
  \mathbb{P}\left(  \| \hat{\gamma}_{\wS} -
	\gamma_{\wS}\|_\infty \leq  \left( 2(A+\tau) + \epsilon \right) \sqrt{ 2
	    \frac{\log k + \log n}{n} } \right) \geq 1 - \frac{1}{n}.
    \]
\end{lemma}
\begin{proof}
    The bound on $\| \hat{\gamma}_{\wS} - \gamma_{\wS}\|_\infty $ holds with
    probability at least $1 - \frac{1}{n}$ conditionally on $\mathcal{D}_{1,n}$
    and the outcome of data splitting, and uniformly over the choice of the procedure
    $w_n$ and of the distribution $P$. Thus, the uniform validity of the bound
    holds also unconditionally.  
\end{proof}

We now construct confidence sets for $\gamma_{\wS}$. Just like we did with the
projection parameters, we consider two types of methods: one based on Normal
approximations and the other on the bootstrap. 

\subsubsection*{Normal Approximation}
Obtaining  high-dimensional Berry-Esseen
bounds for $\hat\gamma_{\wS}$ is nearly straightforward since, conditionally on $\mathcal{D}_{1,n}$ and the splitting, $\hat\gamma_{\wS}$ is
just a vector of averages of bounded and independent variables with non-vanishing variances.
Thus, there is no need for a Taylor approximation and we can apply directly the results in \cite{cherno2}.
In addition, we find that 
the accuracy of the confidence sets for this LOCO parameter is higher than for
the projection parameters.

Similarly to what we did in \Cref{sec:projection},  we derive two approximate
confidence sets: one is an $L_\infty$ ball and the other is a hyper-rectangle
whose $j^{\mathrm{th}}$ side length is proportional to the standard deviation of the $j^{\mathrm{th}}$
coordinate of $\hat{\gamma}_{\wS}$. Both sets are centered at
$\hat{\gamma}_{\wS}$.

Below, we let $\alpha \in (0,1)$ be fixed and let 
\begin{equation}\label{eq:Sigma.loco}
    \hat \Sigma_{\wS} = \frac{1}{n} \sum_{i=1}^n \left( \delta_i - \hat{\gamma}_{\wS}
    \right)\left( \delta_i - \hat{\gamma}_{\wS}  \right)^\top,
\end{equation}
be the empirical covariance matrix of the $\delta_i$'s.
The first confidence set is the $L_\infty$ ball
\begin{equation}\label{eq::gamma.conf-rectangle}
    \hat{D}_{\wS} = \Big\{ \gamma \in \mathbb{R}^k \colon \|\gamma -
    \hat{\gamma}_{\wS} \|_\infty \leq \hat{t}_\alpha \Big\},
\end{equation}
 where $\hat{t}_\alpha$ is such that
 \[
     \mathbb{P}\left( \| Z_n \|_\infty  \leq \hat{t}_\alpha \right) = 1 -
     \alpha,
 \]
 with $Z_n \sim N(0,\hat{\Sigma}_{\wS}$).
The second confidence set we construct is instead the hyper-rectangle
\begin{equation}\label{eq:gamma.hyper:CI}
    \tilde{D}_{\wS} = \bigotimes_{j \in \wS} \hat{D}(j), 
\end{equation}
where, for any $j \in \wS$, $\tilde{D}(j) = \left[ \hat{\gamma}_{\wS}(j) -\hat{t}_{j,\alpha}, \hat{\gamma}_{\wS}(j) +\hat{t}_{j,\alpha}
\right]$,
with $    \hat{t}_{j,\alpha} = z_{\alpha/2k} \sqrt{
    \frac{\hat\Sigma_{\wS}(j,j)}{n}  }.$

The above confidence sets have the same form as the confidence sets for the
projection parameters \eqref{eq::conf-rectangle} \eqref{eq:hyper:CI}. The key
difference is that for the projection parameters we use the estimated covariance
of the linear approximation to $\hat{\beta}_{\wS}$, while for the LOCO parameter
$\hat{\gamma}_{\wS}$ we rely on the empirical covariance \eqref{eq:Sigma.loco},
which is a much simpler estimator to compute.

\vspace{11pt}

In the next result we derive coverage rates for both confidence sets.

\begin{theorem}\label{thm::CLT2}
    There exists a universal  constant $C > 0$ such that   
    \begin{equation}\label{eq:loco.coverage1}
	\inf_{w_n \in  \mathcal{W}_n} \inf_{P \in \mathcal{P}_n^{\mathrm{LOCO}}} \mathbb{P} \left(
    \gamma_{\wS} \in \widehat{D}_{\wS}  \right) \geq 1 - \alpha - C
    \left(  
    \mathrm{E}_{1,n} +  \mathrm{E}_{2,n}  \right)- \frac{1}{n},
\end{equation}
    and
    \begin{equation}\label{eq:loco.coverage2}
	\inf_{w_n \in  \mathcal{W}_n} \inf_{P \in \mathcal{P}_n^{\mathrm{LOCO}}} \mathbb{P} \left(
    \gamma_{\wS} \in \tilde{D}_{\wS}  \right) \geq 1 - \alpha - C \left(
    \mathrm{E}_{1,n}  +  \tilde{\mathrm{E}}_{2,n}  \right) - \frac{1}{n},
\end{equation}
    where
    \begin{align}
	\label{eq:E1n}
    \mathrm{E}_{1,n}   &= \frac{2(A+\tau) + \epsilon }{\epsilon} \left(\frac{  (\log n
    k)^7}{n}\right)^{1/6},\\
    \label{eq:E2n}
    \mathrm{E}_{2,n}  & =    
    \frac{N_n^{1/3} (2 \log 2k)^{2/3}}{\underline{\epsilon}^{2/3}},\\
    \label{eq:tildeE2.n}
    \tilde{\mathrm{E}}_{2,n}  &= \min \left\{ \mathrm{E}_{2,n},\frac{ N_n z_{\alpha/(2k)}}{\epsilon^2}
    \left(\sqrt{ 2 + \log(2k ) } + 2 \right) \right\}
\end{align}
and
\begin{equation}\label{eq:Nn}
N_n =  \left( 2(A+\tau) + \epsilon \right)^2 \sqrt{ 
	    \frac{4\log k + 2 \log n}{n} }.
	\end{equation}

\end{theorem}

\noindent {\bf Remark.} The term $\mathrm{E}_{1,n}$ quantifies the error in
applying the 
high-dimensional normal approximation to  $\hat{\gamma}_{\wS} - \gamma_{\wS}$,
 given in \cite{cherno2}.
The second error term $\mathrm{E}_{2,n}$ is due to the fact that $\Sigma_{\wS}$
is unknown and has to be estimated using the empirical covariance matrix
$\widehat{\Sigma}_{\wS}$. To establish $\mathrm{E}_{2,n}$ we use the Gaussian
comparison Theorem \ref{thm:comparisons}. 
 We point out that  the dependence in $\epsilon$ displayed in the term
$\mathrm{E}_{2,n}$ above does not follow
directly from Theorem 2.1 in \cite{cherno2}. It can be obtained by tracking
constants and using Nazarov's inequality \Cref{thm:anti.concentration} in the
proof of that result. See \Cref{thm:high.dim.clt}  in \Cref{app:high.dim.clt}
for details.

The accuracy of the confidence set \eqref{eq:gamma.hyper:CI} can be easily established to be of order $O \left( \sqrt{ \frac{\log k}{n}} \right)$, a fact made precise in the following result.
\begin{corollary}\label{cor:accuracy.LOCO}
	With probability at least $ 1- \frac{1}{n}$, the maximal length of the sides of the hyper-rectangle $\tilde{C}_n$ is bounded by
\[
C \left(2(A + \tau) + \epsilon \right) \sqrt{  \frac{\log k}{n}  \left( 1 + \frac{(4\log k + 2 \log n)^{1/2}}{n^{1/2}}\right)},
\]
for a universal constant $C>0$, uniformly over all $P \in  \mathcal{P}_n^{\mathrm{LOCO}}$.
 
\end{corollary}

\subsubsection*{The Bootstrap}

We now demonstrate the coverage of the paired bootstrap version of the confidence set
for $\gamma_{\wS}$ given above in \eqref{eq::gamma.conf-rectangle}.

The bootstrap distribution is the empirical measure associated to the $n$
triplets $\left\{  (X_i,Y_i,\xi_i), i \in
    \mathcal{I}_{2,n} \right\}$ and conditionally on $\mathcal{D}_{1,n}$.
Let $\hat{\gamma}^*_{\wS}$ denote the estimator of the LOCO parameters
\eqref{eq:new.gamma} of the form \eqref{eq:new.delta} computed
from an i.i.d. sample of size $n$ drawn from the bootstrap distribution. Notice that $\mathbb{E}\left[
\hat{\gamma}^*_{\wS} \Big|  (X_i,Y_i,\xi_i), i \in
    \mathcal{I}_{2,n} \right] = \hat{\gamma}_{\wS}$.
For a given $\alpha \in (0,1)$, let $\hat{t}^*_\alpha$ be the smallest
    positive number such that 
    \[
	\mathbb{P}\left( \sqrt{n} \| \hat{\gamma}^*_{\wS} - \hat{\gamma}_{\wS}\|
	\leq \hat{t}^*_\alpha \Big| (X_i,Y_i,\xi_i), i \in
    \mathcal{I}_{2,n} \right) \geq 1 - \alpha.
    \]
    Next, let $(\tilde{t}^*_j, j \in \wS)$ be such that
    \[
	\mathbb{P}\left( \sqrt{n} | \hat{\gamma}^*_{\wS}(j) - \hat{\gamma}_{\wS}
	(j) \leq \tilde{t}^*_j, \forall j \Big| (X_i,Y_i,\xi_i), i \in
    \mathcal{I}_{2,n} \right) \geq 1 - \alpha.
    \]
    In particular, using the union bound, each $\tilde{t}^*_j$ can be chosen to
    be the largest positive number such that
    \[
	\mathbb{P}\left( \sqrt{n} | \hat{\gamma}^*_{\wS}(j) - \hat{\gamma}_{\wS}
	(j) > \tilde{t}^*_j, \Big| (X_i,Y_i,\xi_i), i \in
	\mathcal{I}_{2,n} \right) \leq \frac{\alpha}{k}.
    \]
Consider the following two bootstrap confidence sets: 
\begin{equation}\label{eq:ci.boot.loco}
	    \hat{D}^*_{\wS} = \left\{ \gamma \in \mathbb{R}^{\wS} \colon \|
	    \gamma - \hat{\gamma}_{\wS}
	    \|_\infty \leq \frac{ \hat{t}^*_{\alpha}}{\sqrt{n}} \right\} \quad
	    \text{and} \quad 
	    \tilde{D}^*_{\wS} = \left\{ \gamma \in \mathbb{R}^{\wS} \colon |
	    \gamma_j - \hat{\gamma}_{\wS}
	    | \leq \frac{ \tilde{t}^*_{j}}{\sqrt{n}}, \forall j \right\}.
	\end{equation}

\begin{theorem}
\label{thm:boot.loco}
Using the same notation as in \Cref{thm::CLT2}, assume that $n$ is large enough
so that $\epsilon_n  = \sqrt{ \epsilon^2 - N_n }$ is positive.
Then there exists a universal constant $C>0$ such that the coverage of both
confidence sets in \eqref{eq:ci.boot.loco} is at least  
\[
 1 - \alpha
    - C\left(  \mathrm{E}^*_{1,n} +
    \mathrm{E}_{2,n} + \frac{1}{n} \right),
\]
where
\[
    \mathrm{E}^*_{1,n}   = \frac{2(A+\tau) + \epsilon_n }{\epsilon_n} \left(\frac{  (\log n
    k)^7}{n}\right)^{1/6}.
\]
	\end{theorem}

\subsection{Median LOCO parameters}

For the median loco parameters $(\phi_{\wS}(j), j \in \wS)$ given in \eqref{eq:median.LOCO} finite sample
inference is  relatively straightforward using standard confidence intervals for the median
based on order statistics.
In detail, for each $j \in \wS$ and $i \in \mathcal{I}_{2,n}$, recall the
definition of $\delta_i(j)$ in \eqref{eq:delta.i} and  let $\delta_{(1)}(j) \leq
\ldots \leq \delta_{(n)}(j)$ be the corresponding order statistics. We will
not impose any restrictions on the data generating distribution. In particular, for each $j \in \wS$, the median of  $\delta_i(j)$ needs not be  unique.  
Consider
the interval
\[
    E_j = [ \delta_{(l)}(j), \delta_{(u)}(j)]     
\]
where
\begin{equation}\label{eq:lu}
    l =  \Big\lceil \frac{n}{2} - \sqrt{\textcolor{black}{\frac{n}{2}} \log\left( \frac{2k}{\alpha}\right)} \Big\rceil
 \quad \text{and} \quad 
 u =  \Big\lfloor \frac{n}{2} + \sqrt{\textcolor{black}{\frac{n}{2}} \log\left( \frac{2k}{\alpha}\right)}
 \Big\rfloor 
\end{equation}
and construct the hyper-cube
 \begin{equation}
   \hat{E}_{\wS} = \bigotimes_{j \in \wS}^n E_j.	
 \end{equation}
 Then, a standard result about confidence sets for medians along with union
 bound implies that $\hat{E}_{\wS}$ is a $1-\alpha$ confidence set for the median LOCO parameters, uniformly over $\mathcal{P}_n$.

\begin{proposition}
For every $n$,
\begin{equation}
\inf_{w_n \in \mathcal{W}_n} \inf_{P\in {\cal P}_{n}}\mathbb{P}(\phi_{\wS} \in
\hat{E}_{\wS}) \geq 1-\alpha.
\end{equation}
\end{proposition}

{\bf Remark.} Of course, if the median of $\delta_i(j)$ is not unique,
the length of the corresponding confidence interval does not shrink ad
$n$ increases. But if the median is unique for each $j \in \wS$, and
under addition smoothness conditions, we obtain the maximal length the
side of the confidence rectangle $\hat{E}_{\wS}$ is of order $O \left(
\sqrt{\frac{\log k + \log n}{n}} \right)$, with high probability.

\begin{theorem}
\label{thm::median}
Suppose that there exists positive numbers $M$ and $\eta$ such that, for each
 $j \in \wS$,  the cumulative distribution function of each $\delta_i(j)$
is differentiable with derivative no smaller than $M$ at all points at a
distance no larger than
$\eta$ from its (unique) median. Then, 
for all $n$ for which 
\[
  \frac{1}{n}  +
    \sqrt{\frac{1}{2n}\log\left(\frac{2k}{\alpha}\right)} + \sqrt{
	\frac{ \log 2kn}{2n} }
 \leq \eta M,
\]
the sides of $\hat{E}_{\wS}$ have length uniformly bounded by
\[
\frac{2}{M} \left(   \frac{1}{n}  +
    \sqrt{\frac{1}{2n}\log\left(\frac{2k}{\alpha}\right)} + \sqrt{
	\frac{ \log 2kn}{2n} } \right),
\]
with probability at least $1 - \frac{1}{n}$.
\end{theorem}

\subsection{Future Prediction Error}

To construct a confidence interval for the future prediction error parameter $\rho_{\wS}$ consider the set 
$$
\hat{F}_{\wS} = \Bigl[\hat\rho_S - z_{\alpha/2} s/\sqrt{n},\ \hat\rho_S + z_{\alpha/2} s/\sqrt{n}\Bigr]
$$
where $z_{\alpha/2}$ is the $1-\alpha/2$ upper quantile of a standard normal distribution,
\[
\hat\rho_{\wS} = \frac{1}{n}\sum_{i\in {\cal I}_2}\sum_i A_i, \quad 
s^2 = \frac{1}{n}\sum_{i\in {\cal I}_2}(A_i - \hat\rho_{\wS})^2, \quad
\text{and} \quad A_i = |Y_i - \hat\beta_{\wS}^\top X_{i}(\wS)|, \forall i \in \mathcal{I}_{2,n}.
\]
For any $P$,
let $\sigma^2_n = \sigma^2_n(P) = \mathrm{Var}_P(A_1)$ and
$\mu_{3,n}= \mu_{3,n}(P) =  \mathbb{E}_P \left[ |A_1 - \mathbb{E}_P[A_1]|^3 \right]$.
Then, by the one-dimensional Berry-Esseen theorem:
\[
  \inf_{w_n \in \mathcal{W}_n}  \mathbb{P}(\rho_{\wS} \in \hat{F}_{\wS}) \geq 1-\alpha - O \left( \frac{ \mu_{3,n}}{\sigma_n \sqrt{n}} \right).
\]
In order to obtain uniform coverage accuracy guarantees, we may rely
on a modification of the target parameter that we implemented for the
LOCO parameters in \Cref{sec:loco.parameters} and redefine the
prediction parameter to be
\[
\rho_{\wS} = \mathbb{E} \left[ |Y - t_\tau( \hat\beta_{\wS}^\top X_(\wS)) | + \epsilon \xi \right] ,
\]
where $t_{\tau}$ is the hard-threshold function (for any $x \in
\mathbb{R}$, $t_{\tau}(x)$ is $x$ if $|x| \leq \tau$ and
$\mathrm{sign}(x) \tau$ otherwise) and $\xi$ is independent noise
uniformly distributed on $[-1,1]$. Above, the positive parameters
$\tau$ and $\epsilon$ are chosen to ensure that the variance of the
$A_i$'s does not vanish and that their third moment does not explode
as $n$ grows. With this modification, we can ensure that $\sigma^2_n
\geq \epsilon^2$ and $\mu^{3,n} \leq \left( A + \tau + \epsilon
\right)^3$ uniformly in $n$ and also $s \leq 4 (A + \tau +
\epsilon)^2$, almost surely. Of course, we may let $\tau$ and
$\epsilon$ change with $n$ in a controlled manner. But for fixed
choices of $\tau$ and $\epsilon$ we obtain the following parametric
rate for $\rho_{\wS}$, which holds for all possible data generating
distributions:
\[
  \inf_{w_n \in \mathcal{W}_n}  \mathbb{P}(\rho_{\wS} \in \hat{F}_{\wS}) \geq 1-\alpha - C \left( \frac{1}{\sqrt{n}} \right),
\]
for a constant dependent only on $A$, $\tau$ and $\epsilon$. Furthermore, the length of the confidence interval is parametric, of order $\frac{1}{\sqrt{n}}$.

\section{Prediction/Accuracy Tradeoff: Comparing Splitting to Uniform Inference}
\label{section::splitornot}

There is a price to pay for sample splitting:
the selected model may be less accurate because
only part of the data
are used to select the model.
Thus, splitting creates gains in accuracy and robustness for inference but
with a possible loss of prediction accuracy.
We call this the {\em inference-prediction tradeoff}.
In this section we study this phenomenon by comparing
splitting with uniform inference (defined below).
We use uniform inference for the comparison since
this is the any other method we know of that
achieves (\ref{eq::honest}).
We study this use with a simple model
where it is feasible to compare splitting with uniform inference.
We will focus on the {\em many means problem} which is similar to
regression with a balanced, orthogonal design.
The data are
$Y_1,\ldots, Y_{2n} \sim P$
where
$Y_i\in\mathbb{R}^D$.
Let
$\beta = (\beta(1),\ldots, \beta(D))$
where
$\beta(j) = \mathbb{E}[Y_i(j)]$.
In this section, the model ${\cal P}_{n}$
is the set of probability distributions on $\mathbb{R}^D$ such that
$\max_j \mathbb{E}|Y(j)|^3 < C$ and $\min_j {\rm Var}(Y(j)) > c$ for some positive $C$ and $c$, which do not change with $n$ or $D$ (these assumptions could of course be easily relaxed). Below, we will only track the dependence on $D$ and $n$ and will  use the notation $\preceq$ to denote inequality up to constants.

To mimic forward stepwise regression ---
where we would choose a covariate to maximize correlation with the outcome ---
we consider choosing $j$ to maximize the mean.
Specifically,
we take
\begin{equation}\label{eq::J}
  \wS  \equiv w(Y_1,\ldots, Y_{2n}) =\argmax_j \overline{Y}(j) 
\end{equation}
where
$\overline{Y}(j) = (1/2n)\sum_{i=1}^{2n} Y_i(j)$.
Our goal is to infer
the random parameter $\beta_{\wS}$.
The number of models is $D$.
In forward stepwise regression with $k$ steps
and $d$ covariates, the number of models
is $D = d^k$.
So the reader is invited to think of $D$ as being very large.
We will compare splitting versus non-splitting 
with respect to three goals:
estimation, inference and prediction accuracy.

{\bf Splitting:}
In this case we take
Let ${\cal D}_{1,n} = \{i: \ 1 \leq i \leq n\}$ and
${\cal D}_{2,n} = \{i: \ n+1 \leq i \leq 2n\}$.
Then
\begin{equation}\label{eq::J1}
  \wS  \equiv w(Y_1,\ldots, Y_n) =\argmax_j \overline{Y}(j) 
\end{equation}
where
$\overline{Y}(j) = (1/n)\sum_{i=1}^n Y_i(j)$.
The point estimate and confidence interval
for the random parameter $\beta_{\wS}$ are
$$
\hat\beta_{\wS} = \frac{1}{n}\sum_{i=n+1}^{2n} Y_i(\wS)
$$
and
$$
\hat{C}_{\wS}= [\hat\beta_{\wS} - s z_{\alpha/2}/\sqrt{n},\ \hat\beta_{\wS} + s z_{\alpha/2}/\sqrt{n}]
$$
where
$s^2 = n^{-1}\sum_{i=n+1}^{2n} (Y_i(\wS) - \hat\beta_{\wS})^2$.

{\bf Uniform Inference (Non-Splitting).}
By ``non-splitting'' we mean that the selection rule and estimator
are invariant under permutations of the data.
In particular, we consider uniform inference which is defined as follows.
Let $\hat\beta(s) = (2n)^{-1}\sum_i Y_i(s)$ be the average over all the observations.
Let $\hat{S} = \argmax_s \hat\beta(s)$.
Our point estimate is $\hat{\beta}_{\wS} \equiv \hat\beta(\hat{S})$.
Now define
$$
F_{n}(t) = \mathbb{P}(\sup_s \sqrt{2n}|\hat\beta(s)-\beta(s)| \leq t).
$$
We can consistently estimate $F_{n}$ by the bootstrap:
$$
\hat F_{n}(t) = \mathbb{P}( \sup_s \sqrt{2n}\left|\hat\beta^*(s)-\hat
\beta(s)\right| \leq t\,| Y_1,\ldots, Y_{2n}).
$$
A valid confidence set for $\beta$ is
$R= \{ \beta:\ ||\beta - \hat\beta||_\infty \leq t/\sqrt{2n}\}$
where $t=\hat F_{n}^{-1}(1-\alpha)$.
Because this is uniform over all possible models
(that is, over all $s$),
it also defines a valid confidence interval for a randomly selected coordinate.
In particular, we can define
$$
\hat{C}_{\wS}= [\hat\beta_{\hat{S}} - t/\sqrt{2n},\
\hat\beta_{\hat{S}} + t/\sqrt{2n}]
$$
Both confidence intervals satisfy
(\ref{eq::honest}).

We now compare $\hat\beta_{\wS}$ and $\hat{C}_{\wS}$ for both the splitting and
non-splitting procedures.  The reader should keep in mind that, in general,
$\hat{S}$ might be different between the two procedures, and hence
$\beta_{\wS}$ may be different.  The two procedures might be estimating
different parameters.  We discuss that issue shortly.

\vspace{11pt}

{\bf Estimation.}
First we consider estimation accuracy.

\vspace{11pt}

\begin{lemma}
\label{lemma::est-accuracy}
For the splitting estimator:
\begin{equation}
  \sup_{P\in {\cal P}_{n}}\mathbb{E}|\hat\beta_{\wS}-\beta_{\wS}| \preceq n^{-1/2}.
\end{equation}
For non-splitting we have
\begin{equation}\label{eq::lower1}
  \inf_{\hat\beta}\sup_{P\in {\cal P}_{n}}
  \mathbb{E}|\hat\beta_{\wS}-\beta_{\wS}| \succeq
\sqrt{\frac{\log D}{n}}.
\end{equation}
The above is stated for the particular selection rule $\wS = \argmax_s
\hat{\beta}_s$, but the splitting-based result holds for general 
selection rules $w\in\mathcal{W}_n$, so that for splitting
\begin{equation}
  \sup_{w\in {\cal W}_n}\sup_{P\in {\cal
  P}_{n}}\mathbb{E}|\hat\beta_{\wS}-\beta_{\wS}| \preceq n^{-1/2}
\end{equation}
and for non-splitting 
\begin{equation}\label{eq::lower2}
  \inf_{\hat\beta}\sup_{w\in {\cal W}_{2n}}\sup_{P\in {\cal P}_{n}}
  \mathbb{E}|\hat\beta_{\wS}-\beta_{\wS}| \succeq
\sqrt{\frac{\log D}{n}}.
\end{equation}
\end{lemma}

Thus, the splitting estimator converges at a $n^{-1/2}$ rate.
Non-splitting estimators have a slow rate, even with the added assumption of Normality.
(Of course, the splitting estimator and non-splitting estimator may in fact
be estimating different randomly chosen parameters.
We address this issue when we discuss prediction accuracy.)

\vspace{11pt}

{\bf Inference.}
Now we turn to inference.
For splitting, we use the usual Normal interval
$\hat{C}_{\wS} = [\hat\beta_{\wS}-z_\alpha s/\sqrt{n},\
\hat\beta_{\wS}+z_\alpha s/\sqrt{n}]$
where $s^2$ is the sample variance from ${\cal D}_{2,n}$.
We then have, as a direct application of the one-dimensional Berry-Esseen theorem, that:

\vspace{11pt}

\begin{lemma}
  Let $\hat{C}_{\wS}$ be the splitting-based confidence set. Then,
  \begin{equation}\label{eq:lem12a}
    \inf_{P\in {\cal P}_{n}}\mathbb{P}(\beta_{\hat{S}}\in \hat{C}_{\wS}) = 1-\alpha - \frac{c}{\sqrt{n}}
\end{equation}
for some $c$.
Also,
\begin{equation}\label{eq:lem12b}
  \sup_{P\in {\cal P}_{n}}\mathbb{E}[\nu(\hat{C}_{\wS})] \preceq n^{-1/2}
\end{equation}
where $\nu$ is Lebesgue measure. More generally,
  \begin{equation}
    \inf_{w\in\mathcal{W}_n}\inf_{P\in {\cal
    P}_{n}}\mathbb{P}(\beta_{\hat{S}}\in \hat{C}_{\wS}) = 1-\alpha - \frac{c}{\sqrt{n}}
\end{equation}
for some $c$, and 
\begin{equation}
  \sup_{w\in\mathcal{W}_n}\sup_{P\in {\cal P}_{n}}\mathbb{E}[\nu(\hat{C}_{\wS})] \preceq n^{-1/2}
\end{equation}
\end{lemma}

\vspace{11pt}

\begin{lemma}
  Let $\hat{C}_{\wS}$ be the uniform confidence set. Then,
\begin{equation}
  \inf_{P\in {\cal P}_{n}} \mathbb{P}(\beta_{\hat{S}}\in \hat{C}_{\wS}) = 1-\alpha - \left(\frac{ c (\log D)^7 }{n}\right)^{1/6}
\end{equation}
for some $c$.
Also,
\begin{equation}
  \sup_{P\in {\cal P}_{2n}}\mathbb{E}[\nu(\hat{C}_{\wS})] \succeq \sqrt{\frac{\log D}{n}}.
\end{equation}
\end{lemma}

The proof is a straightforward application of results in \cite{cherno1, cherno2}.
We thus see that the splitting method has better coverage
and narrower intervals,
although we remind the reader that the two methods may be estimating different parameters.

\vspace{11pt}

{\bf Can We Estimate the Law of $\hat\beta(\hat{S})$?}
An alternative non-splitting method to uniform inference is to estimate the
law $F_{2n}$ of $\sqrt{2n}(\hat\beta_{\wS} - \beta_{\wS})$.
But we show that the law of
$\sqrt{2n}(\hat\beta_{\wS}-\beta_{\wS})$
cannot be consistently estimated
even if we assume that the data are Normally distributed
and even if $D$ is fixed (not growing with $n$).
This was shown for fixed population parameters
in \cite{leeb2008can}.
We adapt their proof to the random parameter case
in the following lemma.

\vspace{11pt}

\begin{lemma}
\label{lemma::contiguity}
Suppose that $Y_1,\ldots,Y_{2n} \sim N(\beta,I)$.
Let $\psi_n(\beta) = \mathbb{P}(\sqrt{2n}(\hat\beta_{\wS} -
\beta_{\wS})\leq t)$.
There is no uniformly consistent estimator of $\psi_n(\beta)$.
\end{lemma}

{\bf Prediction Accuracy.}
Now we discuss prediction accuracy which is where splitting pays a price.
The idea is to identify a population quantity $\theta$ that model
selection is implicitly targeting and compare splitting versus non-splitting
in terms of how well they estimate $\theta$.
The purpose of model selection in regression is to choose a model with low prediction error.
So, in regression, we might take $\theta$
to be the prediction risk of the best linear model with $k$ terms.
In our many-means model,
a natural analog of this
is the parameter $\theta = \max_j \beta(j)$.

We have the following lower bound,
which applies over all estimators both
splitting and non-splitting.
For the purposes of this lemma, we use Normality.
Of course, the lower bound is even larger if we drop Normality.

\vspace{11pt}

\begin{lemma}
\label{lemma::many-means-bound}
Let $Y_1,\ldots, Y_n \sim P$
where $P=N(\beta,I)$,
$Y_i\in\mathbb{R}^D$,
and $\beta \in \mathbb{R}^D$.
Let $\theta = \max_j \beta(j)$.
Then
$$
\inf_{\hat\theta}\sup_{\beta}E[ (\hat\theta - \theta)^2] \geq \frac{2\log D}{n}.
$$
\end{lemma}

To understand the implications of this result, let us write
\begin{equation}
\hat\beta(S) - \theta =
\underbrace{\hat\beta(S) - \beta(S)}_{L_1} + 
\underbrace{\beta(S) - \theta}_{L_2}.
\end{equation}
The first term, $L_1$, is the focus of most research on post-selection inference.
We have seen it is small for splitting and large for non-splitting.
The second term takes into account the variability due to model
selection which is often ignored.
Because $L_1$ is of order $n^{-1/2}$ for splitting,
and the because the sum is of order
$\sqrt{\log D/n}$
it follows that splitting must, at least in some cases,
pay a price by have $L_2$ large.
In regression, this would correspond to the fact that, in some cases,
splitting leads to models with lower predictive accuracy.

Of course, these are just lower bounds.
To get more insight, we consider a numerical example.
Figure (\ref{fig::price}) shows a plot of the risk of
$\hat\beta(\hat{S})=\overline{Y}(\hat{S})$
for $2n$ (non-splitting) and $n$ (splitting).
In this example we see that indeed, the splitting estimator
suffers a larger risk.
In this example, $D=1,000$,
$n=50$,
and $\beta = (a,0,\ldots, 0)$.
The horizontal axis is $a$ which is the gap between the largest and second largest mean.

\begin{figure}
\begin{center}
\includegraphics[scale=.5]{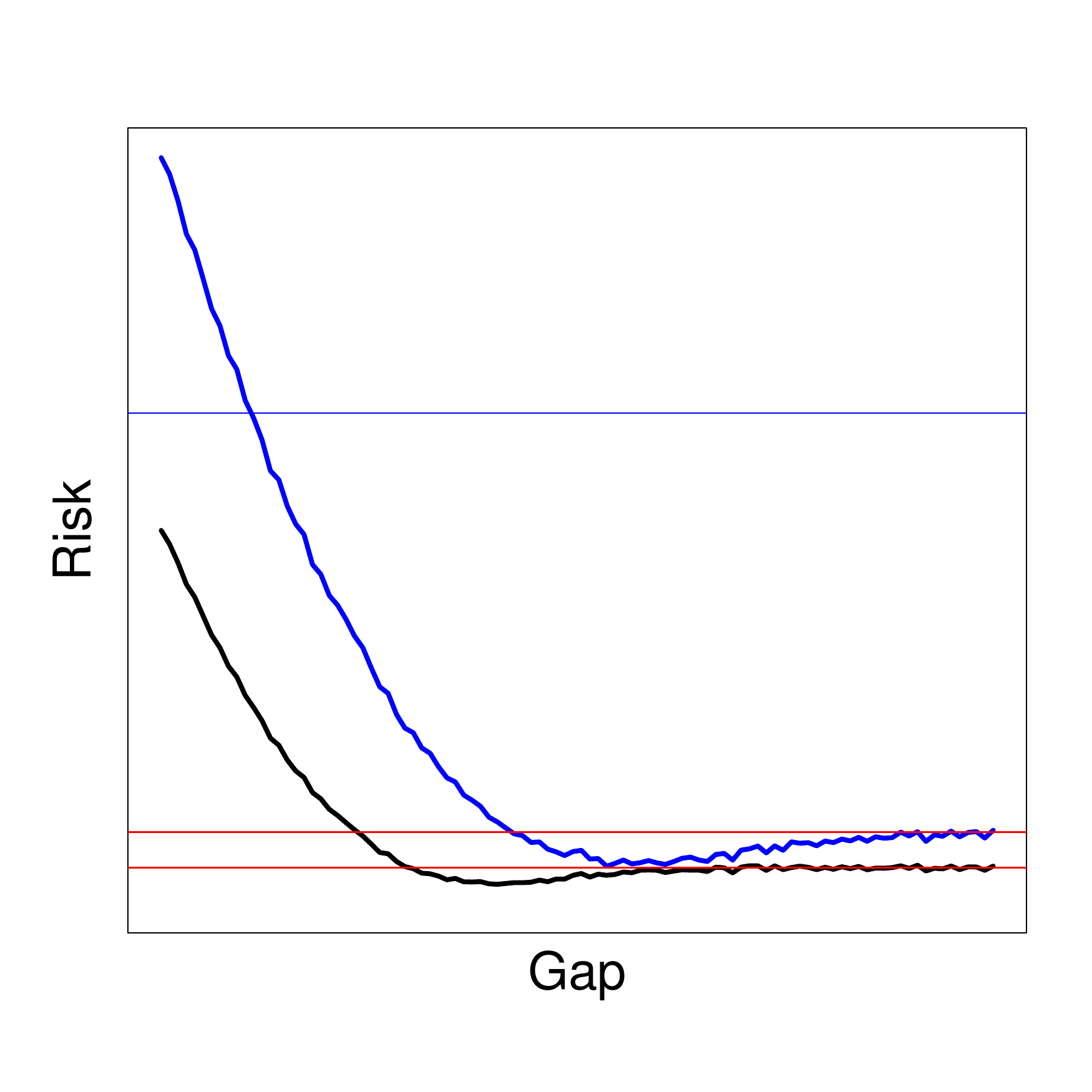}
\end{center}
\caption{\emph{Horizontal axis: the gap $\beta_{(1)} - \beta_{(2)}$.
Blue risk: risk of splitting estimator.
Black line: risk of non-splitting estimator.}}
\label{fig::price}
\end{figure}

To summarize: splitting gives more precise estimates and coverage
for the selected parameter
than non-splitting (uniform) inference.
But the two approaches can be estimating different parameters.
This manifests itself by the fact that splitting can lead to less precise 
estimates of the population parameter $\theta$.
In the regression setting, this would correspond to the fact that
splitting the data can lead to selecting models with poorer
prediction accuracy.

\section{Comments on Non-Splitting Methods}
\label{section::comments}

There are several methods for constructing confidence intervals in
high-dimensional regression.  Some approaches are based on debiasing the lasso
estimator \citep[e.g.,][See Section \ref{sec:related}]{zhang2014confidence,
  vandegeer2014asymptotically, javanmard2014confidence, nickl2013confidence}.
  These approaches tend to require that the linear model to be correct as well
  as assumptions on the design, and tend to target the true $\beta$ which is
  well-defined in this setting.  Some partial exceptions exist:
  \cite{peter.sarah.2015} relaxes the requirement of a correctly-specified
  linear model, while \cite{meinshausen2015group} removes the design
  assumptions.  In general, these debiasing approaches do not provide uniform,
  assumption-free guarantees.

\cite{lockhart2014significance, lee2016exact,taylor2014exact}
do not require the linear model
to be correct nor do they require design conditions.
However, their results only hold for parametric models.
Their method works by inverting a pivot.

In fact, inverting a pivot is, in principle, a very general approach.
We could even use inversion in the nonparametric framework as follows.
For any $P\in {\cal P}$ and any $j$
define $t(j,P)$ by
$$
\mathbb{P}( \sqrt{n}|\hat\beta_S(j) - \beta_S(j)| > t(j,P)) = \alpha.
$$
Note that,
in principle, $t(j,P)$ is known.
For example, we could find $t(j,P)$ by simulation.
Now let
$A = \{P\in {\cal P}:\ \sqrt{n}|\hat\beta_S(j) - \beta_S(j)| < t(j,P)\}$.
Then $\mathbb{P}(P\in A)\geq 1-\alpha$ for all $P\in {\cal P}$.
Write
$\beta_j(S) = f(P,Z_1,\ldots, Z_n)$.
Let
$C= \{f(P,Z_1,\ldots, Z_n):\ P\in A\}$.
It follows that
$\mathbb{P}(\beta_j(S)\in C) \geq 1-\alpha$ for all $P\in {\cal P}$.
Furthermore, we could also choose $t(j,P)$ to satisfy
$\mathbb{P}( \sqrt{n}|\hat\beta_S(j) - \beta_S(j)| > t(j,P)|E_n) = \alpha$
for any event $E_n$ which would given conditional confidence intervals if desired.

There are two problems with this approach.
First, the confidence sets would be huge.
Second, it is not computationally feasible to find $t(j,P)$
for every $P\in {\cal P}$.
The crucial and very clever observation in 
\cite{lee2016exact}
is that
if we restrict to a parametric model
(typically they assume a Normal model with known, constant variance)
then, by choosing $E_n$ carefully,
the conditional distribution reduces, by sufficiency,
to a simple one parameter family.
Thus we only need to find $t$ for this one parameter family which is feasible.
Unfortunately, the method does not provide confidence guarantees of the form
(\ref{eq::honest}) which is the goal of this paper.

\cite{berk2013valid}
is closest to providing the kind of guarantees we have considered here.
But as we discussed in the previous section,
it does not seem to be extendable to the assumption-free framework.

None of these comments is meant as a criticism of the aforementioned methods.
Rather, we just want to clarify that these methods 
are not comparable to our results because
we require uniformity over ${\cal P}$.
Also, except for the method of 
\cite{berk2013valid},
none of the other methods provide any guarantees over unknown selection rules.

\section{Numerical Examples}
\label{section::simulation}

In this section we briefly consider a few illustrative examples.  In a
companion paper, we provide detailed simulations comparing all of the recent
methods that have proposed for inference after model selection.  It would take
too much space, and go beyond the scope of the current paper, to include these
comparisons here.

We focus on linear models, and in particular on inference for the projected
parameter $\beta_{\wS}$ and the LOCO parameter $\gamma_{\wS}$ of \Cref{sec:projection} and \Cref{sec:loco.parameters}, respectively.  The
data are drawn from three distributions:
\begin{description}
  \item[Setting A] \emph{Linear and sparse with Gaussian noise.} A linear model
    with $\beta_i\sim U[0,1]$ for $j=1,\dots,5$ and $\beta_j=0$ otherwise.  
  \item[Setting B] \emph{Additive and sparse with $t$-distributed noise.} An
    additive model with a cubic and a quadratic term, as well as three linear
    terms, and $t_5$-distributed additive noise.
  \item[Setting C] \emph{Non-linear, non-sparse, $t$-distributed noise.}  The
    variables from Setting $B$ are rotated randomly to yield a dense model.
\end{description}
In Settings A and B, $n=100$ (before splitting); in Setting C $n=200$. In all Settings $p=50$ and the noise variance is
0.5.  The linear model, $\hat{\beta}_{\wS}$ is selected on $\mathcal{D}_1$ by lasso with $\lambda$ chosen
using 10-fold cross-validation.  For $\gamma_{\wS}(j)$, $\hat{\beta}_{\wS}(j)$ is
estimated by reapplying the same selection procedure to $\mathcal{D}_1$ with
the $j^{\mathrm{th}}$ variable removed.  Confidence intervals are constructed
using the pairs bootstrap procedure of Section 2 with $\alpha=0.05$.

\begin{figure}
\begin{center}
\begin{tabular}{cc}
\includegraphics[width=0.5\linewidth]{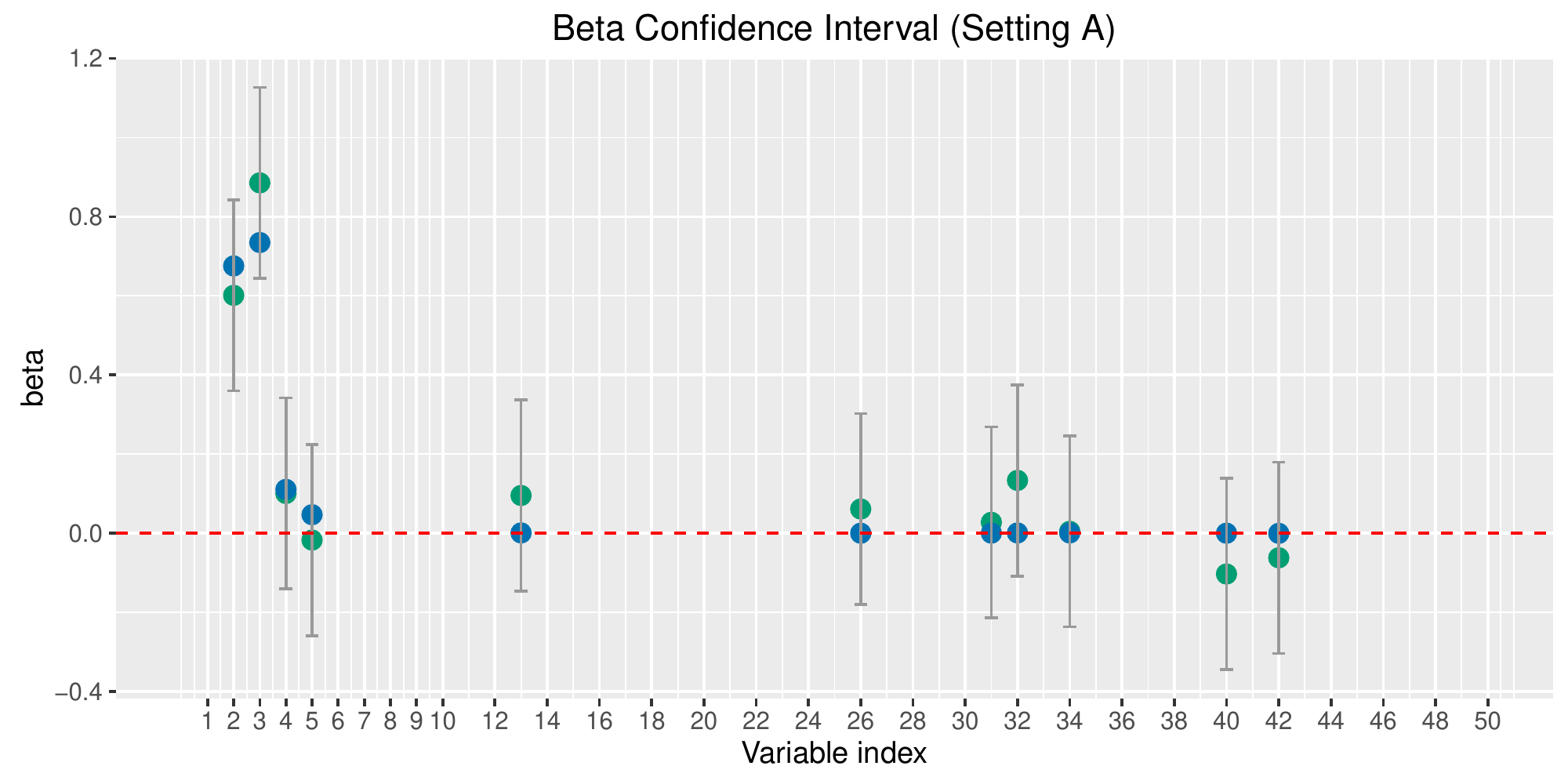} &
\includegraphics[width=0.5\linewidth]{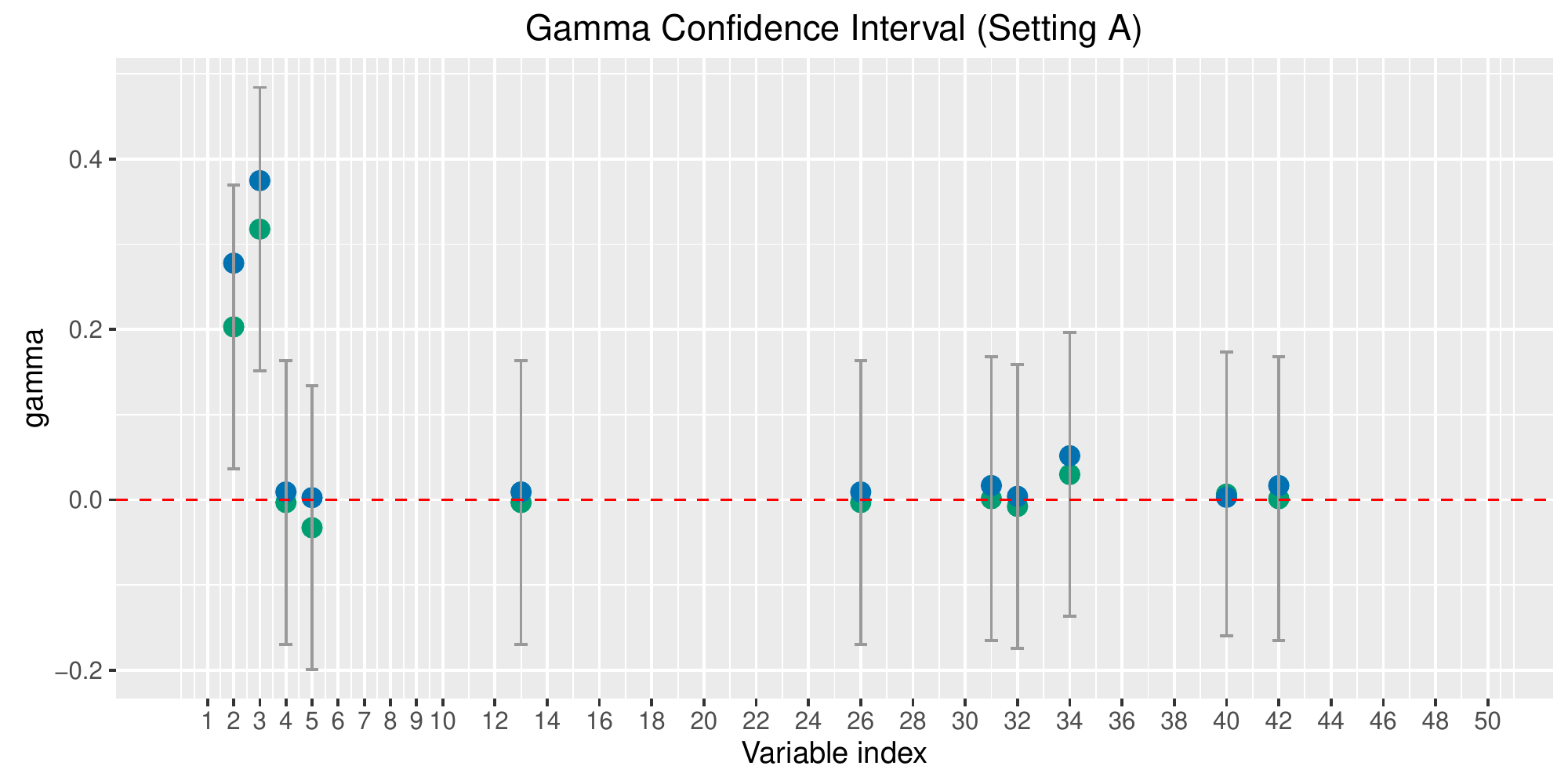}\\
\includegraphics[width=0.5\linewidth]{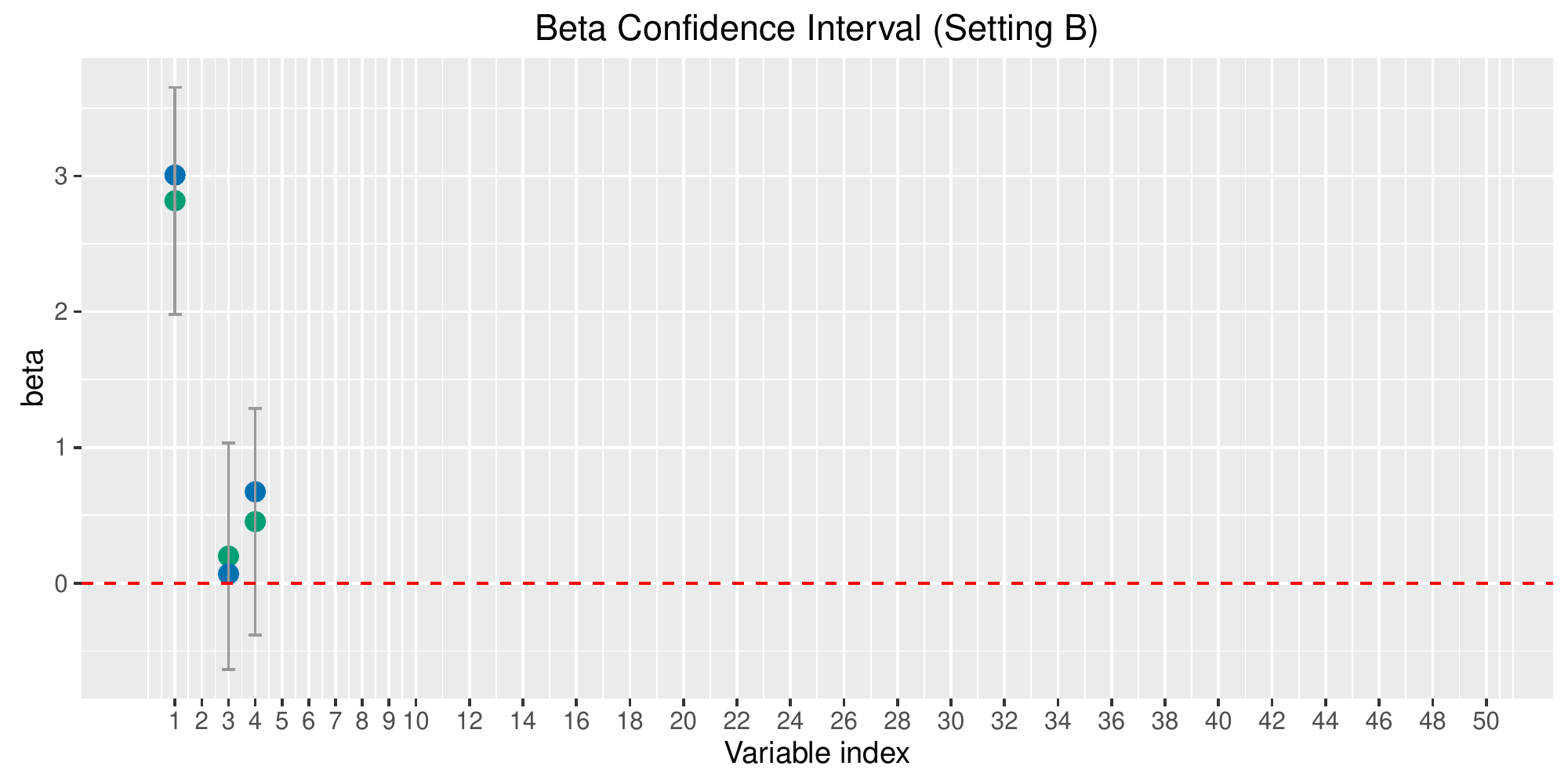} &
\includegraphics[width=0.5\linewidth]{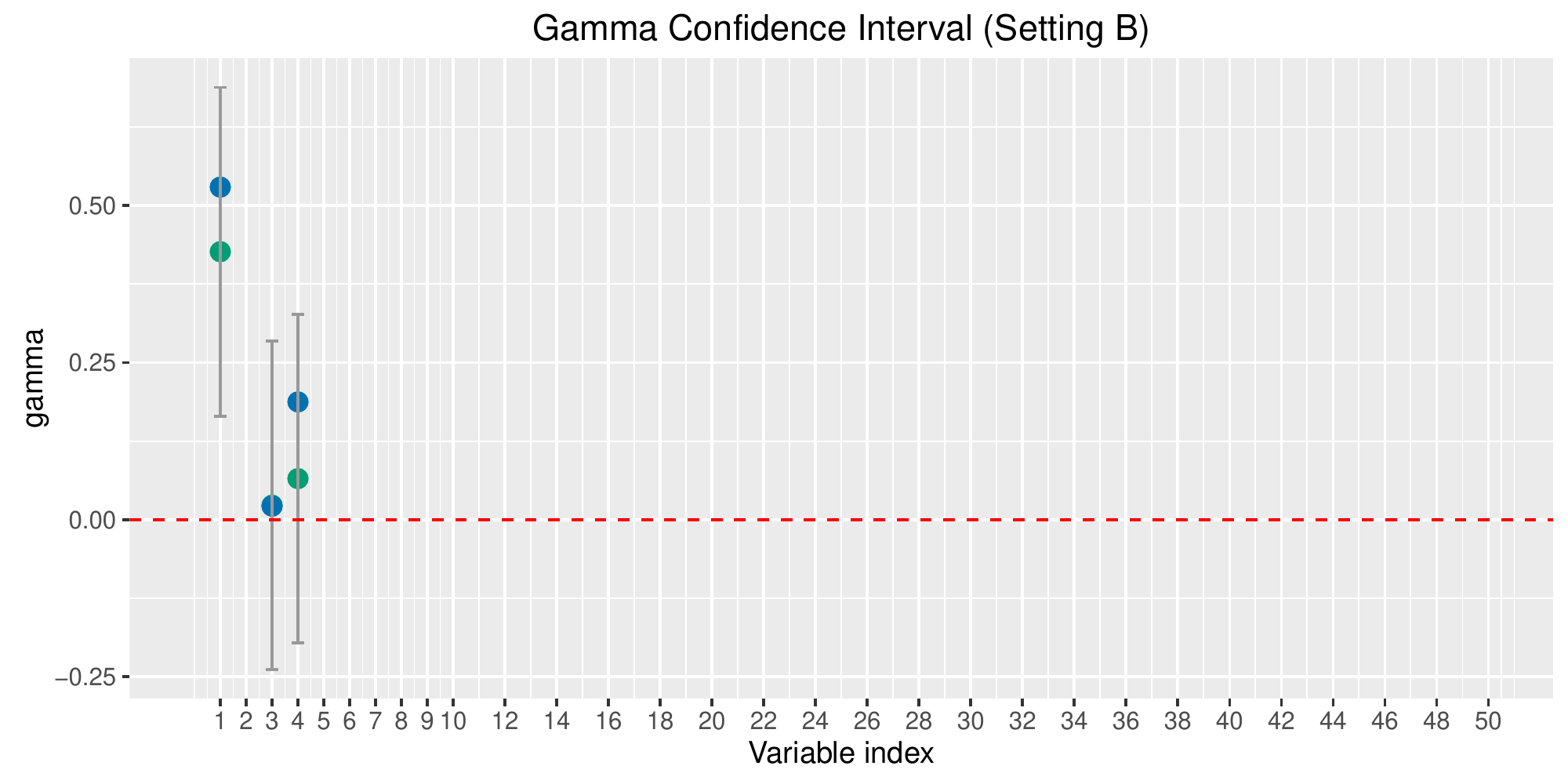}\\
\includegraphics[width=0.5\linewidth]{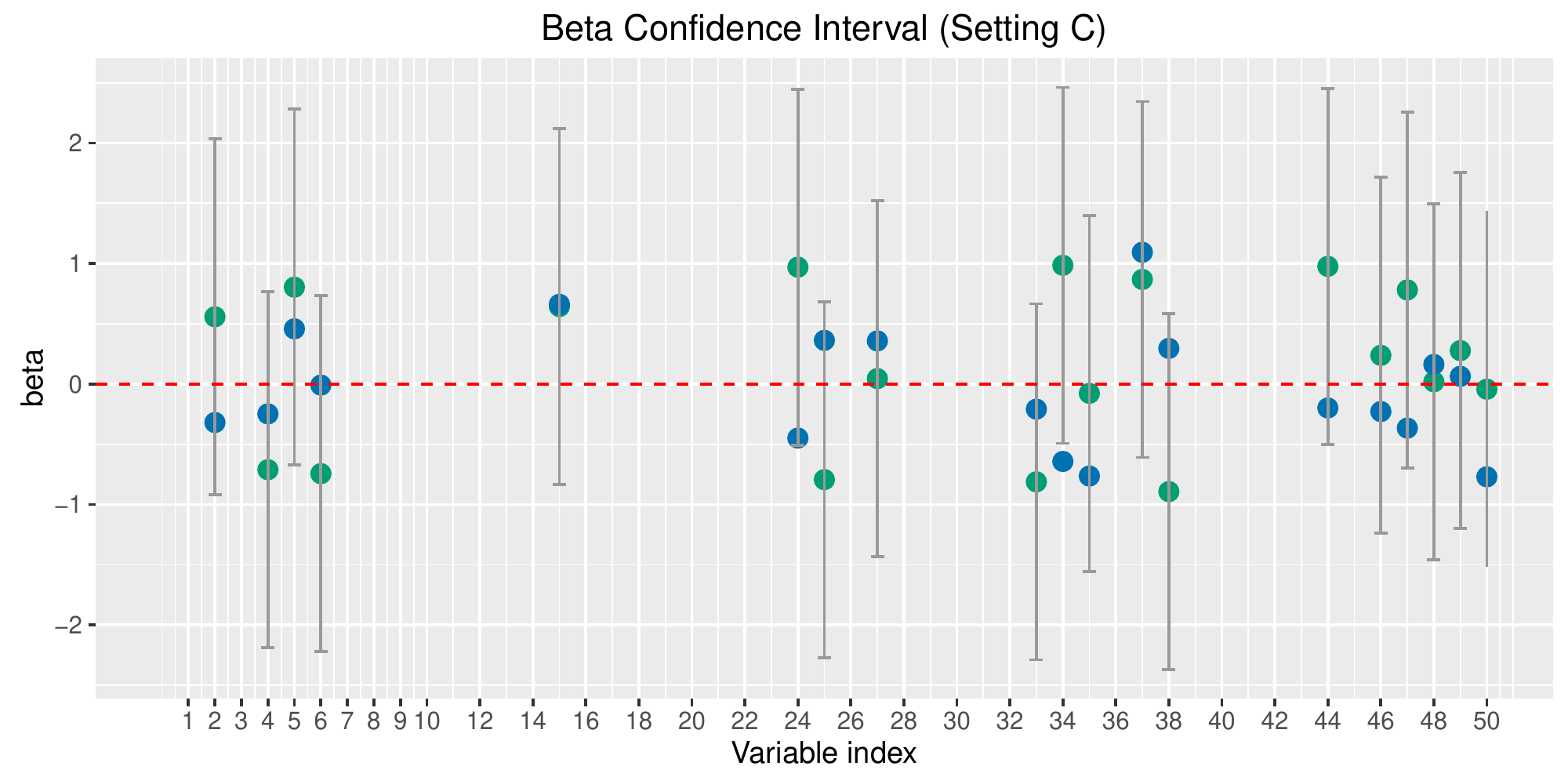} &
\includegraphics[width=0.5\linewidth]{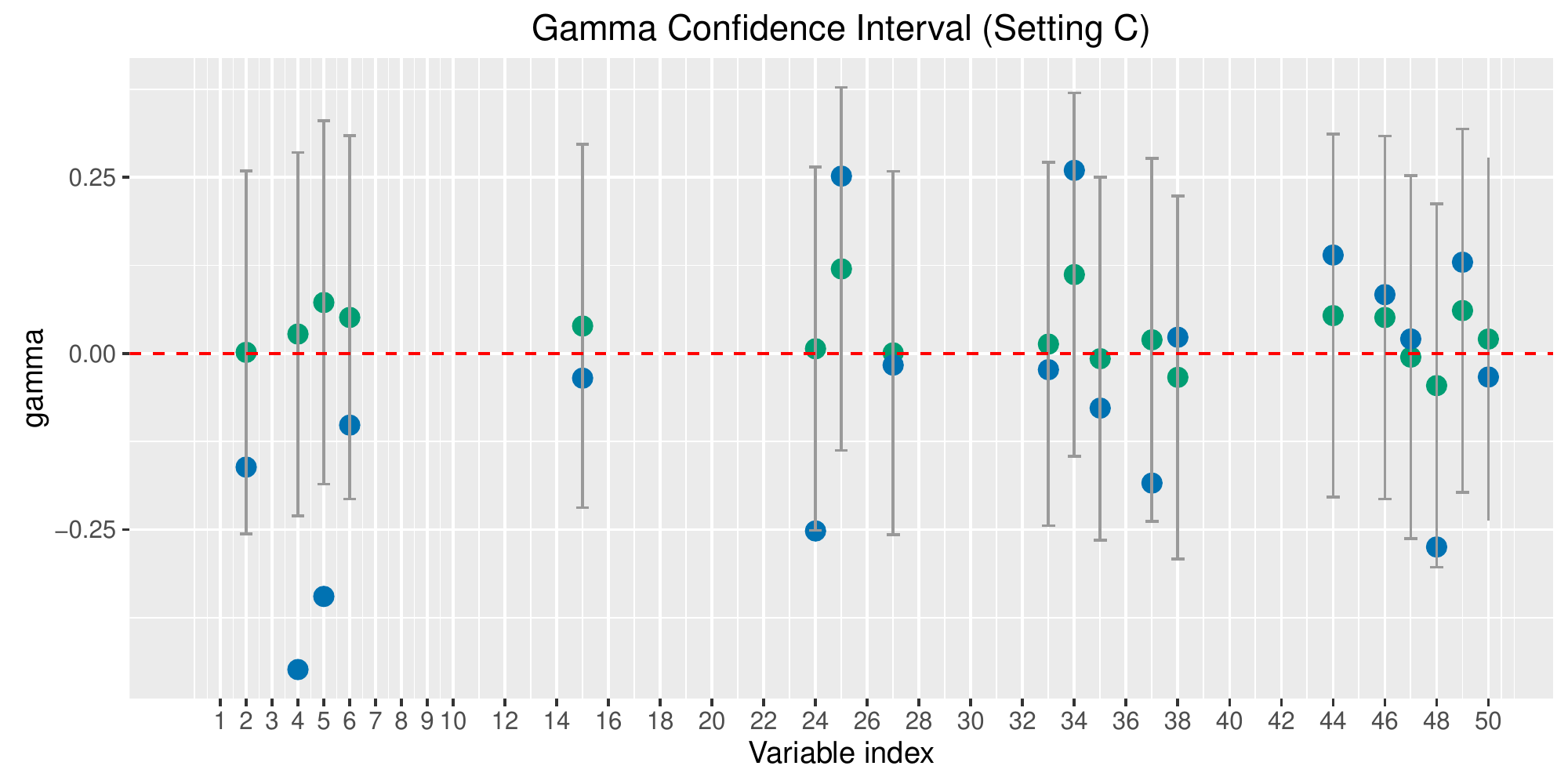}\\
\end{tabular}
\end{center}
\caption{\emph{Typical confidence intervals for the projection parameter (left) and
the LOCO parameter (right) for Settings A, B, and C.  Blue indicates the true
parameter value, and green indicates the point estimate from $\mathcal{D}_2$.  Note that the parameters
are successfully covered even when the underlying signal is non-linear ($X_1$
in Setting B) or dense (Setting C).}}
\label{fig::confint}
\end{figure}

Figure \ref{fig::confint} shows typical confidence intervals for the projection parameter,
$\beta_{\wS}$, and the LOCO parameter, $\gamma_{\wS}$, for one
realization of each Setting.  Notice that confidence intervals are only
constructed for $j\in \wS$.  
The non-linear term is successfully covered
in Setting B, even though the linear model is wrong.

\begin{figure}
\begin{center}
\includegraphics[scale=.5]{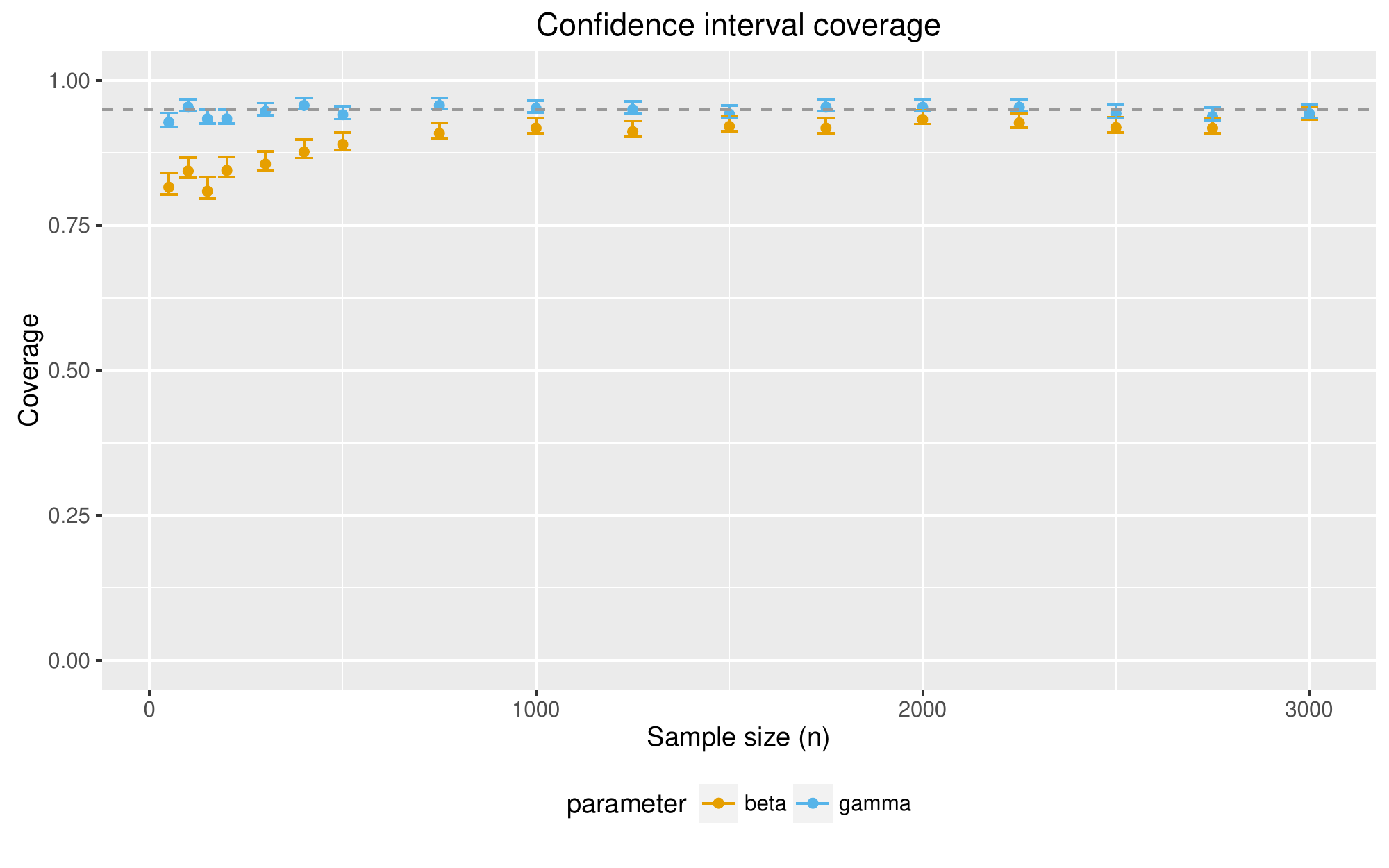} 
\end{center}
\caption{\emph{Joint coverage probability of the intervals for $\beta_{\wS}$ and
  $\gamma_{\wS}$ in Setting B,  as sample
size $n$ varies with $p=50$ held fixed.  The coverage for
$\gamma_{\wS}$ is accurate even at low sample sizes, while the coverage for
$\beta_{\wS}$ converges more slowly.  }}
\label{fig::coverage}
\end{figure}

Figure \ref{fig::coverage} shows the coverage probability for Setting B as a function of $n$,
holding $p=50$ fixed.  The coverage for the LOCO parameter, $\gamma_{\wS}$ is
accurate even at low sample sizes.  The coverage for $\beta_{\wS}$ is low (0.8-0.9)
for small sample sizes, but converges to the correct coverage as the sample
size grows.  This suggests that $\beta_{\wS}$ is an
easier parameter to estimate and conduct inference on.

\section{Berry-Esseen Bounds for Nonlinear Parameters With Increasing Dimension}
\label{section::berry}

The results in this paper
depend on a Berry-Esseen bound for regression
with possibly increasing dimension.
In this section, there is no model selection or splitting.
We set $d=k$ and $S = \{1,\ldots, k\}$
where $k < n$ and $k$ can increase with $n$.
Later, these results will be applied after model selection and sample splitting.
Existing Berry-Esseen results for nonlinear parameters
are given in
\cite{pinelis2009berry, shao2016stein,chen2007normal,
anastasiou2014bounds,anastasiou2015new,anastasiou2016multivariate}.
Our results are in the same spirit
but we keep careful track of the effect of dimension
and the eigenvalues of $\Sigma$,
while leveraging 
results from \cite{cherno1,cherno2}
on high dimensional
central limit theorems for simple convex sets.

We derive a general result
on the accuracy of the Normal approximation
over hyper-rectangles for nonlinear parameters.
We make use of three findings from
\cite{cherno2, chernozhukov2015comparison} and \cite{nazarov1807maximal}:
the Gaussian anti-concentration theorem, the high-dimensional central limit theorem for sparely convex sets,
and the Gaussian comparison theorem,reported in the appendix as Theorems \ref{thm:anti.concentration}, \ref{thm:high.dim.clt} and \ref{thm:comparisons}, respectively.
In fact, in the appendix we re-state these results in a slightly different form
than they appear in the original papers.
We do this because we need to keep track of certain constants that affect our results.

Let $W_1,\ldots, W_n$ be an independent sample from a distribution $P$ on $\mathbb{R}^b$ belonging to the class ${\cal P}_n$ of probability distribution supported on a subset of
$[-A,A]^b$, for some fixed $A>0$ and such that 
\[
    v = \inf_{ P  \in      \mathcal{P}_n} \lambda_{\min}(V(P))) \quad
\text{and} \quad \overline{v} = \sup_{ P  \in      \mathcal{P}_n}
\lambda_{\max}(V(P))) \geq 1,
\]
where $V(P) = \mathbb{E}_P[ (W_i-\psi)(W_i-\psi)^\top]$. We allow the class $\mathcal{P}_n$
to change with $n$, so that  $b$, $v$ and $\overline{v}$ -- but not $A$ -- are to be regarded as functions of $n$,
although we do not express such dependence in our notation for ease of
readability. Notice that, in the worse case, $\overline{v}$ can be of order $b$. 

\noindent {\bf Remark.} The assumption that $\overline{v} \geq 1$ is made out of convenience and is used in the proof of \Cref{lem:hyper} in the Appendix. Our results remain valid if we assume that $\overline{v}$ is bounded away from $0$ uniformly in $n$, i.e. that $\overline{v}  \geq \eta$ for some $\eta > 0$ and all $n$. Then, the term $\eta$ would then appear as another quantity affecting the bounds. We have not kept track of this additional dependence.

Let
$g = (g_1,\ldots,g_s)^\top \colon \mathbb{R}^b  \rightarrow \mathbb{R}^s$ be a
twice-continuously differentiable 
vector-valued function defined over an open, convex subset $\mathcal{S}_n$ of $[-A,A]^b$
such  that, for all $P \in
\mathcal{P}_n$, 
 $\psi =  \psi(P) = \mathbb{E}[W_1] \in \mathcal{S}_n$. Let $\widehat{\psi} = \widehat{\psi}(P) = \frac{1}{n} \sum_{i=1}^n W_i$ and assume
that $\widehat{\psi} \in \mathcal{S}_n$ almost surely, for all $P \in
\mathcal{P}_n$.
Finally, set  
$\theta = g(\psi)$ and $\widehat{\theta} = g(\widehat{\psi})$.
For any point $\psi \in \mathcal{S}_n$ and $j\in \{ 1,\ldots,s\}$, we will write 
$G_j(\psi) \in \mathbb{R}^b$ and $H_j(\psi)\in \mathbb{R}^{b \times b}$ for the gradient  
and Hessian of $g_j$ at $\psi$, respectively. We will set $ G(\psi)$ to 
be the $s\times b$ Jacobian matrix whose $j^{\rm th}$ row is $G^\top_j(\psi)$.

{\bf Remark} The assumption that $\hat{\psi}$ belongs to $\mathcal{S}_n$ almost surely can be relaxed to hold on an event of high probability, resulting in an additional error term in all our bounds. 

To derive a high-dimensional Berry-Esseen bound on $g(\psi)  - g(\hat{\psi})$
we will study  its first order Taylor approximation. Towards that end, we will
require a uniform control over the size of the gradient and Hessian of $g$. Thus we set
\begin{equation}\label{eq:H.and.B}
    B  = \sup_{P \in \mathcal{P}_n  }\max_{j=1,\ldots,s} ||G_j(\psi(P))|| 
\quad \text{and} \quad \overline{H}  = 
\sup_{ \psi\in \mathcal{S}_n }\max_{j=1,\ldots,s} 
\|H_j(\psi)\|_{\mathrm{op}} 
\end{equation}
where
$\|H_j(\psi)\|_{\mathrm{op}}$ is the operator norm.

{\bf Remark.} The quantity $\overline{H}$ can be defined differently, as a function of $\mathcal{P}_n$ and not $\mathcal{S}_n$. In fact,  all  that is required of $\overline{H}$ is that it satisfy the almost everywhere bound
\begin{equation}\label{eq:H.2}
\max_j \int_0^1 \left\|  H_j \left( t\psi(P) - (1-t)\hat{\psi}(P) \right)  \right\|_{\mathrm{op}} dt \leq \overline{H},
\end{equation}
for each $P \in \mathcal{P}_n$ (see \eqref{eq:bound.H} below). This  allows us to establish a uniform bound on the magnitude of the reminder term in the Taylor series expansion of $g(\hat{\psi})$ around $g(\psi)$, as detailed in the proof of \Cref{theorem::deltamethod} below. Of course, we may relax the requirement that \eqref{eq:H.2} holds almost everywhere to the  requirement that it holds on an event of high probability. This is indeed the strategy we use when in applying the present results to the projection parameters in \Cref{sec:projection}.

The covariance matrix of the linear approximation of  $g(\psi)  -
g(\hat{\psi})$, which, 
for any $P \in
\mathcal{P}_n$, is given by
\begin{equation}\label{eq:Gamma}
 \Gamma  = \Gamma(\psi(P),P)=G(\psi(P)) V(P) G(\psi(P))^\top,
\end{equation}
plays a crucial role in our analysis. In particular, our results will depend on
the smallest variance of the linear approximation to $g(\psi)  - g(\hat{\psi})$: 
\begin{equation}\label{eq:sigma}
    \underline{\sigma}^2  =  \inf_{ P  \in \mathcal{P}_n}\min_{j =1,\ldots,s}
    G^\top_j(\psi(P)) V(P) G_j(\psi(P)).
\end{equation}

With these definitions in place we are now ready to prove the following
high-dimensional Berry-Esseen bound.

\begin{theorem}
\label{theorem::deltamethod}
Assume that 
$W_1,\ldots, W_n$ is an i.i.d. sample from some  $P \in {\cal P}_n$ and let $Z_n \sim N(0,\Gamma)$.
Then, there exists a $C>0$, dependent on $A$ only, such that
\begin{equation}
\sup_{P\in {\cal P}_n}
\sup_{t > 0}
\Bigl|\mathbb{P}( \sqrt{n}||\hat\theta - \theta||_\infty \leq t) -\mathbb{P}(
||Z_n||_\infty \leq t)\Bigr| \leq C \Big(
\Delta_{n,1} + \Delta_{n,2} \Big),
\end{equation}
where
\begin{align}
\label{eq::Delta}
\Delta_{n,1} &=
\frac{1}{\sqrt{v}} \left( \frac{ \overline{v}^2 b (\log 2bn)^7}{n}   \right)^{1/6}
\\
\Delta_{n,2} &=
\frac{1}{\underline{\sigma}}\sqrt{\frac{ b \overline{v} \overline{H}^2 (\log n)^2 \log b}{n}}.
\end{align}
\end{theorem}

{\bf Remarks}
The assumption that the support of $P$ is compact is made for simplicity,
and can be modified by assuming that the
coordinates of the vectors
$W_i$ have sub-exponential behavior.
Notice also that the coordinates of the $W_i$'s need not be independent.

The proof of \Cref{theorem::deltamethod} resembles the classic proof of the asymptotic
normality of non-linear functions of averages by the delta method. First, we
carry out a coordinate-wise Taylor
expansion of $\widehat{\theta}$ around $\theta$. We then utilize a a high-dimensional Berry-Esseen theorem for
polyhedral sets established in \cite{cherno2} (see \Cref{lem:hyper} below for
details) to derive a Gaussian approximation to the linear part in the expansion, resulting in the error term
$\Delta_{n,1}$. Finally, we bound the reminder term due to the non-linearity of the
function $g$ with basic concentration
arguments paired with the Gaussian anti-concentration
bound due to \cite{nazarov1807maximal} (see \Cref{thm:anti.concentration} in the Appendix), thus obtaining the second error term $\Delta_{n,2}$. Throughout, we keep track of the dependence on $v$ and $\underline{\sigma}$ in order to
obtain rates with a leading constant dependent only on $A$ (assumed fixed) but not on any other
term that may vary with $k$ or $b$.

\subsubsection*{Asymptotically honest confidence sets: Normal approximation approach}
\label{sec:berry.normal}

We now show how to use the high-dimensional central limit theorem
\Cref{theorem::deltamethod} to construct asymptotically honest confidence sets for $\theta$. 
We will first to obtain a consistent estimator of the covariance matrix
$\Gamma=G(\psi) V(\psi) G(\psi)^\top$ of the linear approximation to
$\hat{\theta} - \theta$.
In conventional fixed-dimension asymptotics,
we would appeal to Slutzky's theorem and ignore the effect of replacing
$\Gamma$ with a consistent estimate.
But in computing Berry-Esseen bounds with increasing dimension we may not
discard 
the effect of estimating $\Gamma$. As we will see below, this extra step will
bring an additional error term that must be accounted for.
We will estimate $\Gamma$ with the plug-in estimator 
\begin{equation}\label{eq:hat.gamma.berry}
    \hat\Gamma = G(\hat\psi) \hat V G(\hat\psi)^\top,
\end{equation}
where $\hat{V} = \frac{1}{n} \sum_{i=1}^n W_i W_i^\top - \hat{\psi}
\hat{\psi}^\top$ 
is the empirical covariance matrix.
Below, we  bound the element-wise
difference between
$\Gamma$ and $\hat{\Gamma}$.  Although this is in general a fairly weak notion
of consistency in covariance matrix estimation, it is all that is needed to
apply the Gaussian comparison theorem  \ref{thm:comparisons}, which will allow
us to extend the Berry-Esseen bound established in \Cref{theorem::deltamethod}
to the case when $\Gamma$ is  estimated.

\begin{lemma}
\label{lemma::upsilon}
Let
\begin{equation}\label{eq:aleph}
\aleph_n = \max \Big\{ \overline{H} B \overline{v} \sqrt{ b\frac{   
 \log n}{n}},  B^2 
 \sqrt{   b \overline{v}  \frac{\log b + \log n }{n}   }\Big\}.
 \end{equation}
There exists a $C > 0$ dependent on $A$ only such  that
\begin{equation}\label{eq:upsilon}
\sup_{P\in {\cal P}_n}
\mathbb{P}\left(\max_{j,l} \left| \hat\Gamma(j,l)-\Gamma(j,l)\right| \geq C
\,\aleph_n\right) \leq \frac{2}{n}.
\end{equation}
\end{lemma}

Now we construct the confidence set.
Let $Q=(Q(1),\ldots, Q(s))$ be i.i.d. standard Normal variables,
independent of the data. 
Let $\hat Z = \hat\Gamma^{1/2} Q$
and define $\hat{t}_\alpha$ by
\begin{equation}
    \label{eq:hat.t.alpha.berry}
\mathbb{P}( ||\hat Z||_\infty > \hat{t}_\alpha \,|\, \hat\Gamma)=\alpha.
\end{equation}
Finally, let
\begin{equation}\label{eq::conf-rectangle}
    \hat{C}_n = \Bigl\{ \theta \in \mathbb{R}^s:\ ||\theta-\hat\theta||_\infty \leq \frac{\hat{t}_\alpha}{\sqrt{n}}\Bigr\}.
\end{equation}

\begin{theorem}
\label{thm::coverage}
There exists a $C>0$, dependent only on $A$, such that 
\begin{equation}
\inf_{P\in {\cal P}}\mathbb{P}(\theta\in R_n)= 1-\alpha -
C \left( \Delta_{n,1} + \Delta_{n,2} + \Delta_{n,3} + \frac{1}{n} \right),
\end{equation}
where
\begin{equation}\label{eq::this-is-upsilon}
\Delta_{n,3}=  \frac{\aleph_n^{1/3} (2 \log 2s)^{2/3}}{\underline{\sigma}^{2/3}}.
\end{equation}
\end{theorem}

{\bf Remark}. The additional term $\Delta_{n,3}$ in the previous theorem is due
to the uncertainty in estimating $\Gamma$, and can be established by using the comparison inequality for
Gaussian vectors of \cite{chernozhukov2015comparison}, keeping track of
the dependence on $\underline{\sigma}^2$; see \Cref{thm:comparisons} below.

In addition to $L_\infty$ balls, we can also construct our confidence set to be a
hyper-rectangle, with side lengths proportional to the standard errors of
the projection parameters. That is, we define 
\begin{equation}\label{eq:hyper:CI}
\tilde C_n = \bigotimes_{j\in S} C(j),
\end{equation}
where
\[
    C(j) = \left[    \hat\beta_S(j) - z_{\alpha/(2s)} \sqrt{
    \frac{\hat\Gamma_n(j,j)}{n} }, \hat\beta_S(j) + z_{\alpha/(2s)} \sqrt{
    \frac{\hat\Gamma_n(j,j)}{n} }\right],
\]
with $\hat\Gamma$ given by (\ref{eq::Ga}) and $z_{\alpha/(2s)}$ the upper $1 -
\alpha/(2s)$ quantile of a standard normal variate. Notice that we use a Bonferroni
correction to guarantee a nominal coverage of $1-\alpha$.
Also, note that
$z_{\alpha/(2s)} = O(\sqrt{\log s})$, for each fixed $\alpha$. The coverage rate
for this other confidence set is derived in the next result.

\vspace{11pt}

\begin{theorem}
\label{thm::bonf}
Let 
\begin{equation}\label{eq.Delta3.tilde}
\tilde{\Delta}_{n,3} = \min\left\{\Delta_{3,n}, \frac{ \aleph_n
    z_{\alpha/(2s)}}{\underline{\sigma}^2 } \left(\sqrt{ 2 + \log(2s ) } + 2
    \right) \right\}. 
\end{equation}
There exists a $C>0$, dependent only on $A$, such that 
$$
\inf_{P \in \mathcal{P}_n} \mathbb{P}(\theta \in \tilde C_n) \geq 
(1-\alpha)  - C \Big(  \Delta_{n,1} +
\Delta_{n,2} + \tilde \Delta_{n,3} + \frac{1}{n} \Big).
$$
\end{theorem}

\subsubsection*{Asymptotically honest confidence sets: the bootstrap approach}

To construct the confidence set \Cref{eq::conf-rectangle}, one has to
compute the estimator $\hat{\Gamma}$ 
and the quantile $\hat{t}_\alpha$ in
\eqref{eq:hat.t.alpha.berry},
which may be computationally inconvenient. 
Similarly, the  hyper-rectangle \Cref{eq:hyper:CI} requires computing the diagonal entries in
$\hat{\Gamma}$. 

Below we rely on the bootstrap 
to construct analogous confidence sets,  centered at
$\hat{\theta}$, which do not need knowledge of $\hat{\Gamma}$.
We let $\hat{\psi}^*$ denote the sample average of an i.i.d. sample of size
$n$ from the bootstrap distribution, which is the empirical measure associated
to the sample $(W_1,\ldots,W_n)$. We also let $\hat{\theta}^* =
g(\hat{\psi}^*)$.

For a fixed $\alpha \in (0,1)$, let $\hat{t}^*_\alpha$ be the smallest
    positive number such that 
    \[
	\mathbb{P}\left( \sqrt{n} \| \hat{\theta}^* - \hat{\theta}\|
	\leq \hat{t}^*_\alpha \Big| (W_1,\ldots,W_n) \right) \geq 1 - \alpha.
    \]
    and let $(\tilde{t}^*_j, j =1,\ldots,s)$ be such that
    \[
	\mathbb{P}\left( \sqrt{n} | \hat{\theta}^*(j) - \hat{\theta}
	(j) \leq \tilde{t}^*_j, \forall j \Big| (W_1,\ldots,W_n) \right) \geq 1 - \alpha.
    \]
By the union bound, each $\tilde{t}^*_j$ can be chosen to
    be the largest positive number such that
    \[
	\mathbb{P}\left( \sqrt{n} | \hat{\theta}^*(j) - \hat{\beta}
	(j) > \tilde{t}^*_j, \Big|  (W_1,\ldots,W_n) \right)  \leq
	\frac{\alpha}{s}.
    \]
Consider the following two bootstrap confidence sets: 
\begin{equation}\label{eq:ci.boot.theta}
	    \hat{C}^*_{n} = \left\{ \theta \in \mathbb{R}^{s} \colon \|
	    \theta - \hat{\theta}
	    \|_\infty \leq \frac{ \hat{t}^*_{\alpha}}{\sqrt{n}} \right\} \quad
	    \text{and} \quad 
	    \tilde{C}^*_{n} = \left\{ \theta \in \mathbb{R}^{s} \colon |
	    \theta(j) - \hat{\theta}(j)
	    | \leq \frac{ \tilde{t}^*_{j}}{\sqrt{n}}, \forall j \in \wS \right\}
	\end{equation}

\begin{theorem}
\label{theorem::boot}
Assume the same conditions of Theorem \ref{theorem::deltamethod}
and that and $\hat{\psi}$ and $\hat{\psi}^*$ belong to  $\mathcal{S}_n$ almost surely.
Suppose that  $n$ is large enough so that the quantities 
$\underline{\sigma}^2_n = \underline{\sigma}^2 - C \aleph_n >0$ and  $v_n  = v - C \daleth_n$ are positive, 
where $C$ is the 
larger of the two constants  in \eqref{eq:upsilon} and in
\eqref{eq:matrix.bernstein.simple.2} and
\[
\daleth_n = 
\sqrt{  b \overline{v} \frac{ \log b + \log n }{n}   }.
\]
Also set $\overline{v}_n = \overline{v} + C \daleth_n$. Then, for a constant  $C$ depending only on $A$,
\begin{equation}\label{eq::boot-cov}
    \inf_{P\in {\cal P}_n}\mathbb{P}(\theta\in \hat{C}^*_n) \geq 1-\alpha -
C\left(\Delta^*_{n,1} + \Delta^*_{n,2} + \Delta_{n,3} +
\frac{1}{n}\right),
\end{equation}
where
\[
\Delta^*_{n,1} =
\frac{1}{\sqrt{v_n}} \left( \frac{ \overline{v}_n b (\log 2bn)^7}{n}   \right)^{1/6}
 ,\quad 
\Delta^*_{n,2} =
\frac{1}{\underline{\sigma}_n}\sqrt{\frac{ b \overline{v}_n  \overline{H}^2 (\log n)^2
\log b}{n}},
\]
 and $\Delta_{n,3}$ is given in (\ref{eq::this-is-upsilon}).
Similarly,
\begin{equation}\label{eq::boot-cov.bonf}
    \inf_{P\in {\cal P}_n}\mathbb{P}(\theta\in \tilde{C}^*_n) \geq 1-\alpha -
C\left(\Delta^*_{n,1} + \Delta^*_{n,2} + \Delta_{n,3} +
\frac{1}{n}\right).
\end{equation}
\end{theorem}

{\bf Remark.} The assumption that $\hat{\psi}$ and $\hat{\psi}^*$ are in $\mathcal{S}_n$
almost surely can be relaxed to a high probability statement without 
 any issue, resulting in an additional bound on the probability of the complementary event.

{\bf Remark.} The proof of the theorem involves enlarging the class of
distributions 
$\mathcal{P}_n$ to a bigger collection $\mathcal{P}^*_n$ that is guaranteed to include the
bootstrap distribution (almost surely or with high probability). The resulting coverage error terms are
larger than the ones obtained in  \Cref{thm::coverage} using Normal approximations precisely
because $\mathcal{P}_n^*$ is a larger class. In the above result we simply increase the
rates arising from \Cref{thm::coverage} so that they hold for $\mathcal{P}^*_n$ without actually
recomputing the quantities $B$,
$\overline{H}$ and
$\underline{\sigma}^2$ in \eqref{eq:H.and.B} and \eqref{eq:sigma} over the new
class $\mathcal{P}^*_n$. 
Of course, better rates may be established should sharper bounds on those
quantities be available. 

{\bf Remark.} The error term $\Delta_{n,3}$ remains the same as in \Cref{thm::coverage} and \Cref{thm::bonf} because it
quantifies an error term, related to the Gaussian comparison \Cref{thm:comparisons}, which does not depend on the bootstrap
distribution.

\section{Conclusions}
\label{section::conclusion}

In this paper we have taken a modern look at 
inference based on sample splitting.
We have also investigated the accuracy of Normal and bootstrap approximations
and we have suggested new parameters for regression.

Despite the fact that sample splitting is on old idea, there remain many open
questions.
For example, in this paper, we focused on a single split of the data.
One could split the data many times and somehow combine the confidence sets.
However, for each split we are essentially estimating a different
(random) parameter.
So currently, it is nor clear how to combine this information.

The bounds on coverage accuracy --- which are of interest beyond sample splitting ---
are upper bounds.
An important open question is to find lower bounds.
Also, it is an open question
whether we can improve the bootstrap rates.
For example,
the remainder term in the Taylor approximation of
$\sqrt{n}(\hat\beta(j) - \beta(j))$ is
$$
\frac{1}{2n}\int
\int \delta^\top  H_j((1-t)\psi + t \hat\psi) \delta\, dt
$$
where
$\delta=\sqrt{n}(\hat\psi - \psi)$.
By approximating this quadratic term it might be possible to
correct the bootstrap distribution.
\cite{pouzo2015bootstrap} has results for bootstrapping quadratic forms that could be useful here.
In Section 
\ref{section::improving}
we saw that a modified bootstrap, that we called the image bootstrap,
has very good coverage accuracy even in high dimensions.
Future work is needed to compute the resulting confidence set efficiently.

Finally, we remind the reader that
we have taken a assumption-free
perspective.
If there are reasons to believe in some parametric model then of
course the
distribution-free, sample splitting approach used in this paper will be sub-optimal.

\section*{Acknowledgements}

The authors are grateful to the AE and the reviewers
for comments that led to substantial improvements
on the paper and the discovery of a mistake in the original version of the manuscript.
We also thank Lukas Steinberger, Peter Buhlmann and Iosif Pinelis for helpful suggestions and Ryan Tibshirani for comments on early drafts.

\bibliographystyle{plainnat}
\bibliography{paper}

\section{Appendix 1: Improving the Coverage Accuracy of the Bootstrap for the Projection Parameters}
\label{section::improving}

Throughout, we treat $S$ as a fixed, non-empty subset of $\{1,\ldots,d\}$ of size $k$ and assume an i.i.d. sample $(Z_1,\ldots,Z_n$) where  $Z_i = (X_i,Y_i)$ for all $i$, from a distribution from $\mathcal{P}_n^{\mathrm{OLS}}$.

The coverage accuracy for LOCO and prediction parameters is much higher than for the projection parameters and
 the inferences for $\beta_S$ are less accurate if $k$ is allowed to increase with $n$.
Of course, one way to ensure accurate inferences is simply
to focus on $\gamma_S$ or $\phi_S$ instead of $\beta_S$.
Here we discuss some other approaches to ensure coverage accuracy.

If we use ridge regression instead of least squares,
the gradient and Hessian with respect to $\beta$ are bounded and the error terms
are very small.
However, this could degrade prediction accuracy.
This leads to a tradeoff between inferential accuracy and prediction accuracy.
Investigating this tradeoff will be left to future work.

Some authors have suggested the estimator
$\hat\beta_S = \tilde\Sigma_S^{-1} \hat\alpha_S$
where $\tilde\Sigma_S$ is a block diagonal estimator of $\Sigma$.
If we restrict the block size to be bounded above by a constant,
then we get back the accuracy of the sparse regime.
Again there is a 
tradeoff between inferential accuracy and prediction accuracy.

The accuracy of the bootstrap can be increased by using
the {\em image bootstrap} as we now describe.
First we apply the bootstrap to get a confidence set for $\psi_S$.
Let
\[
H_n = \Biggl\{ \psi_S:\
||\psi_S - \hat\psi_S||_\infty \leq \frac{t^*_\alpha}{\sqrt{n}} \Biggr\}
\]
where
$t^*_\alpha$ is the bootstrap quantile defined by
$\hat F^*(t^*_\alpha) = 1-\alpha$ and
\[
\hat F^*(t) = P(\sqrt{n}||\hat\psi^*_S - \hat\psi_S||_\infty \leq t\,|\, Z_1,\ldots, Z_n).
\]
Since $\psi_S$ is just a vector of moments,
it follows from \Cref{thm:high.dim.clt} Theorem K.1 of
\cite{cherno1} and the Gaussian anti-concentration (\Cref{thm:anti.concentration}) that, for a constant $C$ depending on $A$ only, 
\begin{equation}\label{eq:image}
\sup_{P\in {\cal P}^{\mathrm{OLS}}_n}|\mathbb{P}(\psi \in H_n) - (1-\alpha)| \leq \frac{C}{a_n}\left( \frac{(\log k)^7}{n}\right)^{1/6}.
\end{equation}
In the above display $a_n = \sqrt{a - C \sqrt{ \frac{\log k }{n}}}$ and is positive for $n$ large enough,  and 
\[
 a \leq \inf_{P \in \mathcal{P}_n^{\mathrm{OLS}}} \min_{j \in \{1,\ldots,d\}} {\rm Var}_P(W_i(j)).
 \] Notice that $a$ is positive  since $a \geq v$, where $v$ is given the definition  \ref{def:Pdagger} of $\mathcal{P}_n^{\mathrm{OLS}}$. However,  $a$ can be significantly larger that $v$. The term $C \sqrt{\log k  }{n}$ appearing in the definition of $a_n$ is just a high probability bound on the maximal element-wise difference between $V$ and $\hat{V}$, valid for each $ P \in \mathcal{P}_n^{\mathrm{OLS}}$.

Next, recall that $\beta_S = g(\psi_S)$.
Now define
\begin{equation}
C_n = \Biggl\{ g(\psi) :\ \psi \in H_n \Biggr\}.
\end{equation}
We call $C_n$ the {\em image bootstrap confidence set}
as it is just the nonlinear function $g$ applied to the confidence set $H_n$.
Then, by \eqref{eq:image}, 
\[
\inf_{P\in {\cal P}_n'}\mathbb{P}(\beta \in C_n) \geq 1-\alpha - \frac{C}{a_n}\left( \frac{\log k}{n}\right)^{1/6}.
\]
In particular, the implied confidence set for $\beta(j)$ is
$$
C_j = \Biggl[\inf_{\psi \in H_n}g(\psi),\ \sup_{\psi \in H_n}g(\psi)\Biggr].
$$

Remarkably, in the coverage accuracy of the image-bootstrap the dimension $k$ enters only logarithmically. 
This is in stark contrast with the coverage accuracy guarantees for the projection parameters from \Cref{sec:projection}, which depend polynomially in $k$ and on the other eigenvalue parameters. 

The image bootstrap is usually avoided because it generally leads to
conservative confidence sets.
Below we derive bounds on the accuracy of the image bootstrap.

\begin{theorem}\label{thm:beta.accuracy}
Let $u_n$ be as in \eqref{eq:un} and assume that $k \geq u_n^2$. Then, for each $P \in \mathcal{P}_n^{\mathrm{OLS}}$, with probability at least $\frac{1}{n}$, the diameter of the image bootstrap confidence set $H_n$ is bounded by 
\[
C 
    \frac{k^{3/2}}{u_n^2}\sqrt{ \frac{\log k + \log n}{n}}.
\]
where $C>0$ depends on $A$ only. 
\end{theorem}

{\bf Remark.} The assumption that $k \geq u_n^2$ is not necessary and can be relaxed, resulting in a slightly more general bound.

Assuming  non-vanishing $u$,
the diameter tends uniformly to $0$ if
$k (\log k)^{1/3} = o(n^{1/3})$.
Interestingly, this is the same condition required in
\cite{portnoy1987central}
although the setting is quite different.

Currently, we do not have a computationally efficient method to
find the supremum and infimum.
A crude approximation is given by
taking a random sample
$\psi_1,\ldots, \psi_N$ from $H_n$ and taking
$$
a(j) \approx \min_{j} g(\hat\psi_j),\ \ \ 
b(j) \approx \max_{j} g(\hat\psi_j).
$$

{\bf Proof of \Cref{thm:beta.accuracy}.}
We will establish the claims by bound the quantity $\left\| 
    \hat\beta_{S} - \beta_{S} 
    \right\| $ uniformly over all $\beta_S \in H_n$.\\
Our proof relies on a first order
Taylor series expansion of of $g$ and on the uniform
bound on the norm of the gradient of each $g_j$ given in\Cref{eq::B-and-lambda}.
Recall that, by conditioning on $\mathcal{D}_{1,n}$, we can regard $S$ and $\beta_{S}$ as a
fixed. Then, letting $G(x)$ be the  $|S| \times b$-dimensional Jacobian of
$g$ at $x$ and using the mean value theorem,  we have that
\begin{align*}
    \left\| 
    \hat\beta_{S} - \beta_{S} 
    \right\| & = 
\left\| 
\left(  
\int_0^1 G\bigl( (1-t)\psi_{S} + ut \hat \psi_{S} \bigr) dt
\right) 
(\hat\psi_{S} - \psi_{S}) 
\right\|\\
& \leq 
 \left\| \int_0^1 G\bigl( (1-t)\psi_{S} + t \hat
 \psi_{S}\bigr) dt \right\|_{\mathrm{op}} \left\| \hat{\psi}_{S} - \psi_{S} \right\|.
\end{align*}
To further bound the previous expression we use the fact, established in
the proof of \Cref{lemma::upsilon}, that  $\| \hat{\psi}_{S} - \psi_{S} \| \leq C k \sqrt{ \frac{ \log n + \log k}{n}
}$ with probability at least $1/n$, where $C$ depends on $A$, for each $P \in  \mathcal{P}_n^{\mathrm{OLS}}$. 
Next, 
\[
 \left\| \int_0^1 G\bigl( (1-t)\psi_{S} + t \hat
 \psi_{S} \bigr) dt \right\|_{\mathrm{op}}  \leq 
 \sup_{t \in (0,1)} \left\| G\left( 1-t)\psi_{S} + t \hat
 \psi_{S} \right) \right\|_{\mathrm{op}}  \leq \sup_{t \in (0,1)} \max_{j \in S} \left\| G\left( 1-t)\psi_{S} + t \hat
 \psi_{S} \right) \right\|
 \]
where $G_i(\psi)$ is the $j^{\mathrm{th}}$ row of $G(\psi)$, which is the gradient of $g_j$ at
$\psi$. Above, the first inequality relies on the convexity  of the operator norm and the
second inequality uses that the fact that the operator norm of a matrix is
bounded by the maximal Euclidean norm of the rows. 
For each $P \in  \mathcal{P}_n^{\mathrm{OLS}}$ and each $t \in (0,1)$ and $j \in S$, the  bound in \eqref{eq:Gj.constants} yields that, for a $C>0$ depending on $A$ only,
\[
\left\| G\left( 1-t)\psi_{S} + t \hat
 \psi_{S} \right) \right\| \leq C \left( \frac{\sqrt{k}}{\hat{u}_t ^2} + \frac{1}{\hat{u}_t} \right), 
\]
where $\hat{u}_t \geq  (1 - t) \lambda_{\min}(\Sigma_{S}) + t \lambda_{\min}(\hat{\Sigma}_{S})$. By \eqref{eq:matrix.bernstein.simple.2} in \Cref{lem:operator} and Weyl's theorem, and using the fact that $ u > u_n$, on an event with probability at lest $1 - \frac{1}{n}$, 
\[
\left\| G\left( 1-t)\psi_{S} + t \hat
 \psi_{S} \right) \right\| \leq C \left( \frac{\sqrt{k}}{u_n ^2} + \frac{1}{u_n} \right) \leq C \frac{\sqrt{k}}{u_n^2},
\] 
where in the last inequality we assume $n$ large enough so that $k \geq u_n^2$. The previous bound does not depend on $t$, $j$ or $P$. The result now follows.
$\Box$

\section{Appendix 2: Proofs of the results in \Cref{section::splitting}}

In all the proofs of the results from \Cref{section::splitting}, we will condition on
the outcome of the sample splitting step, resulting in the random equipartition  $\mathcal{I}_{1,n}$ and $\mathcal{I}_{2,n}$ of $\{1,\ldots,2n\}$, and on $\mathcal{D}_{1,n}$. Thus, we can treat the outcome of the model selection and estimation procedure $w_n$ on $\mathcal{D}_{1,n}$ as a fixed. As a result, we regard $\wS$ as a deterministic, non-empty subset of
$\{1,\ldots,d\}$ of size  by $k < d$ and the projection parameter $\beta_{\wS}$ as a fixed vector of length $k$. Similarly, for the LOCO parameter $\gamma_{\wS}$, the quantities $\widehat{\beta}_{\wS}$ and $\widehat{\beta}_{\wS(j)}$, for $j \in \wS$, which depend on $\mathcal{D}_{1,n}$ also become fixed.
Due to the independence of $\mathcal{D}_{1,n}$ and $\mathcal{D}_{2,n}$, all the probabilistic statements made in the proofs are
therefore referring to the randomness in $\mathcal{D}_{2,n}$ only. Since all our bounds will
depend on $\mathcal{D}_{1,n}$ through the cardinality of $\wS$, which is fixed
at $k$, the same bounds will hold uniformly over all possible values taken on by $\mathcal{D}_{1,n}$ and $\mathcal{I}_{1,n}$  and all possible 
 outcomes of all model selection and estimation procedures $w_n \in \mathcal{W}_n$ run on $\mathcal{D}_{1,n}$. In particular, the bounds are valid unconditionally with respect to the joint distribution of the entire sample and of the splitting outcome.
 	
Also, in the proof $C$ denotes a positive positive that may depend on $A$ only
but not on any other variable,
and whose value may change from line to line.

{\bf Proof of \Cref{thm:beta.accuracy2}.}
As usual, we condition on  $\mathcal{D}_{1,n}$ and thus treat $\wS$ as a fixed subset of
$\{1,\ldots,d\}$ of size $k$.
Recalling the definitions of $\hat{\beta}_{\wS}$ and  $\beta_{\wS}$ given in
\eqref{eq:least.squares} and
\eqref{eq::projection-parameter}, respectively,  and dropping the dependence on
$\wS$ in the notation for convenience, we have that
\begin{align*}
    \| \hat{\beta}_{\wS} -  \beta_{\wS}\| & = \left\|  \left( \hat{\Sigma}^{-1}
    - \Sigma^{-1} \right) \hat{\alpha} + \Sigma^{-1}\left( \hat{\alpha} -
\alpha \right) \right\|\\ 
& \leq  \left\|  \hat{\Sigma}^{-1} - \Sigma^{-1} \right\|_{\mathrm{op}}
\|\hat{\alpha} \| + \frac{1}{u} \| \hat{\alpha} - \alpha\|\\
& = T_1 + T_2.
\end{align*}
By the vector
Bernstein inequality  \eqref{eq:vector.bernstein.simple},
\[
    \| \hat{\alpha} - \alpha \| \leq C A \sqrt{ \frac{k \log n}{n} },
\]
with probability at least $1 - \frac{1}{n}$ and for some universal constant $C$
(independent of $A$). 
Since 
the smallest eigenvalue of $\Sigma$ is bounded from below by $u$, we have that
\[
  T_1 \leq  C   \frac{1}{u} \sqrt{ \frac{k \log n}{n}}.
\]
To bound $\left\|  \hat{\Sigma}^{-1} - \Sigma^{-1} \right\|_{\mathrm{op}}$ in
the term $T_2$ we
write $\hat{\Sigma} = \Sigma + E$ and assume for the moment that $\|E\|_{\mathrm{op}} \|
\Sigma^{-1}\|_{\mathrm{op}} < 1 $ (which of course implies that $\| E \Sigma^{-1}
\|_{\mathrm{op}} < 1$). Since $E$ is symmetric, we have, by formula
5.8.2 in \cite{Horn:2012:MA:2422911}, that
\[
\left\|  \hat{\Sigma}^{-1} - \Sigma^{-1} \right\|_{\mathrm{op}} = \left\|
(\Sigma + E)^{-1} - \Sigma^{-1} \right\|_{\mathrm{op}} \leq \|
\Sigma^{-1}\|_{\mathrm{op}} \frac{\| E 
\Sigma^{-1}\|_{\mathrm{op}}  } {  1 - \| E 
\Sigma^{-1}\|_{\mathrm{op}}  },
\]
which in turn is upper  bounded by 
\[
    \|\Sigma^{-1}\|^2_{\mathrm{op}}  \frac{\|
    \hat{\Sigma} - \Sigma 
    \|_{\mathrm{op}}  } {  1 -  \|\hat{\Sigma} - \Sigma \|_{\mathrm{op}} \|
    \Sigma^{-1}\|_{\mathrm{op}}  }.
\]
The matrix Bernstein inequality \eqref{eq:matrix.bernstein.simple.2} along with the assumption that $U \geq \eta > 0$  yield that,
for a positive  $C$ (which depends on $\eta$), 
\[
\|\hat{\Sigma} - \Sigma \|_{\mathrm{op}} \leq C A    \sqrt{ k U  \frac{ \log k +
\log n}{n}},   
\]
with probability at least $1 - \frac{1}{n}$.  Using the fact that $\| \Sigma^{-1} \|_{\mathrm{op}} \leq
\frac{1}{u}$ and the assumed asymptotic scaling on $B_n$ we see that $\|
\Sigma^{-1} E \|_{\mathrm{op}}  \leq 1/2$ for all $n$ large enough.
Thus, for all such $n$, we obtain that, with probability at least $ 1-
\frac{1}{n}$, 
\[
T_2 \leq 2 C A 
    \frac{ k}{u^2}
    \sqrt{ U \frac{ \log k +
    \log n}{n}},
\]
since $\| \hat{\alpha} \| \leq A \sqrt{k}$ almost surely.  
Thus we have shown that \eqref{eq::beta2} holds, with
probability at least $1 - \frac{2}{n}$ and for all $n$ large enough.  
This bound holds uniformly over all $P \in \mathcal{P}_n^{\mathrm{OLS}}$.
$\Box$

\noindent{\bf Proof of \Cref{thm::big-theorem}.}
In what follows,
any term of the order $\frac{1}{n}$ are
absorbed into terms of asymptotic bigger order.

As remarked at the beginning of this section, we first condition on $\mathcal{D}_{1,n}$ and the outcome of the sample splitting, so that $\wS$ is regarded as a fixed non-empty subset $S$ of $\{1,\ldots,d\}$ of size at most $k$.
The bounds \eqref{eq:big-theorem.Linfty} and \eqref{eq:big-theorem.hyper} are established using \Cref{thm::coverage} and \Cref{thm::bonf} from \Cref{section::berry},
where we may take the function $g$ as in \eqref{eq:g.beta}, $s = k$,  $b =
\frac{k^2 + 3k}{2} $ and $\psi = \psi_{\wS}$ and $\hat{\psi}_{\wS}$ as in \eqref{eq:psi.beta} and
\eqref{eq:hat.psi.beta}, respectively.
As already noted, 
$\psi$ is always in the domain of $g$ and, as long as $n \geq d$, so is $\hat{\psi}$,
almost surely. 
A main technical difficulty in applying the results of \Cref{section::berry} is to
obtain  good approximations for the quantities
$\underline{\sigma}, \overline{H}$ and $B$. This can be accomplished using the bounds provided in
\Cref{lemma::horrible} below, which rely on matrix-calculus. Even so, the
claims in the theorem do not simply follow by plugging those bounds in Equations
\eqref{eq::Delta}, \eqref{eq::this-is-upsilon} and \eqref{eq.Delta3.tilde} from \Cref{section::berry}.
 Indeed, close inspection of the proof of \Cref{theorem::deltamethod} (which is needed by both Theorems \ref{thm::coverage} and \ref{thm::bonf}) shows that the quantity $\overline{H}$, defined in \eqref{eq:H.and.B}, is used there only once, but critically, to obtain the almost everywhere bound in Equation \eqref{eq:bound.H}.  Adapted to the present setting, such a bound would be of the form 
\[
 \max_{j \in S} \int_{0}^1 \left\| H_j\left((1-t) \psi_S(P) + t \hat{\psi}_S(P) \right) \right\|_{\mathrm{op}} dt \leq \overline{H},
\]
almost everywhere, for each $S$ and $P \in \mathcal{P}_n^{\mathrm{OLS}}$, where $\psi_S = \psi_S(P)$ and $\hat{\psi}_S = \hat{\psi}_S(P)$ are given in \eqref{eq:psi.beta} and \eqref{eq:hat.psi.beta}, respectively. Unfortunately, the above inequality cannot be expected to hold almost everywhere, like we did in \Cref{section::berry}. Instead we will derive a high probability bound. In detail, using the second inequality in \eqref{eq::B-and-lambda} below we obtain that, for any $t \in [0,1]$, $S$, $j \in S$ and $P \in \mathcal{P}_n^{\mathrm{OLS}}$,
\[
  \left\| H_j\left((1-t) \psi_S(P) + t \hat{\psi}_S(P) \right) \right\|_{\mathrm{op}} \leq C \frac{k}{\hat{u}_t^3}
\]
 where $\hat{u}_t  = \lambda_{\min}( t \Sigma_S + ( 1- t) \hat{\Sigma}_S) \geq t \lambda_{\min}(\Sigma_S) + (1-t) \lambda_{\min}(\hat{\Sigma}_S)$ and the constant $C$ is the same as in \eqref{eq::B-and-lambda} ( the dependence  $\Sigma_S$ and $\hat{\Sigma}_S$ on $P$ is implicit in our notation). Notice that, unlike in \eqref{eq:H.and.B} in the proof of \Cref{theorem::deltamethod}, the above bound is random. By assumption, $\lambda_{\min}(\Sigma_S) \geq u$ and, by \eqref{eq:matrix.bernstein.simple.2} in \Cref{lem:operator} and Weyl's theorem, $\lambda_{\min}(\hat{\Sigma}_S) \geq u_n$ with probability at least $1 - \frac{1}{n}$ for each $P \in \mathcal{P}_n$. Since $u_n \leq u$,  we conclude that, for each  $S$, $j \in S$ and $P \in \mathcal{P}_n^{\mathrm{OLS}}$,
   \[
  \max_{j \in S} \int_0^1 \left\| H_j\left((1-t) \psi_S(P) + t \hat{\psi}_S(P) \right) \right\|_{\mathrm{op}} dt \leq C \frac{k}{u_n^3},
 \]
  on an event of probability at least $1 - \frac{1}{n}$.
  The same arguments apply to the bound \eqref{eq:here} in the proof \Cref{lemma::upsilon}, yielding that the term $\aleph_n$, given in \eqref{eq:aleph}, can be bounded, on an event of probability at least $1 - \frac{1}{n}$ and using again \Cref{lemma::horrible}, by
  \begin{equation}\label{eq:new.aleph}
  C \frac{k^{5/2}}{u_n^3 u^2} \overline{v} \sqrt{ \frac{ \log n}{n}}, 
  \end{equation}
  for each $P \in \mathcal{P}_n^{\mathrm{OLS}}$ and some $C>0$ dependent on $A$ only. (In light of the bounds derived next in \Cref{lemma::horrible}, the dominant term in the bound on $\aleph_n$ given in \eqref{eq:upsilon} is $ \overline{H} B \overline{v} \sqrt{ b\frac{   
 \log n}{n}}$, from which \eqref{eq:new.aleph} follows. We omit the  details).
  
 Thus, for each $P \in \mathcal{P}_n^{\mathrm{OLS}}$,  we may now apply Theorems \ref{thm::coverage} and \ref{thm::bonf} on event with probability no smaller than $ 1- \frac{1}{n}$, whereby the term $\overline{H}$ is replaced by $C \frac{k}{u^3_n}$ and the terms $B$ and $\overline{\sigma}$ are bounded as in \Cref{lemma::horrible}.

\begin{lemma}
\label{lemma::horrible}
For any $j \in \wS$,  let $\beta_{\wS}(j) = e_j^\top \beta_{\wS} = g_j(\psi)$
where $e_j$ is the $j^{\mathrm{th}}$ standard unit vector.
Write $\alpha = \alpha_{\wS}$ and $\Omega = \Sigma^{-1}_{\wS}$ and assume that $k
\geq u^2$.
The gradient and Hessian of $g_j$ are given by
\begin{equation}\label{eq:Gj} 
G^\top_j = e^\top_j \Big( \left[ - \left( \alpha^\top \otimes I_k \right) 
(\Omega \otimes \Omega) \;\;\;\;\; \Omega\right] \Big)  D_h
\end{equation}
and
\begin{equation}\label{eq:Hj}
H_j = D_h^\top A_j D_h, 
\end{equation}
respectively, 
where
$$
A_j = 
\frac{1}{2}\left(   (I_b \otimes e^\top_j) H + H^\top (I_b \otimes e_j)
\right),  
$$
and
$$
 H =  \left[
    \begin{array}{c}
- \Big( ( \Omega \otimes \Omega) \otimes I_k \Big)  \Big[0_{k^3 \times
k^2} \;\;\;\;\; ( I_k \otimes \mathrm{vec}(I_k)) \Big]  +
 \Big( I_{k^2} \otimes (\alpha^\top
    \otimes I_k) \Big) G \Big[ (\Omega \otimes \Omega) \;\;\;\;\;
    0_{k^2 \times k}\Big]\\
 \;\\ 
 \Big[ - (\Omega \otimes \Omega) \;\;\;\;\; 0_{k^2 \times k} \Big]
    \end{array}
\right],
$$
and
$D_h$ is the modified duplication matrix defined by 
$D \psi_h = \psi$,
with
$\psi_h$ the vector consisting of the subset of $\psi$ not including
entries that correspond to the upper diagonal entries of $\Sigma$.
Assume that $k \geq u^2$.
Then, 
\begin{equation}\label{eq::B-and-lambda}
B= \sup_{P \in \mathcal{P}_n^{\mathrm{OLS}} } \max_j \|G_j(\psi(P)) \| \leq C \frac{ \sqrt{k} }{u^2},\ \ \ 
\overline{H}=\max_j \sup_{P \in \mathcal{P}_n^{\mathrm{OLS}}} \| H_j(\psi(P))\|_{\mathrm{op}}  \leq C
\frac{k}{u^3},
\end{equation}
and
\begin{equation}\label{eq:sigmamin}
    \underline{\sigma} = \inf_{P \in \mathcal{P}^{\mathrm{OLS}}_n} \min_j \sqrt{ G_j V G_j^\top}
    \geq \frac{ \sqrt{v } }{ U }, 
\end{equation}
where $C>0$ depends on $A$ only.
\end{lemma}

{\bf Remark.}
The assumption that $k \geq u^2$
is not actually needed but this is the most common case and it simplifies the expressions a bit.

{\bf Proof of \Cref{cor:accuracy.beta}.}
The  maximal length is of the sides of $\tilde{C}_n$ is
\[
2 \max_{j \in \wS}   z_{\alpha/(2k)}
    \sqrt{\frac{ \hat\Gamma_{\wS}(j,j)}{n}} \leq 2 \max_{j \in \wS} z_{\alpha/(2k)}
    \sqrt{\frac{ \Gamma_{\wS}(j,j) + \left| \hat\Gamma(j,j)-\Gamma(j,j)\right|}{n}}. 
  \]
By \Cref{lemma::upsilon} and Equation \eqref{eq:new.aleph}, the event that 
\[
\max_{  j,l  \in \wS } \left| \hat\Gamma(j,l)-\Gamma(j,l)\right| \leq C \frac{k^{3/2}}{u_n^3 u^2} \overline{v} \sqrt{ \frac{k^2 \log n}{n}}
\]
holds with probability at least $1 - \frac{2}{n}$ and for each $P \in \mathcal{P}_n^{\mathrm{OLS}}$, where $C > 0$ depends on $A$ only.
Next, letting $G = G(\psi_{\wS})$ and $V = V_{\wS}$, we have that, for each $j \in \wS$ and $P \in \mathcal{P}_n^{\mathrm{OLS}}$,
\[
\Gamma_{\wS}(j,j) = G_j V G_j^\top \leq \|G_j\|^2 \lambda_{\max}(V) \leq B^2 \overline{v} \leq C \frac{k }{u^4} \overline{v}
\]
where $G_j$ denotes the $j^{\mathrm{th}}$ row of $G$ and, as usual, $C>0$ depends on $A$ only. The second inequality in the last display follows from property 3. in \Cref{def:Pdagger} and by the definition of $B$ in \eqref{eq:H.and.B}, while the third inequality uses the first bound in Equation \eqref{eq::B-and-lambda}.
    The result follows from combining the previous bounds and the fact that $z_{\alpha/(2k)} = O \left( \sqrt{ \log k} \right)$. $\Box$

{\bf Proof of \Cref{theorem::beta.boot}.}
We  condition on $\mathcal{D}_{1,n}$ and the outcome of the sample
splitting. 
The claimed results follows almost directly from 
\Cref{theorem::boot}, with few additional technicalities.
The first difficulty is that the least squares estimator is not always
well-defined under the bootstrap measure, which is the probability distribution
of $n$ uniform draws with replacement from $\mathcal{D}_{2,n}$. In fact, any
draw consisting of less than $d$ distinct elements of
$\mathcal{D}_{2,n}$ will be such that the corresponding empirical covariance
matrix will be rank deficient and therefore not invertible. On the other hand, because the
distribution of $\mathcal{D}_{2,n}$ has a Lebesgue density by assumption, any
set of $d$ or more points from $\mathcal{D}_{2,n}$ will be in general position and
therefore will yield a unique set of least squares coefficients.
To deal with such
complication we will simply apply \Cref{theorem::boot} on the event that the
bootstrap sample contains  $d$ or more distinct elements of
$\mathcal{D}_{2,n}$, whose complementary event, given the assumed scaling of $d$ and $n$, has
probability is exponentially small in $n$, as shown
next.

\begin{lemma}\label{eq:lem.occupancy}
For $d \leq n/2$, the probability that sampling with replacement $n$ out of $n$ distinct objects
will result in a set with less than $d$ distinct elements is no larger than
\begin{equation}\label{eq:occupancy}
 \exp \left\{ - \frac{n (1/2 - e^{-1})^2}{2}  \right\}.
\end{equation}
\end{lemma}

{\bf Remark.} The condition that $d \leq n/2$ can be replaced by the condition
that $d \leq c n$, for any $ c \in (0, 1 - e^{-1})$.

Thus, we will assume that the event that the bootstrap sample contains
$d$ or more distinct elements of $\mathcal{D}_{2,n}$. This will result in an extra term 
 that is of smaller order than any of the other terms and therefore can be
discarded by choosing a larger value of the leading constant.

At this point, the proof of the theorem is nearly identical to the proof of 
\Cref{theorem::boot} except for the way the term $A_3$ is handled. The
assumptions that $n$ be large enough so that $v_n$ and $u_n$ are
both positive implies, by \Cref{lem:operator} and Weyl's theorem, that, for each $P \in \mathcal{P}_{n}^{\mathrm{OLS}}$ and with probability at least $
1- \frac{2}{n}$  with respect to the distribution of $\mathcal{D}_{2,n}$, the bootstrap distribution
belongs to the class $\mathcal{P}_n^*$ of probability distributions for the pair $(X,Y)$ that
satisfy the properties of the probability distributions in the class
$\mathcal{P}_n^{\mathrm{OLS}}$ with two differences: (1) the quantities $U$, $u$, $v$ and
$\overline{v}$ are replaced by $U_n$, $u_n$, $v_n$ and $\overline{v}_n$, respectively, and (2) the distributions in $\mathcal{P}^*_n$ need not have
 a Lebesgue density. Nonetheless, since the Lebesgue density assumption
is only used to guarantee that empirical covariance matrix is invertible, a
fact that is also true for the bootstrap distribution under the event that the
bootstrap sample consists of $d$ or more distinct elements of $\mathcal{D}_{2,n}$, the
bound on the term $A_3$ established in  
\Cref{theorem::boot}
holds for the larger class $\mathcal{P}_n^*$ as well. 

Next, \Cref{lemma::horrible} can be used to bound the quantities 
$\overline{\sigma}$ and $B$ for the class $\mathcal{P}_n^*$. 
As for the bound on $\overline{H}$, we proceed as in the proof of \Cref{thm::big-theorem} and conclude that, for each non-empty subset $S$ of $\{1,\ldots,d\}$ and $P \in \mathcal{P}_n^*$,
\[
 \max_{j \in S} \int_{0}^1 \left\| H_j\left((1-t) \psi_S(P) + t \hat{\psi}_S(P) \right) \right\|_{\mathrm{op}} dt  \leq \frac{C}{k}{u_n^3}
\]
on an event of probability at least $1 - \frac{1}{n}$, where $C$ is the constant appearing in the second bound in \eqref{eq::B-and-lambda}.
Thus, we may take $\frac{C}{k}{u_n^3}$ in lieu of $\overline{H}$ and then apply \Cref{theorem::boot} (noting that the high probability bound in the  last display holds for each  $P \in \mathcal{P}_n^*$ separately).

\noindent {\bf Proof of \Cref{thm::CLT2}.}    As remarked at the beginning of this appendix, throughout the proof all probabilistic statements will be made conditionally on the
    outcome of the splitting and on 
    $\mathcal{D}_{1,n}$. Thus, in particular, $\wS$ is to be regarded as a fixed subset of
    $\{1,\ldots,d\}$ of size $k$.

    Let $Z_n\sim N(0,\hat \Sigma_{\wS})$, with $ \hat{\Sigma}_{\wS}$  given in \ref{eq:Sigma.loco}. 
    Notice that $\hat{\Sigma}_{\wS}$ is almost surely positive definite,
     a consequence of adding extra noise in the definition of
    $\gamma_{\wS}$ and $\hat{\gamma}_{\wS}$.
    Then, using Theorem 2.1 in
    \cite{cherno2},  there
    exists a universal  constant $C > 0$ such that   
\begin{equation}\label{eq::secondx}
    \sup_{ t = (t_j, j \in \wS) \in \mathbb{R}^{\wS}_{+}} \Bigl| \mathbb{P}(
    \sqrt{n}|\hat\gamma_{\wS}(j)  - \gamma_{\wS}(j) | \leq t_j, \forall j \in \wS) - 
    \mathbb{P}(|Z_n(j)|  \leq t_j, \forall j \in \wS )\Bigr| \leq C
    \mathrm{E}_{1,n},
\end{equation}
where $\mathrm{E}_{1,n}$ is given in \eqref{eq:E1n}.
By restricting the supremum in the above display to all $t \in \mathbb{R}^{\wS}_+$
with identical coordinates, we also obtain that
\begin{equation}\label{eq::firstx}
    \sup_{t > 0} \Bigl| \mathbb{P}(\sqrt{n}||\hat\gamma_{\wS} - \gamma_{\wS}||_\infty \leq t) - 
    \mathbb{P}\left(||Z_n||_\infty \leq t \right)\Bigr| \leq C \mathrm{E}_{1,n}. 
\end{equation}
In order to show \eqref{eq:loco.coverage1} and \eqref{eq:loco.coverage2}, we will use the
same arguments used in the proofs of \Cref{thm::coverage} and  \Cref{thm::bonf}.
We  first define $\mathcal{E}_n$ to be the event that
\begin{equation}\label{eq:loco.aleph}
\max_{i,j} \left| \widehat{\Sigma}_{\wS}(i,j) - \Sigma_{\wS}(i,j) \right| \leq
N_n, 
	\end{equation}
	where $N_n$ is as in \eqref{eq:Nn}.
	Each entry
of $\widehat{\Sigma}_{\wS} - \Sigma_{\wS}$  is   bounded in absolute value by
$\left( 2(A+\tau) + \epsilon \right)^2$, and therefore is a sub-Gaussian with
parameter $\left( 2(A+\tau) + \epsilon \right)^4$. Using 
a standard derivation for bounding the maximum of sub-Gaussian random variables we obtain
that
$\mathbb{P}(\mathcal{E}_n^c) \leq \frac{1}{n} $. 
The bound \eqref{eq:loco.coverage1} follows from the same arguments as in the proof
\Cref{thm::coverage}: combine the Gaussian comparison Theorem \ref{thm:comparisons} with
\eqref{eq::firstx} and notice that $\epsilon/\sqrt{3}$ is a lower bound on the
standard deviation
of the individual coordinates of the $\delta_i$'s. In particular, the Gaussian
comparison theorem yields the additional error term $C \mathrm{E}_{2,n} + 
\frac{1}{n}$ given in \eqref{eq:E2n}, for some universal positive constant $C$. Similarly, 
\eqref{eq:loco.coverage2} can be established along the lines of the proof of
\Cref{thm::bonf}, starting from the bound \eqref{eq::secondx}. In this case we pick up an additional error term
$C \tilde{\mathrm{E}}_{2,n} + \frac{1}{n}$ of different
form, shown in \eqref{eq:tildeE2.n}, where $C>0$ is a different universal
constant. 

Since all the bounds we have derived do not depend on
$\mathcal{D}_{1,n}$, the outcome of the splitting and $w_n$, the same bounds therefore hold for the joint
probabilities, and uniformly over the model selection and estimation procedures. 
The above arguments hold for each $P \in  \mathcal{P}_n^{\mathrm{LOCO}}$.
$\Box$

\noindent {\bf Proof of \Cref{cor:accuracy.LOCO}.}
Following the proof of \Cref{cor:accuracy.beta}, for each $P \in  \mathcal{P}_n^{\mathrm{LOCO}}$ and on the event $\mathcal{E}_n$ given in \eqref{eq:loco.aleph} (which has probability at leas $1- \frac{1}{n}$), we have that 
\begin{align*}
2 \max_{j \in \wS}   z_{\alpha/(2k)}
    \sqrt{\frac{ \hat\Sigma_{\wS}(j,j)}{n}} & \leq 2 \max_{j \in \wS} z_{\alpha/(2k)}
    \sqrt{\frac{ \Sigma_{\wS}(j,j) + \left| \hat\Sigma(j,j)-\Sigma(j,j)\right|}{n}}\\
    & \leq z_{\alpha/(2k)}
    \sqrt{ \frac{ (2(A + \tau) + \epsilon)^2 +  N_n }{n}}.
    \end{align*}
    The claimed bound follows from the definition of $N_n$ as in \eqref{eq:Nn}.
$\Box$

	\noindent {\bf Proof of \Cref{thm:boot.loco}.} All the probabilistic
	statements that follow are to be understood 
	conditionally on the outcome of the sample splitting and on
	$\mathcal{D}_{1,n}$. Thus, $\mathcal{I}_{1,n}$, $\wS$, $\hat{\beta}_{\wS}$ and, for each $j
	\in \wS$, $\hat{\beta}_{\wS(j)}$ are to be regarded as fixed, and the
	only randomness is with respect to the joint marginal distribution of
	$\mathcal{D}_{2,n}$ and $(\xi_i, i \in \mathcal{I}_{2,n})$, and two
	auxiliary independent standard Gaussian vectors in $\mathbb{R}^{\wS}$,
	$Z_1$ and $Z_2$, independent of everything else. 

     Let $\hat{\gamma}^*_{\wS} \in \mathcal{R}^{\wS}$ denotes the vector 
of LOCO parameters arising from the bootstrap distribution corresponding to the
empirical measure associated to the $n$  triplets $\left\{  (X_i,Y_i,\xi_i), i \in
    \mathcal{I}_{2,n} \right\}$. 
    
Next, 
\[
    \mathbb{P}\left( \sqrt{n} \| \gamma_{\wS} - \hat{\gamma}_{\wS}  \|_\infty
    \leq \hat{t}^*_{\alpha} \right) \geq  
    \mathbb{P}\Big( \sqrt{n} \| \hat{\gamma}^*_{\wS} - \hat{\gamma}_{\wS}   \|_\infty \leq \hat{t}^*_
	\alpha  \left | (X_i,Y_i,\xi_i), i \in
    \mathcal{I}_{2,n} \right. \Big) 
	-( A_1 + A_2 + A_3),
\]
where
\begin{align*}
    A_1 & = \sup_{t >0} \left| \mathbb{P}\left( \sqrt{n} \| \hat{\gamma}_{\wS} - \gamma_{\wS}   \|_\infty
    \leq t \right) - \mathbb{P}( \| Z \|_\infty \leq t ) \right|,\\
    A_2 & = \sup_{t>0} \left|  \mathbb{P}( \| Z \|_\infty \leq t ) - \mathbb{P}( \|
    \hat{Z} \|_\infty \leq t )      \right|,\\
    \text{and} & \\
    A_3 & = \sup_{t >0} \Big|
   \mathbb{P}\Big( \sqrt{n} \| \hat{\gamma}^*_{\wS} - \hat{\gamma}_{\wS}
   \|_\infty \leq t   \left | (X_i,Y_i,\xi_i), i \in
	\mathcal{I}_{2,n} \right. \Big) -  \mathbb{P}( \| \hat{Z} \|_\infty \leq
	t ) \Big|,
\end{align*}
with $Z = \Sigma_{\wS}^{1/2} Z_1$ and $\hat{Z} = \widehat{\Sigma}_{\wS} Z_2$.

Then, $A_1 \leq C \mathrm{E}_{1,n}$ by \eqref{eq::firstx} and $A_2 \leq
C \mathbb{E}_{2,n} + \frac{1}{n} $, by
 applying the Gaussian comparison Theorem \ref{thm:comparisons}
 on the event $\mathcal{E}_n$ that \eqref{eq:loco.aleph} holds, whereby
 $\mathbb{P}(\mathcal{E}_n^c) \leq \frac{1}{n}$ as argued in the proof of \Cref{thm::CLT2}.
Finally the bound on $A_3$ follows from applying Theorem 2.1 in
    \cite{cherno2} to the bootstrap measure, conditionally on $(X_i,Y_i,\xi_i), i \in
	\mathcal{I}_{2,n}$, just like it was done in the proof of \Cref{thm::CLT2}.
	In this case, we need to restrict to the even $\mathcal{E}_n$
	to ensure that the minimal variance for the bootstrap measure is bounded
	away from zero. To that end, it will be enough to take $n$ large enough
	so that
	$\epsilon_n $ is positive and to 
	replace $\epsilon$ with $\epsilon_n$. The price for this extra step is a
	factor of $\frac{1}{n}$, which upper bounds
	$\mathbb{P}(\mathcal{E}_n^c)$. Putting all the pieced together we arrive at the
	bound
	\[
	    A_3 \leq C \mathrm{E}^*_{1,n} + \frac{1}{n}.
	\]
	Finally notice that $\mathrm{E}_{1,n} \leq \mathrm{E}^*_{1,n}$ since
	$\epsilon_n \leq \epsilon$.

	The very same arguments apply to the other bootstrap confidence set
	$\tilde{C}^*_\alpha$, producing the very same bound. We omit the proof
	for brevity but refer the reader to the proof of
	\Cref{theorem::boot} for details.

	All the bounds obtained so far are conditionally on the outcome of the
	sample splitting and on $\mathcal{D}_{1,n}$ but are not functions of
	those random variables. Thus, the same bounds hold also unconditionally, for each $P \in  \mathcal{P}_n^{\mathrm{LOCO}}$. 

$\Box$

Let $F_{n,j}$ denote the empirical cumulative distribution function of $\{
    \delta_i(j), i \in \mathcal{I}_{2,n}\}$ and $F_j$ the true cumulative
    distribution function of $\delta_i(j)$.
    Thus, setting $\beta_l = l/n$ and $\beta_u = u/n$, we see that  
    $\delta_{(l)}(j) = F_{n,j}^{-1}(\beta_l)$ and $\delta_{(u)}(j) =
    F_{n,j}^{-1}(\beta_u)$
and, furthermore, that $F_{n,j}(F_{n,j}^{-1}(\beta_l)) = \beta_l$ and
$F_{n,j}F(_{n,j}^{-1}(\beta_u)) = \beta_u$.
In particular notice that $\beta_l$ is smaller than  $ \frac{1}{2} -
\sqrt{\frac{1}{2n}\log\left(\frac{2k}{\alpha}\right)}$ by at most $1/n$ and,
similarly, $\beta_u$ is larger than $ \frac{1}{2} +
\sqrt{\frac{1}{2n}\log\left(\frac{2k}{\alpha}\right)}$ by at most $1/n$.

By assumption, the median $\mu_j = F_j^{-1}(1/2)$ of
$\delta_i(j)$ is unique and the derivative of $F_j$ is larger than $M$ at all points 
within a distance of $\eta$ from $\mu_j$. Thus, by the mean value theorem, we must have that, for all $x \in \mathbb{R}$ such
that $| x -\mu_j | < \eta$,
\[
    M |x - \mu_j| \leq  | F_j(x)  -F_j(\mu_j)|.
\]
As a result, if
\begin{equation}\label{eq:M.inverse}
| F_j(x) - F_j(\mu_j) | \leq M \eta,
\end{equation}
it is the case that $|x- \mu_j| \leq \eta$, and therefore, that $| x- \mu_j|
\leq \frac{F_j(x) - F_j(\mu_j)}{M}$.
By the DKW inequality and the union bound, with probability at least $1-1/n$, 
\begin{equation}\label{eq:dkw.median}
    \max_{j \in \wS} \|F_{n,j} - F_j  \|_\infty \leq \sqrt{\frac{ \log 2kn}{2n} }.
\end{equation}
Thus, for any $j \in \wS$,
\[
    \left|  F_{n,j}( \delta_{(u)}(j)) - F_j( \delta_{(u)}(j))\right| \leq
    \sqrt{\frac{ \log 2kn}{2n} }.
    \]
Since 
\[
    F_{n,j}( \delta_{(u)}(j))   = \beta_u \leq 1/2 + \frac{1}{n}  + 
    \sqrt{\frac{1}{n}\log\left(\frac{2k}{\alpha}\right)} = F_j(\mu_j) +  \frac{1}{n}  +
    \sqrt{\frac{1}{n}\log\left(\frac{2k}{\alpha}\right)},
\]
using \eqref{eq:M.inverse}, we conclude that, on the event \eqref{eq:dkw.median}
and provided that $  \frac{1}{n}  +
    \sqrt{\frac{1}{2n}\log\left(\frac{2k}{\alpha}\right)} + \sqrt{
	\frac{ \log 2kn}{2n} }
 \leq \eta M$,
 \[
     | \mu_j - \delta_{(u)}(j) | \leq \frac{1}{M} \left(  \frac{1}{n}  +
    \sqrt{\frac{1}{2n}\log\left(\frac{2k}{\alpha}\right)} + \sqrt{
	\frac{ \log 2kn}{2n} } \right).
    \]
Similarly, under the same conditions, 
\[
     | \mu_j - \delta_{(l)}(j) | \leq \frac{1}{M} \left(   \frac{1}{n}  +
    \sqrt{\frac{1}{2n}\log\left(\frac{2k}{\alpha}\right)} + \sqrt{
	\frac{ \log 2kn}{2n} } \right).
    \]
The claim now follows by combining the last two displays. Notice that
the result holds uniformly
over all $j \in \wS$ and all distributions satisfying the conditions of the
theorem. $\Box$

\section{Appendix 3: Proof of the results in \Cref{section::splitornot}}

{\bf Proof of Lemma \ref{lemma::est-accuracy}.}
The upper bounds are obvious.
The lower bound (\ref{eq::lower1}) is from Section 4 in
\cite{sackrowitz1986evaluating}.
We now show (\ref{eq::lower2}).
Let $\hat\beta =g(Y)$ be any estimator
where $Y=(Y_1,\ldots, Y_n)$.
Given any $Y$ and any $w(Y)$,
$\hat\beta$ provides an estimate of $\beta(J)$
where $J= w(Y)$.
Let $w_j$ be such that
$w_j(X)=j$.
Then define
$\hat\beta = ( g(Y,w_1(Y)),\ldots, g(Y,w_D(Y)))$.
Let
$w_0(Y) = \argmax_j |\beta(j)-\hat\beta(j)|$.
Then
$\mathbb{E}[|\hat \beta(J) - \beta(J)|]=  \mathbb{E}[||\hat \beta - \beta||_\infty]$.
Let $P_0$ be multivariate Normal with mean
$(0,\ldots, 0)$ and
identity covariance.
For $j=1,\ldots, D$ let
$P_j$ 
be multivariate Normal with mean
$\mu_j=(0,\ldots,0,a,0, 0)$ and
identity covariance
where $a = \sqrt{ \log D/(16n)}$.
Then
\begin{align*}
\inf_{\hat\beta}\sup_{w\in {\cal W}_n}\sup_{P\in {\cal P}_n}\mathbb{E}[|\hat \beta(J) - \beta(J)|] &\geq
\inf_{\hat\beta}\sup_{P\in M}\mathbb{E}[|\hat \beta(J) - \beta(J)|]\\
&=
\inf_{\hat\beta}\sup_{P\in M}\mathbb{E}[||\hat \beta - \beta||_\infty]
\end{align*}
where $J= w_0(Y)$ and
$M = \{P_0,P_1,\ldots,P_D\}$.
It is easy to see that
$$
{\rm KL}(P_0,P_j) \leq \frac{\log D}{16 n}
$$
where KL denotes the Kullback-Leibler distance.
Also,
$||\mu_j - \mu_k||_\infty \geq a/2$
for each pair.
By Theorem 2.5 of 
\cite{tsybakov2009introduction},
$$
\inf_{\hat\beta}\sup_{P\in M}\mathbb{E}[||\hat \mu - \mu ||_\infty] \geq \frac{a}{2}
$$
which completes the proof. $\Box$

{\bf Proof of Lemma \ref{lemma::contiguity}.}
We use a contiguity argument like that in 
\cite{leeb2008can}.
Let $Z_1,\ldots, Z_D \sim N(0,1)$.
Note that $\hat\beta(j) \stackrel{d}{=} \beta(j)+ Z_j/\sqrt{n}$.
Then
\begin{align*}
\psi_n(\beta) &= \mathbb{P}(\sqrt{n}(\hat\beta(S) - \beta(S))\leq t) =
\sum_j \mathbb{P}(\sqrt{n}(\hat\beta(j) - \beta(j))\leq t,\ \hat\beta(j) > \max_{s\neq j}\hat\beta_s)\\
&=
\sum_j \mathbb{P}(\max_{s\neq j}Z_s + \sqrt{n}(\beta(s)-\beta(j)) < Z_j < t) = \sum_j \Phi(A_j)
\end{align*}
where
$\Phi$ is the $d$-dimensional standard Gaussian measure and
$$
A_j = \Bigl\{ \max_{s\neq j}Z_s + \sqrt{n}(\beta(s)-\beta(j) < Z_j < t \Bigr\}.
$$
Consider the case where
$\beta = (0,\ldots, 0)$.
Then
$$
\psi_n(0)= D\, \Phi(\max_{s\neq 1}Z_s   < Z_1 < t)  \equiv b(0).
$$
Next consider
$\beta_n = (a/\sqrt{n},0,0,\ldots, 0)$
where $a>0$ is any fixed constant.
Then
\begin{align*}
\psi(\beta_n) &= \Phi( (\max_{s\neq 1}Z_s )-a < Z_1 < t)\\
&\ \ \ \ \  +
\sum_{j=2}^D \Phi(\max\{Z_1+a,Z_2,\ldots, Z_{j-1},Z_{j+1},\ldots, Z_D\} < Z_j < t)\\
&\equiv b(a).
\end{align*}
Suppose that $\hat\psi_n$ is a consistent estimator of $\psi_n$.
Then, under $P_0$,
$\hat\psi_n \stackrel{P}{\to} b(0)$.
Let $P_n = N(\beta_n,I)$ and $P_0 = N(0,I)$.
It is easy to see that
$P_0^n(A_n)\to 0$ implies that
$P_n^n(A_n)\to 0$ so that
$P_n$ and $P_0$ are contiguous.
So, by Le Cam's first lemma \citep[see, e.g.][]{green.book},
under $P_n$, we also have that
$\hat\psi_n \stackrel{P}{\to} b(0)$.
But $b(0)\neq b(a)$, which contradicts the assumed consistency of $\hat\psi_n$. $\Box$

{\bf Proof of Lemma \ref{lemma::many-means-bound}.}
Let $P_0 = N(\mu_0, \frac{1}{n}I_D)$, where $\mu_0 = 0$, and for $j=1,\ldots,D$ let $P_j = N(\mu_j,
\frac{1}{n}I_D)$, where $\mu_j$ is the $D$-dimensional vector with $0$
entries except along the $j^{\mathrm{th}}$ coordinate, which takes the value $\sqrt{c
\frac{\log D}{n}}$, where $0 < c < 1$. Consider the mixture
$\overline{P} = \frac{1}{D}
\sum_{j=1}^D P_j$. Then, letting $\theta_j$ and $\theta_0$ be the
largest coordinates of $\mu_j$ and $\mu_0$ respectively, we have that
$\| \theta_j - \theta_0 \|^2 = \frac{c \log D}{n}$ for all $j$. 
Next, some algebra yields that	
the $\chi^2$ distance
between $P_0$ and the mixture $\overline{P} = \frac{1}{D}
\sum_{j=1}^D P_j$ is
$\frac{1}{D} e^{ c \log D} - \frac{1}{D}$, which vanishes as $D$ tends to
$\infty$. Since this is also an upper bound on the squared  total variation distance
between $P_0$ and $\overline{P}$, the result follows from an application
of Le Cam Lemma \citep[see, e.g.][]{tsybakov2009introduction}.  $\Box$

\section{Appendix 4: Proof of the results in \Cref{section::berry}}

{\bf Proof of Theorem \ref{theorem::deltamethod}.}
For ease of readability, we will write $G_j$ and $G$ instead of $G_j(\psi)$
and $G(\psi)$, respectively. Throughout the proof,  $C$ will indicate
a positive number whose value may change from line to line and which depends on
$A$ only, but on none of the remaining variables.

For each $j \in \{1,\ldots,s\}$, we use  a second order Taylor
expansion of $\widehat{\theta}_j$ to obtain that 
$$
\hat\theta_j = \theta_j + G_j^\top(\hat\psi - \psi) +
\frac{1}{2n}\delta^\top \Lambda_j \delta, \quad \forall j \in \{1, \ldots s\}
$$
where
$\delta = \sqrt{n}(\hat\psi - \psi)$ and
$\Lambda_j = \int_0^1 H_j( (1-t)\psi + t \hat\psi) dt \in \mathbb{R}^{b \times b}$.
Hence,
\begin{equation}\label{eq::taylor}
\sqrt{n}(\hat\theta - \theta) = \sqrt{n}(\hat\nu - \nu) + R
\end{equation}
where
$\nu = G\psi$,
$\hat\nu = G \hat\psi$ and $R$ is a random vector in $\mathbb{R}^s$ whose
$j^{\mathrm{th}}$ coordinate is
$$
R_j = \frac{1}{2\sqrt{n}} \delta^\top
\left[ \int_0^1 H_j( (1-t)\psi + t \hat\psi) dt \right]
\delta.
$$
By Lemma \ref{lem:hyper} below, there exists a constant $C>0$, depending on $A$
only, such that 
\begin{equation}\label{eq::CLT}
\sup_{P\in {\cal P}_n}
\sup_t
\Bigl|\mathbb{P}(\sqrt{n}||\hat\nu - \nu||_\infty \leq t) -
\mathbb{P}(||Z_n||_\infty \leq t)\Bigr| \leq C
\frac{1}{\sqrt{v}} \left( \frac{ \overline{v}^2 b (\log 2bn)^7}{n}
\right)^{1/6},
\end{equation}
where $Z_n \sim N_s(0,\Gamma)$.

Now we bound the effect of remainder $R$ 
in (\ref{eq::taylor}).
First, by assumption (see Equation \ref{eq:H.and.B}), we have that, almost everywhere,
\begin{equation}\label{eq:bound.H}
\sup_{u \in [0,1]}    \|  H_j( (1-u)\psi + u \hat\psi) \|_{\mathrm{op}} \leq
\overline{H},
\end{equation}
from which  it is follows that 
$$
\| R \|_\infty \leq \frac{\overline{H} ||\delta||^2}{2\sqrt{n}},
$$
with the  inequality holding uniformly in $\mathcal{P}_n$.
Next, consider the event
$\mathcal{E}_n = \Bigl\{ \frac{\overline{H} ||\delta||^2}{2\sqrt{n}} < \epsilon_n\Bigr\}$
where
\begin{equation}\label{eq:epsilon}
    \epsilon_n = C \sqrt{\frac{b \overline{v} \overline{H}^2 (\log n)^2}{n}},
\end{equation}
for
 a sufficiently large, positive constant $C$ to be specified later.
 Thus, since $\delta = \sqrt{n} (\hat{\psi}- \psi)$, we have that 
 \begin{align}
\nonumber
\mathbb{P}(\mathcal{E}_n^c) &= \mathbb{P}\left( \frac{\overline{H}
||\delta||^2}{2\sqrt{n}} > \epsilon_n\right)\\
\nonumber
    & =  \mathbb{P}\left( ||\hat{\psi} - \psi || >
    \sqrt{ \frac{2 \epsilon_n}{ \sqrt{n} \overline{H}}}\right)\\
    \nonumber
    & =  \mathbb{P}\left( ||\hat{\psi} - \psi || >
    C \sqrt{ \overline{v} b \frac{\log n}{n} } \right)\\
    \label{eq:Ac}
    & \leq \frac{1}{n},
 \end{align}
 where in the third identity we have used the definition of $\epsilon_n$ in
 \eqref{eq:epsilon} and the final inequality  inequality follows from the vector
 Bernstein inequality \eqref{eq:vector.bernstein} and by taking the constant $C$
 in \eqref{eq:epsilon} appropriately large. 
In fact, the bound on the probability of the event $\mathcal{E}_n^c$ 
holds uniformly over all $P \in \mathcal{P}_n$.

Next, for any $t > 0$ and uniformly in $P \in \mathcal{P}_n$,
\begin{align}
    \mathbb{P}( \sqrt{n}||\hat\theta - \theta||_\infty \leq t) &=
    \mathbb{P}( \sqrt{n}||\hat\theta - \theta||_\infty \leq t,\ \mathcal{E}_n) + 
    \mathbb{P}( \sqrt{n}||\hat\theta - \theta||_\infty \leq t,\ \mathcal{E}_n^c)
\nonumber \\
& \leq
\mathbb{P}( \sqrt{n}||\hat\nu - \nu||_\infty \leq t+\epsilon_n) + 
\mathbb{P}(\mathcal{E}_n^c) \nonumber \\
& =
\mathbb{P}( ||Z_n||_\infty \leq t+\epsilon_n) + 
C \frac{1}{\sqrt{v}} \left( \frac{\overline{v}^2 b (\log 2bn)^7}{n}
\right)^{1/6} + 
\mathbb{P}(\mathcal{E}_n^c)
\label{eq:sorryrogeryouaretigernow}
\end{align}
where the inequality follows from \eqref{eq::taylor} and the fact that $\|
R \|_\infty \leq \epsilon_n $ on the event $\mathcal{E}_n$ and the second identity
from the Berry-Esseen bound (\ref{eq::CLT}). 
By the Gaussian anti-concentration inequality of \Cref{thm:anti.concentration},
\[
    \mathbb{P}( ||Z_n||_\infty \leq t + \epsilon_n )\leq \mathbb{P}( ||Z_n||_\infty
\leq t) + \frac{\epsilon_n}{\underline{\sigma}} (\sqrt{2 \log b}  +2). 
\]
Using the previous inequality on the first term of
\eqref{eq:sorryrogeryouaretigernow}, we obtain that
\begin{align*}
\mathbb{P}( \sqrt{n}||\hat\theta - \theta||_\infty \leq t)& \leq
\mathbb{P}( ||Z_n||_\infty \leq t) + C \left[ \frac{\epsilon_n}{\underline{\sigma}} (\sqrt{2 \log b}  +2) +
    \frac{1}{\sqrt{v}} \left( \frac{\overline{v}^2 b (\log 2bn)^7}{n}
\right)^{1/6} 
\right] +
\mathbb{P}(\mathcal{E}_n^c)\\
& \leq
\mathbb{P}( ||Z_n||_\infty \leq t) + C \left [\frac{\epsilon_n}{\underline{\sigma}} (\sqrt{2 \log b}  +2)  +
    \frac{1}{\sqrt{v}} \left( \frac{\overline{v}^2 b (\log 2bn)^7}{n}
\right)^{1/6} 
  \right],
\end{align*}
where in the second inequality we have used the fact that
$\mathbb{P}(\mathcal{E}^c_n) \leq \frac{1}{n}$ by \eqref{eq:Ac} and have
absorbed this lower order term into higher order terms by increasing the
value of $C$.
By a symmetric argument, we have
$$
\mathbb{P}( \sqrt{n}||\hat\theta -  \theta||_\infty \leq t) \geq
\mathbb{P}( ||Z_n||_\infty \leq t) -C \left [\frac{\epsilon_n}{\underline{\sigma}} (\sqrt{2 \log b}  +2)  +
    \frac{1}{\sqrt{v}} \left( \frac{\overline{v}^2 b (\log 2bn)^7}{n}
\right)^{1/6} 
  \right].
$$
The result now follows by bounding $\epsilon_n$ as in \eqref{eq:epsilon}. $\Box$

The following lemma
shows that the linear term $\sqrt{n}(\hat\nu - \nu)$ in \eqref{eq::taylor} has a Gaussian-like
behavior and is key ingredient of our results. It is an application of the
Berry-Esseen \Cref{thm:high.dim.clt}, due to \cite{cherno2}. The proof is in 
\Cref{appendix:auxilary}. 

\begin{lemma}\label{lem:hyper}
There exists a constant $C>0$, depending on $A$ only, such  that 
\begin{equation}
\sup_{P\in {\cal P}}
\sup_t
\Bigl|\mathbb{P}(\sqrt{n}||\hat\nu - \nu||_\infty \leq t) -
\mathbb{P}(||Z_n||_\infty \leq t)\Bigr| \leq C
\frac{1}{\sqrt{v}} \left( \frac{\overline{v}^2 b (\log 2bn)^7}{n}
\right)^{1/6},
\end{equation}
where $Z_n \sim N_s(0,\Gamma)$.
\end{lemma}

{\bf Proof of Lemma \ref{lemma::upsilon}.}
Throughout the proof, we set $G = G(\psi)$, where $\psi = \psi(P)$ for some $P \in \mathcal{P}_n$, and $\hat{G} = G(\hat{\psi})$ where $\hat{\psi} = \hat{\psi}(P)$ is the sample average from an i.i.d. sample from $P$. Recall that the matrices $\Gamma$ and $\hat{\Gamma}$ are given in Equations \eqref{eq:Gamma} and \eqref{eq:hat.gamma.berry}, respectively. For convenience we will suppress the dependence of $\hat{\Gamma}$ and $\hat{G}$, and of $\Gamma$ and $G$ on $\hat{\psi}$ and $\psi$, respectively.

Express $\hat\Gamma - \Gamma$   as
\begin{align*}
(\hat G - G) V G^\top + G V (\hat G - G)^\top + &  (\hat G - G) V (\hat G - G)^\top
+\\
(\hat G - G)(\hat{V} - V) G^\top + &  G (\hat{V} - V) ( \hat{G}- G)^\top   + G (\hat
V - V)G^\top + (\hat{G} - G) (\hat{V} - V)^\top (\hat{G} - G )^\top.
\end{align*}
The first,  second and sixth terms are dominant, so it will be enough to compute 
high-probability bounds for 
 $(\hat G - G) V G^\top$ and $G (\hat
V - V)G^\top$. 

We first bound $(\hat G - G) V G^\top$.
 For any $j$ and $l$ in $\{1,\ldots,s\}$ and using the Cauchy-Schwartz
 inequality, we have that
\begin{equation}\label{eq::aaa}
   \left| \left( \hat{G}_j - G_j  \right) V G_l^\top \right| \leq \lambda_{\max}(V) \|
   \hat{G}_j - G_j\| B \leq \overline{v} B \| \hat{G}_j - G_j\|,
\end{equation}
by the definition of $B$ (see Equation \ref{eq:H.and.B}), where we recall that $G_j$ denotes the $j^{\mathrm{th}}$ row of $G$.

It remains to bound the stochastic term $\max _j \|
    \hat{G}_j - G_j\|$. Towards that end, we will show that, for some constant  $C$ dependent on $A$ only,
    \begin{equation}\label{eq:hatGjmGj}
    \mathbb{P} \left( \max_j \|\hat G_j - G_j\|\leq C \overline{H} \sqrt{b
    	\frac{  \log
	n}{ n}} \right) \geq 1 - 1/n.
\end{equation}
Indeed, by a Taylor expansion, 
\begin{align*}
    \hat G_j - G_j =(\hat\psi -\psi)^\top \int_0^1 H_j((1-t)\psi +t\hat{\psi})dt
    \quad \text{for all } j \in \{1,\ldots,s\},
\end{align*}
so that
\[
   \max_j \|\hat G_j - G_j \| \leq \| \psi - \hat{\psi} \| \max_j  \Big \|  \int_0^1 H_j((1-t)\psi
    +t\hat{\psi})dt \Big\|_{\mathrm{op}}. 
\]
Since the coordinates of $\hat{\psi}$ are bounded in absolute value by $A$,
the bound \eqref{eq:vector.bernstein.simple} implies that,  
 for some positive
constant $C$ dependent on $A$ only,
$ \mathbb{P} \left( \|\hat\psi-\psi\| \leq C  \sqrt{b (\log n)/n})
\right) \geq 1 - 1/n$,
for all $P\in {\cal P}_n^{\mathrm{OLS}}$.
We restrict to this event.
By convexity of the operator norm 
$||\cdot ||_{\rm op}$
and our assumption, we have that 
\begin{equation}\label{eq:here}
\max_j   \Biggl|\Biggl|\int_0^1 H_j((1-t)\psi +t \hat{\psi})dt\Biggr|\Biggr|_{\mathrm{op}}\le \overline{H},
\end{equation}
yielding the bound in \eqref{eq:hatGjmGj}.
Combined with (\ref{eq::aaa}), we conclude that on an event of probability at least $1 - 1/n$,
$\max_{j,l} |\hat\Gamma(j,l) - \Gamma(j,l)|\preceq   \aleph_n$. This bound holds
uniformly over $P \in \mathcal{P}_n$. 

As for the other term $G (\hat
V - V)G^\top$, we have that, by 
\eqref{eq:matrix.bernstein.simple.2} in \Cref{lem:operator},
\[
    \max_{j,l} \left| G_j (\hat
V - V)G_l^\top \right| \leq B^2 \| \hat{V} - V \|_{\mathrm{op}} \leq C B^2 
\sqrt{  b \overline{v} \frac{ \log b + \log n }{n}   },
\]
with probability at least $ 1- \frac{1}{n}$, where $C$ depends only on $A$ and we
have used the fact that $\max_j \| G_j (\psi(P))\|^2 \leq B^2$ uniformly over $P
\in\mathcal{P}_n$.

Thus, by a union bound, the claim holds on an event of probability at least $1- \frac{2}{n}$.
$\Box$

{\bf Proof of Theorem \ref{thm::coverage}.}
Let 
$Z_n \sim N(0,\Gamma)$ and recall that $\hat{Z}_n \sim N(0,\hat{\Gamma})$.
Using the triangle inequality, we have that 
$$
\mathbb{P}(\theta \in \hat{C}_n) =
\mathbb{P}(\sqrt{n}||\hat\theta - \theta||_\infty \leq  \hat{t}_\alpha) \geq
\mathbb{P}(||\hat Z_n||_\infty \leq \hat{t}_\alpha) - A_1 - A_2,
$$
where
$$
A_1 = 
\sup_{t > 0} | \mathbb{P}(\sqrt{n}||\hat\theta - \theta||_\infty \leq t) -
\mathbb{P}(||Z_n||_\infty \leq t) |
$$
and
$$
A_2 = \sup_{t > 0} | \mathbb{P}(||Z_n||_\infty \leq t) - \mathbb{P}( ||\hat
Z_n||_\infty \leq t)|.
$$
Now 
$$
\mathbb{P}(||\hat Z_n||_\infty \leq \hat{t}_\alpha) 
=\mathbb{E}[\mathbb{P}(||\hat Z_n||_\infty \leq \hat{t}_\alpha |
\hat\Gamma)] = 1-\alpha,
$$
by the definition of $\hat{t}_\alpha$.
\Cref{theorem::deltamethod} implies that
$A_1 \leq C ( \Delta_{1,n} + \Delta_{2,n}) $, where $C$ depends on $A$ only.
To bound $A_2$, consider the event 
$\mathcal{E}_n= \{ \max_{j,k} |\widehat{\Gamma} - \Gamma| \leq C \aleph_n\}$, where the constant $C$ is the same as in \Cref{lemma::upsilon}.
Then, by the same Lemma, $\mathbb{P}(\mathcal{E}_n) \geq 1 - 1/n$,
uniformly over all $P$ in ${\cal P}_n$. 
Next, we have that
$$
A_2 \leq
\mathbb{E}\left[ \sup_{t > 0} \left|
\mathbb{P}(||Z_n||_\infty \leq t) - \mathbb{P}( ||\hat Z_n||_\infty \leq
t|\hat\Gamma)\right|; \mathcal{E}_n \right] + \mathbb{P}(\mathcal{E}_n^c),
$$
where $\mathbb{E}[\cdot;\mathcal{E}_n]$ denotes expectation restricted to the
event $\mathcal{E}_n$.
By the Gaussian comparison \Cref{thm:comparisons}
the term inside the expected value is bounded by $\Delta_{n,3}$. $\Box$

{\bf Proof of \Cref{thm::bonf}.}
For $j=1,\ldots,s$, let $\gamma_j =\sqrt{\Gamma_{j,j}}$, $\hat{\gamma}_j = \sqrt{ \hat{\Gamma}_{j,j}}$ and
$\hat{t}_j = z_{\alpha/(2s)} \hat{\gamma}_j$
We use the same arguments and notation as in the proofs of
\Cref{theorem::deltamethod} and \Cref{lem:hyper}.
Thus, let $\mathcal{E}_n$ be the event that
$ \frac{\overline{H} ||\delta||^2}{2\sqrt{n}} < \epsilon_n$,
where $\frac{\overline{H} ||\delta||^2}{2\sqrt{n}}$ is an upper bound on $\|
R\|_\infty$, with $R$ the reminder in the Taylor series expansion
\Cref{eq::taylor} and $\epsilon_n$ as in \Cref{eq:epsilon}.  Then,
$\mathbb{P}\left( \mathcal{E}_n^c \right) \leq n^{-1}$ (see equation \ref{eq:Ac}).  

Next, for each $t \in \mathbb{R}^{2s}_+$ and any Jacobian matrix $G = G(\psi(P))$,
with $P \in \mathcal{P}_n$, let
 \begin{equation}\label{eq:polyhedron2}
    P(G,t) = \left\{  x \in \mathbb{R}^b \colon  v_l^\top x  \leq t_l , \forall
    v_l \in \mathcal{V}(G) \right\},
\end{equation}
where $\mathcal{V}(G)$ is defined in the proof of \Cref{lem:hyper}.
Then, for any positive numbers $(t'_1,\ldots,t'_s)$
 \[
     |     \sqrt{n}(\hat{\nu}_j - \nu_j ) | \leq t'_j, j=1,\ldots,s \quad
 \text{if and only if } \quad 
\sqrt{n} (\hat{\psi} - \psi) \in P(G,t),
 \]
 where the coordinates of $t \in \mathbb{R}^{2s}$ are as follows: for
 $j=1,\ldots,s$, $t_{2l-1} = t_{2l} = \frac{t'_l}{\|G_j\|}$. 

 Consider now the class of subsets of $\mathbb{R}^b$ of the
form specified in \eqref{eq:polyhedron2}, where $t$ ranges over the positive
vectors in $\mathbb{R}^{2s}$ and $G$ ranges  in $ \{  G(\psi(P)),  P \in
    \mathcal{P}_n\}$. 
This is a class comprised by polytopes with at most $2s$ faces in
$\mathbb{R}^b$. Thus, using the same arguments as in the proof of
\Cref{lem:hyper}, we obtain that
\begin{equation}\label{eq:delta1n.bonf}
    \sup_{t =(t_1,\ldots,t_s) \in \mathbb{R}^s_+}   \left|
    \mathbb{P}\left(\sqrt{n}|\hat\nu_j - \nu_j| \leq t_j, \forall j\right) -
  \mathbb{P}\left(|Z_{n,j} | \leq t_j, \forall j\right) \right| \leq 
  C \frac{1}{\sqrt{v}} \left( \frac{\overline{v}^2 b (\log 2bn)^7}{n}
\right)^{1/6},  
\end{equation}
for some $C>0$ depending only on $A$, where $Z_n \sim N(0,\Gamma)$.
Using the above display, and following the same arguments as in the proof of
\Cref{theorem::deltamethod}, we have that
\begin{align*}
    \mathbb{P}( \sqrt{n} |\hat\theta_j - \theta_j| \leq \hat{t}_j, \forall j) &=
    \mathbb{P}(\sqrt{n}|\hat\theta_j - \theta_j| \leq \hat{t}_j, \forall j;\
    \mathcal{E}_n) + 
    \mathbb{P}( \sqrt{n} |\hat\theta_j - \theta_j| \leq \hat{t}_j, \forall j;\
    \mathcal{E}_n^c)
 \\
& \leq
\mathbb{P}( \sqrt{n}|\hat\nu_j - \nu_j| \leq \hat{t}_j+\epsilon_n, \forall j) + 
\mathbb{P}(\mathcal{E}_n^c) \nonumber \\
& \leq
\mathbb{P}( |Z_{n,j}| \leq \hat{t}_j+\epsilon_n, \forall j) + 
C \frac{1}{\sqrt{v}} \left( \frac{\overline{v}^2 b (\log 2bn)^7}{n}
\right)^{1/6} + 
\frac{1}{n}
\\
& \leq \mathbb{P}( |Z_{n,j}| \leq \hat{t}_j, \forall j) + C \left[
\frac{\epsilon_n}{\underline{\sigma}} (\sqrt{2 \log b}  +2)  +
\frac{1}{\sqrt{v}} \left( \frac{\overline{v}^2 b (\log 2bn)^7}{n}
\right)^{1/6}  \right],
\end{align*}
where in the second-to-last inequality we have used the fact that
$\mathbb{P}(\mathcal{E}^c_n) \leq \frac{1}{n}$ and in the last inequality we
have applied the Gaussian anti-concentration
inequality in
\Cref{thm:anti.concentration} (and have absorbed the term
$\frac{1}{n}$ into higher order terms by increasing the value of $C$).
A similar argument gives
\[
    \mathbb{P}( \sqrt{n} |\hat\theta_j - \theta_j| \leq \hat{t}_j, \forall
    j)\geq \mathbb{P}( |Z_{n,j}| \leq \hat{t}_j, \forall j) - C \left[
\frac{\epsilon_n}{\underline{\sigma}} (\sqrt{2 \log b}  +2)  +
\frac{1}{\sqrt{v}} \left( \frac{\overline{v}^2 b (\log 2bn)^7}{n}
\right)^{1/6} \right].
\]
To complete the proof, we will show that
\begin{equation}\label{eq:min}
\mathbb{P}( |Z_{n,j}| \leq z_{\alpha/(2s)} \hat{\gamma}_j, \forall j) \geq
(1-\alpha) - \frac{1}{n} - \min \left\{ C \Delta_{3,n},  \frac{ \aleph_n z_{\alpha/(2s)}}{(\min_j
\gamma_j)^2} \left(\sqrt{ 2 + \log(2s ) } + 2 \right)\right\}. 
\end{equation}
Let $\hat{Z}_n \sim N(0,\hat \Gamma)$. By the Gaussian comparison \Cref{thm:comparisons}, 
\[
\mathbb{P}( |Z_{n,j}| \leq z_{\alpha/(2s)} \hat{\gamma}_j, \forall j) \geq
\mathbb{P}( |\hat{Z}_{n,j}| \leq z_{\alpha/(2s)} \hat{\gamma}_j, \forall j) - I
\geq 1 - \alpha - I
\]
where
\[
I \leq 
\mathbb{E}\left[ \sup_{t = (t_1,\ldots,t_s) \in \mathbb{R}^s_+} \left|
    \mathbb{P}(|Z_{j,n} |  \leq t_j, \forall j) - \mathbb{P}( |\hat Z_{j,n}| \leq
t, \forall j |\hat\Gamma)\right|; \mathcal{F}_n \right] +
\mathbb{P}(\mathcal{F}_n^c) \leq C \Delta_{3,n} + \frac{1}{n}.
\]
In the above expression the constant $C$ is the same as in \Cref{lemma::upsilon} and
$\mathcal{F}_n$ is the event that $\{ \max_{j,k}
|\widehat{\Gamma} - \Gamma| \leq C \aleph_n\}$, which is of probability at least $ 1-
\frac{1}{n}$, again by \Cref{lemma::upsilon}. This gives the first bound in \eqref{eq:min}.

To prove the second bound in \eqref{eq:min} we let
$\Xi_n  = C \frac{\aleph_n}{\min_j \gamma_j}$, where $C$ is the constant in
\Cref{lemma::upsilon}, and then notice that, on the event 
$\mathcal{F}_n$, 
\[
    |\hat\gamma_j - \gamma_j| =
    \frac{|\hat\gamma_j^2 - \gamma_j^2|} {|\hat\gamma_j + \gamma_j|}
     \leq
    \frac{|\hat\gamma_j^2 - \gamma_j^2|} {\gamma_j} \leq
    \frac{\max_j |\hat\gamma_j^2 - \gamma_j^2|}{\min_j \gamma_j} \leq \Xi_n.
\] 
Thus,
\begin{align*}
\mathbb{P}\left( |Z_{n,j}| \leq z_{\alpha/(2s)} \hat{\gamma}_j, \forall j\right)
& = \mathbb{P}\left( |Z_{n,j}| \leq z_{\alpha/(2s)} \hat{\gamma}_j, \forall
j\right) -\mathbb{P}\left( |Z_{n,j}| \leq z_{\alpha/(2s)} \gamma_j,
\forall j\right) + \mathbb{P}\left( |Z_{n,j}| \leq z_{\alpha/(2s)} \gamma_j, \forall j\right) \\
& \geq 
\mathbb{P} \left( |Z_{n,j}| \leq z_{\alpha/(2s)} \hat{\gamma}_j, \forall
j;\mathcal{F}_n\right)  -\mathbb{P}\left( |Z_{n,j}| \leq z_{\alpha/(2s)} \gamma_j,
\forall j\right) + \mathbb{P}\left( |Z_{n,j}| \leq z_{\alpha/(2s)} \gamma_j,
\forall j\right)\\
& \geq  \mathbb{P} \left( |Z_{n,j}| \leq z_{\alpha/(2s)} \hat{\gamma}_j, \forall
j;\mathcal{F}_n\right)  -\mathbb{P}\left( |Z_{n,j}| \leq z_{\alpha/(2s)} \gamma_j,
\forall j\right) + (1-\alpha),
    \end{align*}
    where in the last step we have used the union bound.
    Next,
    \[
\mathbb{P} \left( |Z_{n,j}| \leq z_{\alpha/(2s)} \hat{\gamma}_j, \forall
j;\mathcal{F}_n\right)  \geq \mathbb{P} \left( |Z_{n,j}| \leq z_{\alpha/(2s)}
(\gamma_j - \Xi_n), \forall
j; \mathcal{F}_n \right) \geq \mathbb{P} \left( |Z_{n,j}| \leq z_{\alpha/(2s)}
(\gamma_j - \Xi_n), \forall
j \right)  - \mathbb{P}\left( \mathcal{F}_n^c \right).
    \]
    Thus,
    \begin{align*}
\mathbb{P}\left( |Z_{n,j}| \leq z_{\alpha/(2s)} \hat{\gamma}_j, \forall j\right)
& \geq (1-\alpha)- \mathbb{P}\left( \mathcal{F}_n^c \right)+ \mathbb{P}\left(
|Z_{n,j}| \leq z_{\alpha/(2s)} (\gamma_j - \Xi_n), \forall
j \right) -\mathbb{P}\left( |Z_{n,j}| \leq z_{\alpha/(2s)} \gamma_j,
\forall j\right)\\ 
& \geq (1-\alpha) - \frac{1}{n} - \frac{ \Xi_n z_{\alpha/(2s)}}{\min_j
\gamma_j} \left(\sqrt{ 2 + \log(2s ) } + 2 \right), 
    \end{align*}
    since, by
    the Gaussian anti-concentration inequality of \Cref{thm:anti.concentration},
\[
\mathbb{P}\left( |Z_{n,j}| \leq z_{\alpha/(2s)} (\gamma_j - \Xi_n), \forall
j \right) -\mathbb{P}\left( |Z_{n,j}| \leq z_{\alpha/(2s)} \gamma_j,
\forall j\right) \geq - \frac{ \Xi_n z_{\alpha/(2s)}}{\min_j
\gamma_j} \left(\sqrt{ 2 + \log(2s ) } + 2 \right). 
\]

The result follows by combining all the above bounds and the fact that
$\underline{\sigma}^2 = \min_{P \in \mathcal{P}_n} \min_j \Gamma(j,j)$. As
usual, we have absorbed any lower order term (namely $\frac{1}{n}$) into higher
order ones.
$\Box$

\vspace{11pt}

{\bf Proof of \Cref{theorem::boot}.}
Let $Z_n \sim N(0,\Gamma)$ where $\Gamma = G V G^\top$ and 
$\hat Z_n \sim N(0,\hat\Gamma)$
where we recall that 
$\hat\Gamma = \hat G \hat V \hat G^\top$, $\hat G = G(\hat \psi)$ and
$\hat V = n^{-1}\sum_{i=1}^n (W_i - \hat\psi)(W_i - \hat\psi)^\top$.
Take $\mathcal{E}_n$ to be the event that 
$$
\left\{ \max_{j,k} |\widehat{\Gamma} - \Gamma| \leq C \aleph_n  \right\} \cap
\left\{     \| V - \hat{V} \|_{\mathrm{op}} \leq C  \daleth_n
\right\},
$$
where $C$ is the larger of the two constants in \eqref{eq:upsilon} and in
\eqref{eq:matrix.bernstein.simple.2}. 
Then, by  \Cref{lemma::upsilon} and \Cref{lem:operator}, $\mathbb{P}\left(  \mathcal{E}_n \right) \geq 1-2/n$, uniformly over all the distributions in ${\cal P}_n$.
By the triangle inequality, 
\begin{equation}\label{eq:F.boot}
\mathbb{P}(\theta \in \hat{C}^*_n) =
\mathbb{P}(\sqrt{n}||\hat\theta - \theta||_\infty \leq  \hat{t}^*_\alpha) \geq
\mathbb{P}( \sqrt{n}||\hat \theta^* - \hat{\theta}||_\infty \leq
\hat{t}^*_\alpha|(W_1,\ldots,W_n)) - (A_1 + A_2 + A_3),
\end{equation}
where
\begin{align*}
    A_1 & = \sup_{t >0} \left| \mathbb{P}\left( \sqrt{n} \| \hat{\theta} -
    \theta   \|_\infty
    \leq t \right) - \mathbb{P}( \| Z_n \|_\infty \leq t ) \right|,\\
    A_2 & = \sup_{t>0} \left|  \mathbb{P}( \| Z_n \|_\infty \leq t ) - \mathbb{P}( \|
    \hat{Z}_n \|_\infty \leq t )      \right|,\\
    \text{and} & \\
    A_3 & = \sup_{t >0} \left|  \mathbb{P}( \| \hat{Z}_n \|_\infty \leq
	t -  \mathbb{P}\left( \sqrt{n} \| \hat{\theta}^* - \hat{\theta}   \|_\infty \leq t
   \Big| (W_1,\ldots,W_n) \right)  \right|.
\end{align*}
Since, by definition, $\mathbb{P}( \sqrt{n}||\hat \theta^* - \hat{\theta}||_\infty \leq
\hat{t}^*_\alpha|(W_1,\ldots,W_n))   \geq 1 - \alpha$, it follows from
\eqref{eq:F.boot} that, in order to establish  (\ref{eq::boot-cov}) we will need to upper bound
each of the terms $A_1$,
$A_2$ and $A_3$ accordingly. 
The term $A_1$ has already been bounded by $C( \Delta_{1,n} + \Delta_{2,n} )$ in the earlier
\Cref{theorem::deltamethod}.
For $A_2$ we use the Gaussian comparison \Cref{thm:comparisons} as in the proof of
\Cref{thm::coverage} restricted to the event $\mathcal{E}_n$ to conclude that
$A_2 \leq C \Delta_{n,3} + \frac{2}{n}$.
Finally, to bound $A_3$,
 one can apply the same arguments as in Theorem
\ref{theorem::deltamethod}, but restricted to the event $\mathcal{E}_n$, to the larger class of
probability distributions
$\mathcal{P}^*_n$ differing from $\mathcal{P}_n$ only in the fact that $v$ is replaced by the
smaller quantity $v_n > 0$ and $\overline{v}$ by the larger quantity $\overline{v}_n  =\overline{v} + C \daleth_n$.
In particular, the bootstrap distribution belongs to $\mathcal{P}^*_n$. In detail, one can replace $\psi$ and  with $\hat\psi$, 
and $\hat\psi$ with $\hat\psi^*$ and, similarly,  $\Gamma$ with
$\hat{\Gamma}$ and $\hat{\Gamma}$ with $\hat{\Gamma}^* = G(\hat{\psi}^*)
\hat{V}^* G(\hat{\psi}^*)^\top$, where $\hat{V}^*$ is the empirical covariance
matrix based on a sample of size $n$ from the bootstrap distribution.
The assumption that $n$ is large enough so that $v_n$ and $\sigma^2_n$ are
positive ensures  that, on the event $\mathcal{E}_n$ of probability at
least $1-2/n$, $\min_j
\sqrt{\hat{\Gamma}(j,j)} > \sqrt{\underline{\sigma}^2 - C \aleph_n} >0$ and, by Weyl's inequality,
the minimal eigenvalue of $\hat{V}$ is no smaller than $v - C
\daleth_n > 0$. In particular, the error terms $\Delta^*_{n,1}$ and $\Delta^*_{n,2}$  are
well-defined (i.e. positive). 
Thus we have that 
\begin{equation}\label{eq:A3}
A_3 \leq C \left( \Delta^*_{n,1} + \Delta^*_{n,2} \right) +
\frac{2}{n},
 \end{equation}
 where the lower order term $\frac{1}{n}$ is reported to account for the
 restriction to the event $\mathcal{E}_n$. 
The result now follows by combining all the bounds, after noting that
$\Delta_{1,n} \leq \Delta^*_{1,n}$ and $\Delta_{2,n} \leq \Delta^*_{2,n}$.
 
To show  that the same bound holds for the  coverage of $\tilde{C}^*_\alpha$ we
proceed in a similar manner. 
Using the triangle inequality, and uniformly over all the distributions in
${\cal P}_n$,  
\begin{align*}
\mathbb{P}(\theta \in \tilde{C}^*_n)  & =
\mathbb{P}(\sqrt{n} |\hat\theta_j - \theta _j| \leq  \tilde{t}^*_{j,\alpha},
\forall j)\\
& \geq
\mathbb{P}\left( \sqrt{n} |\hat \theta^*_j - \hat{\theta}_j | \leq
\tilde{t}^*_{j,\alpha}, \forall j \Big| (W_1,\ldots,W_n) \right) - (A_1 + A_2 +
A_3)\\
& \geq	 (1 - \alpha) - (A_1 + A_2 +
A_3),
\end{align*}
where
\begin{align*}
    A_1 & = \sup_{t = (t_1,\ldots,t_s) \in \mathbb{R}_+^s} \left| \mathbb{P}\left( \sqrt{n} |
    \hat{\theta}_j-
    \theta_j   |
    \leq t_j, \forall j \right) - \mathbb{P}( |Z_{n,j}| \leq t_j, \forall j ) \right|,\\
    A_2 & = \sup_{t = (t_1,\ldots,t_s) \in \mathbb{R}_+^s } \left|  \mathbb{P}(
    | Z_{n,j} | \leq t_j, \forall j ) - \mathbb{P}( |
    \hat{Z}_{n,j} | \leq t_j, \forall j )      \right|,\\
    \text{and} & \\
    A_3 & = \sup_{t = (t_1,\ldots,t_s) \in \mathbb{R}_+^s   } \Big|  \mathbb{P}(
    | \hat{Z}_{n,j} | \leq
	t_j, \forall j) -  \mathbb{P}\left( \sqrt{n} | \hat{\theta}_j^* -
	\hat{\theta}_j   | \leq t_j, \forall j
   \Big| (W_1,\ldots,W_n) \right)  \Big|.
\end{align*}
The term $A_1$ is bounded by $C (\Delta_{1,n} + \Delta_{2,n})$, as shown in the
first part of the
proof of  \Cref{thm::bonf}.  The Gaussian comparison \Cref{thm:comparisons}
yields that
$A_2 \leq C \Delta_{n,3} + \frac{2}{n}$. To bound the term $A_3$, we repeat the
arguments  used in the first part of the proof of \Cref{thm::bonf}, applied
to the larger class $\mathcal{P}_n^*$ and restricting to the event $\mathcal{E}_n$. As argued above, we will replace $\psi$
with $\hat\psi$ and 
$\hat\psi$ with $\hat\psi^*$ and, similarly, $\Gamma$ with
$\hat{\Gamma}$ and $\hat{\Gamma}$ with $\hat{\Gamma}^*$. The assumption that  $n$ is large enough guarantees that, with
probability at least $1 - \frac{2}{n}$, both $v_n$ and $\sigma^2_n$ are positive.
Thus, the right hand side of \eqref{eq:A3} serves as an upper bound for the
current term $A_3$ as well. The claimed bound \eqref{eq::boot-cov.bonf}  then
follows.
$\Box$

\section{Appendix 5: Proofs of Auxiliary Results}
\label{appendix:auxilary}

{\bf Proof of \Cref{eq:lem.occupancy}.}
Let $Z$ be the number of objects that are not selected.
Then $\mathbb{E}[Z] = n \left( 1 - \frac{1}{n} \right)^n \leq \frac{n}{e}$.
Next, by the bounded difference inequality, 
\[
    \mathbb{P}\left(| Z - \mathbb{E}[Z] | \geq t \right) \leq 2 e^{
	-\frac{t^2}{2 n}},
\]
which implies that
\[
\mathbb{P}\left( Z > n - d \right)   \leq \exp\left\{ - \frac{(n -d
    -n(1-1/n)^n)^2 }{2n} \right\}.
\]
The claim \Cref{eq:occupancy} follows immediately, since $n \geq \frac{d}{2}$ and $\left( 1 - \frac{1}{n} \right)^n \leq
e^{-1}$ for all $n=1,2,\ldots$.

$\Box$

{\bf Proof of Lemma \ref{lem:hyper}.}
Let  $\psi$ be an arbitrary point in $\mathcal{S}_n$ and $G =
G(\psi) \in \mathbb{R}^{s \times b }$ be the corresponding Jacobian. Recall
that, for $j=1,\ldots,s$ the $j^{\mathrm{th}}$ row of $G$ is the transpose of $G_j =
G_j(\psi)$, the gradient of $g_j$ at $\psi$.
Let $\mathcal{V} =
\mathcal{V}(G) = \left\{ v_1,\ldots,v_{2s} \right\}$, where for
$j=1,2,\ldots,s$, we define $v_{2j-1} =
\frac{G_j}{\| G_j \|}$ and $v_{2j} = -\frac{G_j}{\| G_j\|}$.
For a given $t>0$ and for any Jacobian matrix $G = G(\psi)$, set
\begin{equation}\label{eq:polyhedron}
P(G,t) =     \left\{ x \in \mathbb{R}^b \colon  v_l^\top x  \leq t_l , \forall v_l \in \mathcal{V}(G) \right\},
\end{equation}
where, for $j=1,\ldots,s$, $t_{2j-1} = t_{2j} = \frac{t}{\|G_j\|}$.

Recalling that $\widehat{\nu} = G \widehat{\psi}$, we have that
$$
\left\|     \sqrt{n}(\hat{\nu} - \nu ) \right\|_\infty \leq t \quad
\text{if and only if } \quad 
\sqrt{n} (\hat{\psi} - \psi) \in P(G,t).
$$
Similarly, if $\tilde{Z}_n \sim
N_b(0,V)$ and $Z_n  = G \tilde{Z}_n \sim N_s(0,\Gamma)$
$$
\| Z_n \|_\infty \leq t\quad
\text{if and only if } \quad \tilde{Z}_n \in P(G,t). 
$$

Now consider the class $\mathcal{A}$ of all subsets of $\mathbb{R}^b$ of the
form specified in \eqref{eq:polyhedron}, where $t$ ranges over the positive
reals and $G$ ranges  in $ \{  G(\psi(P)),  P \in \mathcal{P}\}$. 
Notice that this class is comprised of polytopes with at most $2s$ facets.
Also, from the discussion above, 
\begin{equation}\label{eq:simple.convex}
\sup_{ P \in \mathcal{P}_n} \sup_{t >0}  \left |  \mathbb{P}\left( \| \sqrt{n} (\hat{\nu} - \nu) \|_\infty
\leq t\right)  -\mathbb{P}\left( \| Z_n \|_\infty \leq t  \right) \right|  = 
\sup_{A \in \mathcal{A}} \left |  \mathbb{P}(\sqrt{n} (\hat{\psi} - \psi) \in A)- \mathbb{P}( \tilde{Z}_n \in A) \right|. 
\end{equation}

The claimed result follows from applying  the Berry-Esseen bound for
polyhedral classes,
\Cref{thm:high.dim.clt} in the appendix, due to
\cite{cherno2} to the term on the left hand side of \Cref{eq:simple.convex}. 
To that end, we need to ensure that
conditions (M1'), (M2') and (E1') in that Theorem are satisfied.

For each $i=1,\ldots,n$, set $\tilde{W}_i=
(\tilde{W}_{i,1}, \ldots, \tilde{W}_{i,2s})= \left( (W_i - \psi)^\top v,v \in \mathcal{V}(G) \right)$.
Condition (M1') holds since, for each $l=1,\ldots,2s$,
\[
    \mathbb{E}\left[ \tilde{W}^2_{i,l} \right] \geq \min_l v_l^\top V v_l \geq
    \lambda_{\min}(V),
\]
where $V = \mathrm{Cov}[W]$.
Turning to condition (M2'), we have  that, for for each $l=1,\ldots,2s$ and $k=1,2$, 
\begin{align*}
\mathbb{E}\left[ | \tilde{W}_{i,l}|^{2+k} \right]  & \leq \mathbb{E}\left[ |v_l^\top
(W_i - \psi)|^2  \|W_i - \psi\|^{k} \right]\\
& \leq   \mathbb{E}\left[ |v_l^\top (W_i - \psi)|^2 \right] \left( 2A \sqrt{b}
\right)^k\\
& \leq \overline{v} \left( 2A \sqrt{b} \right)^k,
\end{align*}
where the first inequality follows from the bound $| v^\top_l  (W_i - \psi) |\leq \|
W_i - \psi\|$ (as each $v_l$ is of unit norm), the second from the fact that
the coordinates of $W_i$ are bounded in absolute value by $A$ and the third
by the fact that $\overline{v}$ is the largest eigenvalue of $V$.

Thus we see that by  setting $B_n = \overline{v} \left( 2A \sqrt{b} \right)$,  condition
(M2') is satisfied (here we have used the fact that $\overline{v} \geq 1$). 
Finally, condition (E1') is easily satisfied, possibly by increasing the
constant in the term  $B_n$.

Thus, \Cref{thm:high.dim.clt} gives
\[
              \sup_{A \in \mathcal{A}} \left |  \mathbb{P}(\sqrt{n} (\hat{\psi}
              - \psi)
                  \in A)
                      - \mathbb{P}( \tilde{Z}_n \in A) \right| 
                           \leq C \frac{1}{\sqrt{\lambda_{\min}(V)}} \left(
			   \frac{ \overline{v}^2 b (\log 2bn)^7}{n}   \right)^{1/6},
		      \]
and the result follows from \Cref{eq:simple.convex}, the fact that the choice of
$G = G(\psi)$ is arbitrary
and the fact that $\lambda_{\rm min}(V(P)) \geq v$ for all $P\in {\cal P}_n$, by
assumption. $\Box$

\vspace{11pt}

\begin{lemma}
    \label{lem:operator}
    Let $X_1,\ldots,X_n$ be independent, mean-zero vectors in $\mathbb{R}^p$,
    where $p
    \leq n$, such  that $\max_{i=1\dots,n} \|X_i \|_\infty \leq K$ almost surely for some $K>0$ with common covariance matrix $\Sigma$ with $\lambda_{\max}(\Sigma)
    \leq U$.  
 Then, there exists a universal constant $C>0$ such that
    \begin{equation}\label{eq:vector.bernstein.simple}
	\mathbb{P}\left(  \frac{1}{n} \left\| \sum_{i=1}^n X_i \right\|	\leq C K \sqrt{ p
	    \frac{\log n }{n}} \right) \geq 1 - \frac{1}{n}.
    \end{equation}
    Letting
    $\hat{\Sigma} = \frac{1}{n} \sum_{i=1}^n X_i X_i^\top$, if $U
    \geq \eta > 0$, then there exists a $C>0$, dependent on $\eta$ only, such that
    \begin{equation}\label{eq:matrix.bernstein.simple.2}
\mathbb{P}\left(  \| \widehat{\Sigma} - \Sigma \|_{\mathrm{op}} \leq C K 
\sqrt{p U \frac{ \log p + \log n}{n} } \right) \geq 1 - \frac{1}{n}.
    \end{equation}
\end{lemma}

    \noindent {\bf Proof of \Cref{lem:operator}.}
Since $\| X_i \| \leq K \sqrt{p}$ and $\mathbb{E}\left[ \| X_i \|^2
 \right] \leq U p$ for all $i = 1,\ldots,n$, Proposition 1.2 in \cite{hsu12}
 yields that
 \begin{equation}\label{eq:vector.bernstein}
	\mathbb{P}\left( \frac{1}{n} \left\| \sum_{i=1}^n X_i \right\|	\leq \sqrt{\frac{U p}{n}} + \sqrt{ 8 \frac{U p}{n} \log n} +
	\frac{4 K \sqrt{p}}{3 n} \log n \right) \geq 1 - \frac{1}{n}.
    \end{equation}
Equation \Cref{eq:vector.bernstein.simple} follows by bounding
$\mathbb{E}\left[ \| X_i \|^2 \right]$ with $K^2 p$ instead of $Up$.

Next, we prove \eqref{eq:matrix.bernstein.simple.2}. We let $\preceq$ denote the positive
semi-definite ordering, whereby, for any $p$-dimensional symmetric matrices $A$ and $B$, $A \preceq B$ if and only if $B-A$ is positive
semi-definite. For each $i =1,\ldots,n$, the triangle inequality and the
assumptions in the statement  yield
the bound
\[
    \left\| X_i X_i^\top - \Sigma\right\|_{\mathrm{op}}  \leq \| X_i \|^2 +
    \lambda_{\max}(\Sigma) \leq K^2 p + U.
\]
Similarly, $\| \mathbb{E}\left[ (X_i X_i^\top)^2 \right] - \Sigma\|_{\mathrm{op}} \leq
K^2p U$ for each $i = 1,\ldots, n$, since 
\[
    \mathbb{E}\left[ (X_i X_i^\top)^2 \right] - \Sigma^2 \preceq
    \mathbb{E}\left[ \| X_i \|^2 X_i X_i^\top \right] \preceq K^2 p \Sigma
    \preceq K^2 p U I_{p}.
\]
 with $I_p$ the $p $-dimensional identity matrix. 
 Thus, applying the Matrix Bernstein inequality  \citep[see Theorem
 1.4 in][]{Tropp2012}, we obtain that
    \begin{equation}\label{eq:matrix.bernstein}
	\mathbb{P}\left(  \| \widehat{\Sigma} - \Sigma \|_{\mathrm{op}} \leq \sqrt{ 2 K^2 p U  \frac{\log p + \log	2n }{n}} +
	\frac{2}{3} (K^2 p + U) \frac{\log p + \log 2n }{n}\right)\geq 1 -
	\frac{1}{n}.
    \end{equation}
The bound \eqref{eq:matrix.bernstein.simple.2} follows from choosing $C$ large
enough, depending on $\eta$, and using the fact that $p \leq n$.
    $\Box$

    {\bf Remark.} From \eqref{eq:matrix.bernstein}, by using  
    the looser
    bounds 
\[ \left\| X_i X_i^\top - \Sigma\right\|_{\mathrm{op}}  \leq 2 K^2 p \quad
    \text{and} \quad \mathbb{E}\left[ (X_i X_i^\top)^2 \right] - \Sigma^2
    \preceq K^4 p^2 I_p,
\]
    one can obtain directly that
    \begin{equation}\label{eq:matrix.bernstein.simple}
\mathbb{P}\left(  \| \widehat{\Sigma} - \Sigma \|_{\mathrm{op}} \leq C K^2 p
\sqrt{\frac{ \log p + \log n}{n} } \right) \geq 1 - \frac{1}{n},
    \end{equation}
    for some universal constant $C>0$. 
Clearly, the scaling in $p$ is worse.

{\bf Proof of Lemma \ref{lemma::horrible}.}
Throughout, we drop the dependence on $\wS$ in our notation and assume without
loss of generality that $\wS = \{1,\ldots,k\}$.
We refer the reader to \cite{magnus07} for a comprehensive treatment of matrix calculus techniques.
Recall that
$\psi = \left[ \begin{array}{c} \sigma \\ \alpha\\ \end{array} \right]$ and $\xi = \left[ \begin{array}{c} w\\ \alpha\\
\end{array} \right]$, where 
$\sigma =\mathrm{vec}(\Sigma)$
and $w = \mathrm{vec}(\Omega)$. 
The dimension of both $\psi$ and $\xi$ is $b  = k^2 + k$. 
For $ 1 \leq j \leq n$, let 
\[
\beta_j = g_j(\psi) = e^\top_j \Omega\alpha,
\]
where $e_j$ is the $j^{\mathrm{th}}$ elements of the standard basis in $\mathbb{R}^n$.
Then, we can write
\[
g_j(\psi) = g(f(\psi)),
\]
with
$f(\psi) = \xi \in \mathbb{R}^b$ and
$g(\xi) = e^\top_j \Omega \alpha \in \mathbb{R}$.

Using the chain rule, the derivative  of $g_j(\psi)$ is
\[
 D g_j(\psi) = D g(\xi) D f(\psi) =  e_j^\top \Big[\left( \alpha^\top \otimes I_k
    \right) E + \Omega
    F\Big] 
\left[
    \begin{array}{cc}
	- \Omega \otimes \Omega & 0 \\
	0 & I_k
    \end{array}
\right],
\]
where 
\[
    E = \Big[I_{k^2} \;\;\;\;\; 0_{k^2 \times k}\Big] = \frac{d w}{d \psi} \in
    \mathbb{R}^{k^2 \times b} \quad \text{and} \quad 
    F =
\Big[0_{k \times k^2} \;\;\;\;\; I_k\Big] = \frac{d \alpha}{d \psi} \in
\mathbb{R}^{ k \times b}.
\]
Carrying out the calculations, we have that
\begin{align*}
\left( \alpha^\top \otimes I_k
    \right) E 
\left[
    \begin{array}{cc}
	- \Omega \otimes \Omega & 0 \\
	0 & I_k
    \end{array}
\right] & = 
\left( \alpha^\top \otimes I_k
    \right) \Big[I_{k^2} \;\;\;\;\; 0_{k^2 \times k}\Big]  
\left[
    \begin{array}{cc}
	- \Omega \otimes \Omega & 0 \\
	0 & I_k
    \end{array}
\right] \\
& =  \Big[ - \left( \alpha^\top \otimes I_k
\right) (\Omega \otimes \Omega) \;\;\;\;\; 0_{k \times k}\Big] 
\end{align*}
and
\begin{align*}
\Omega
    F 
\left[
    \begin{array}{cc}
	- \Omega \otimes \Omega & 0 \\
	0 & I_k
    \end{array}
\right] & = 
\Omega
    \Big[0_{k \times k^2} \;\;\;\;\; I_k\Big] 
\left[
    \begin{array}{cc}
	- \Omega \otimes \Omega & 0 \\
	0 & I_k
    \end{array}
\right]  \\
& = \Omega \Big[ 0_{k \times k^2} \;\;\;\;\; I_k \Big] =
\Big[ 0_{k \times k^2} \;\;\;\;\; \Omega \Big] .
\end{align*}
Plugging the last two expressions into the initial formula for $D
g_j(\psi)$ we obtain that
\begin{align}
\nonumber
D g_j(\psi) & = e^\top_j \Big(   \left[ - \left( \alpha^\top \otimes I_k
\right) (\Omega \otimes \Omega) \;\;\;\;\; 0_{k \times k}\right]   +
\left[ 0_{k \times k^2} \;\;\; \Omega \right] \Big)\\
& = \label{eq::gj}
e^\top_j \Big( \left[ - \left( \alpha^\top \otimes I_k
\right) (\Omega \otimes \Omega) \;\;\;\;\; \Omega\right] \Big). 
\end{align}
The gradient of $g_j$ at $\psi$ is just the transpose of $Dg_j(\psi)$. 
Thus, the Jacobian of the function $g$  is
\begin{equation}\label{eq::GG}
\beta(j)/d\psi =  G = \left(
\begin{array}{c}
G_1^\top\\
\vdots\\
G_k^\top
\end{array}
\right).
\end{equation}

Next, we compute $Hg_j (\psi)$, the $b \times b$ Hessian of $g_j$ at $\psi$. 
Using the chain rule, 
\[
H g_j(\psi) =  D (D g_j(\psi)) =   (I_b \otimes e^\top_j) \frac{ d \; \mathrm{vec} \Big( \left[ - \left( \alpha^\top \otimes I_k
\right) (\Omega \otimes \Omega) \;\;\;\;\; \Omega\right] \Big)}{
d \psi},
\]
where the first matrix is of dimension $b \times kb$ and the second matrix is of
dimension $kb \times b$.
Then,
\begin{equation}\label{eq:H}
\frac{ d \; \mathrm{vec} \Big( \left[ - \left( \alpha^\top \otimes I_k
\right) (\Omega \otimes \Omega) \;\;\;\;\; \Omega\right] \Big)}{
d \psi} = \left[ 
    \begin{array}{c}
	-\frac{d   \left(   \alpha^\top \otimes I_k
\right) (\Omega \otimes \Omega) }{d \psi}\\
\;\; \\
	\frac{d \Omega}{d \psi}
    \end{array}
\right].
\end{equation}

The derivative at the bottom of the previous expression is
\[
    \frac{d \Omega}{d \psi} = \frac{d \Omega}{d \Sigma} \frac{d
    \Sigma}{d \psi}  =  - (\Omega \otimes \Omega) E = -
    (\Omega \otimes \Omega) [I_{k^2} \;\;\;\;\; 0_{k^2 \times k}]  =
    \Big[ - (\Omega \otimes \Omega) \;\;\;\;\; 0_{k^2 \times k} \Big].
\]
The top derivative in \eqref{eq:H} is more involved. By the product rule,
\[
\frac{d   \left(   \alpha^\top \otimes I_k
\right) (\Omega \otimes \Omega) }{d \psi} = 
    \Big( ( \Omega \otimes \Omega) \otimes I_k \Big) \frac{d
    (\alpha^\top \otimes I_k) }{d \psi} + \Big( I_{k^2} \otimes (\alpha^\top
    \otimes I_k) \Big) \frac{ d (\Omega \otimes \Omega)}{d
\psi}.
\]
The first derivative in the last expression is
\begin{align*}
\frac{d
(\alpha^\top \otimes I_k) }{d \psi}  & = 
\frac{d (\alpha^\top \otimes I_k) }{d \alpha} \frac{d \alpha}{d \psi} =
(I_k \otimes K_{1,k} \otimes I_k) (I_k \otimes \mathrm{vec}(I_k)) F \\
& = (I_k \otimes \mathrm{vec}(I_k) )F =
( I_k \otimes \mathrm{vec}(I_k)) \Big[0_{k \times k^2} \;\;\;\;\; I_k     \Big] \\
 & = \Big[0_{k^3 \times k^2} \;\;\;\;\;  I_k \otimes \mathrm{vec}(I_k) \Big] ,
\end{align*}
where $K_{k,1}$ is the appropriate  commutation matrix and the third identity follows since $K_{k,1} = I_k$ and, therefore, 
$(I_k \otimes K_{1,k} \otimes I_k)  = I_{k^3}$. 
Continuing with the second derivative in \eqref{eq:H},
\begin{align*}
\frac{ d (\Omega \otimes \Omega)}{d
\psi} & = \frac{ d (\Omega \otimes \Omega)}{d \Omega} \frac{d \Omega}{d \Sigma} \frac{d \Sigma}{d \psi}
= - J (\Omega \otimes \Omega) E \\
& = - J (\Omega \otimes \Omega) \Big[ I_{k^2} \;\;\;\;\; 0_{k^2 \times     k}\Big]
= -J \Big[ \Omega \otimes \Omega\ ;\;\;\;\; 0_{k^2 \times k}
\Big], 
\end{align*}
where
\[
    J =   \Big[ (I_k \otimes \Omega) \otimes I_{k^2}
    \Big] \Big( I_k \otimes K_{k,k} \otimes
    I_k \Big) \Big( I_{k^2} \otimes \mathrm{vec}(I_k) \Big) +  \Big[ I_{k^2} \otimes(
    \Omega \otimes I_k) \Big] \Big( I_k \otimes K_{k,k} \otimes
    I_k \Big) \Big( \mathrm{vec}(I_k) \otimes I_{k^2}  \Big).
\]
To see this,
notice that, by the product rule, we have
\[
J = \frac{d (\Omega\otimes
    \Omega)}{d \Omega}  = \frac{d (\Omega \otimes I_k)( I_k
    \otimes \Omega) } {d \Omega} = \Big[ (I_k \otimes \Omega) \otimes I_{k^2}
    \Big] \frac{d (\Omega \otimes I_k)}{d \Omega} + \Big[ I_{k^2} \otimes(
    \Omega \otimes I_k) \Big] \frac{d (I_k \otimes \Omega)}{d \Omega}.
\]
Next,
\[
    \frac{d (\Omega \otimes I_k)}{d \Omega}  = \Big( I_k \otimes K_{k,k} \otimes
    I_k \Big) \Big( I_{k^2} \otimes \mathrm{vec}(I_k) \Big) = \Big(
I_{k^2} \otimes K_{k,k} \Big) \Big(I_k \otimes \mathrm{vec}(I_k) \otimes I_k \Big)
\]
and
\[
    \frac{d (I_k \otimes \Omega )}{d \Omega}  = \Big( I_k \otimes K_{k,k} \otimes
    I_k \Big) \Big( \mathrm{vec}(I_k) \otimes I_{k^2}  \Big) = \Big(
K_{k,k} \otimes I_{k^2} \Big) \Big(I_k \otimes \mathrm{vec}(I_k) \otimes I_k
\Big).
\]
The formula for $J$ follows from the last three expressions.
Notice that $J$ is matrix of size $k^4 \times k^2$.
Finally, plugging the expressions for 
$\frac{d (\alpha^\top \otimes I_k) (\Omega \otimes \Omega) }{d \psi}$ and 
$\frac{ d \Omega }{d \psi}$  in \eqref{eq:H} we get that the Hessian
$H g_j(\psi)$ is 
\begin{equation}\label{eq::Hessian}
\frac{1}{2}\left(   (I_b \otimes e^\top_j) H + H^\top (I_b \otimes e_j)  \right)  
\end{equation}
where
\begin{equation}\label{eq:Halcazzo}
H =  \left[
    \begin{array}{c}
- \Big(  (\Omega \otimes \Omega) \otimes I_k \Big)  \Big[0_{k^3 \times
k^2} \;\;\;\;\;  I_k \otimes \mathrm{vec}(I_k) \Big]  +
 \Big( I_{k^2} \otimes (\alpha^\top
    \otimes I_k )\Big) J \Big[ \Omega \otimes \Omega \;\;\;\;\;
    0_{k^2 \times k}\Big]\\
 \;\\ 
 \Big[ - \Omega \otimes \Omega \;\;\;\;\; 0_{k^2 \times k} \Big]
    \end{array}
\right].
\end{equation}

So far we have ignored the facts that $\Sigma$ is symmetric.
Account for the symmetry,
the Hessian of $g_j(\psi)$ is
\[
D_h^\top  H g_j(\psi) D_h,
\]
where $D_h$ is the modified duplication matrix such that
$D \psi_h = \psi$,
with $\psi_h$ the vector comprised by the sub-vector of $\psi$ not including
the entries corresponding to the upper (or lower) diagonal entries of $\Sigma$.

We now prove the bounds \eqref{eq::B-and-lambda} and \eqref{eq:sigmamin}. We
will use repeatedly the fact that $\sigma_1(A \otimes B) = \sigma_1(A)
\sigma_1(B)$ and, for a vector $x$, $\sigma_1(x) = \|x\|$. For notational
convenience, we drop the dependence on $\psi$, since all our bounds hold
uniformly over all $\psi \in \mathcal{S}_n$.
The first bound in \eqref{eq::B-and-lambda} on the norm of the gradient of $g_j$
is straightforward:
\begin{align}
\nonumber
||G_j|| & \leq ||e_j|| \times   \sigma_1\left( \left[ - \left( \alpha^\top \otimes I_k
\right) (\Omega \otimes \Omega) \;\;\;\;\; \Omega\right] \right)\\
\nonumber
& \leq  \Big(
||\alpha||\times \sigma_1(\Omega)^2 + \sigma_1(\Omega) \Big)\\
\label{eq:Gj.constants}
& \leq \frac{A^2 \sqrt{ k}}{u^2} + \frac{1}{u}\\
    \nonumber
& \leq  
C \frac{\sqrt{k}}{u^2},
\end{align}
since  $\sigma_1(\Omega) \leq
\frac{1}{u}$, $\| \alpha \| \leq \sqrt{A^2 \mathrm{tr}(\Sigma)} \leq 
 A^2 \sqrt{k}$,  and we assume that $k \geq u^2$.

Turning to the second bound in \eqref{eq::B-and-lambda},
we will bound the largest singular
values of the individual terms in \eqref{eq::Hessian}.
First, for the lower block matrix in \eqref{eq:Halcazzo}, we have that 
$$
\sigma_1([ \Omega\otimes\Omega \;\;\;\;\; 0_{k^2 \times k}]) = 
\sigma_1( \Omega\otimes\Omega) = \sigma_1^2(\Omega) = 1/u^2.
$$
Next, we consider the two matrix in the upper block part of \eqref{eq:Halcazzo}.
For the first matrix we have that 
\begin{align}
    \label{eq:mammamia}
    \sigma_1 \Big( (\Omega\otimes\Omega\otimes I_k) \Big[ 0_{k^3 \times k^2 } \;\;\;\; I_k \otimes {\rm
vec}(I_k) \Big] \Big)&=
\sigma_1 \Big( \Big[ 0_{k^3 \times k^2} \;\;\;\; \Omega \otimes {\rm vec}(\Omega)
\Big]
\Big)\\
\nonumber
& = \sigma_1\left( \Omega \otimes \mathrm{vec}(\Omega)\right) \\
\nonumber
&=
\sigma_1 ( \Omega) \sigma_1({\rm vec}(\Omega))\\
\nonumber
& \leq \frac{\sqrt{k}}{u^2},
\end{align}
since
$$
\sigma_1({\rm vec}(\Omega)) = ||\Omega||_F =
\sqrt{\sum_{i=1}^k \sigma_i^2(\Omega)}\leq \sqrt{k}\sigma_1(\Omega) = \frac{\sqrt{k}}{u}.
$$
The identity in \eqref{eq:mammamia} is established using the following facts, valid for conformal matrices
$A$, $B$, $C$, $D$ and $X$:
\begin{itemize}
    \item $(A \otimes B)(C \otimes D) = AC \otimes BD$ , with $A =
	\Omega$, $B = \Omega \otimes I_k$, $C = I_k$ and $D =
	\mathrm{vec}(\Omega)$, and
    \item $AXB = C$ is equivalent to $\left(  B^\top \otimes A\right)
	\mathrm{vec}(X) = \mathrm{vec}(C)$, with $B=C = \Omega$ and $X = A =
	I_k$.
\end{itemize}
We now bound 
$\sigma_1 \Big([I_{k^2}\otimes \alpha^\top \otimes I_k] \; J \; [ \Omega\otimes\Omega
\;\;\;\; 0_{k^2\times k}] \Big)$, the second matrix in the upper block in
\eqref{eq:Halcazzo}.
We have that
\begin{align*}
\sigma_1(J) & \leq
2\sigma_1( (I_k \otimes \Omega \otimes I_{k^2})(I_k \otimes K_{k,k}\otimes I_k)(I_{k^2}\otimes {\rm vec}(I_k))\\
&=
2\sigma_1(\Omega)||I_k||_F\\
&= 2\sqrt{k}\sigma_1(\Omega), 
\end{align*}
since $\sigma_1(K_{k,k}) = 1$. 
Hence, using the fact that $\sigma_1([I_{k^2}\otimes \alpha^\top \otimes I_k]) =  ||\alpha||$,
$$
\sigma_1 \Big([I_{k^2}\otimes \alpha^\top \otimes I_k] \; J \; [ \Omega\otimes\Omega
\;\;\;\; 0_{k^2\times k}] \Big) \leq
2\sqrt{k} ||\alpha|| \sigma_1^3(\Omega) \leq 2 \sqrt{A U}  \frac{k}{u^3},
$$
since $\| \alpha \| \leq \sqrt{A U k}$.
Thus, we have obtained the following bound for the
largest singular value of the matrix $H$ in \eqref{eq:Halcazzo}:
$$
\sigma_1(H)\leq 
C \Big( \frac{1}{u^2} + \frac{\sqrt{k}}{u^2}+ \frac{k}{u^3} \Big),
$$
where $C$ is a positive number depending on $A$ only.
Putting all the pieces together,
\begin{align*}
\sigma_1(H_j) &=
\sigma_1\left( \frac{1}{2}((I_b \otimes e_j)H + H^\top (I_b\otimes e_j))\right)\\
& \leq
\sigma_1((I_b \otimes e_j)H)\\
& \leq
\sigma_1(I_b)\sigma_1(e_j)\sigma_1(H)\\
& \leq C \Big( \frac{1}{u^2} + \frac{\sqrt{k}}{u^2}+ \frac{k}{u^3} \Big).
\end{align*}
Whenever $u \leq \sqrt{k}$, the dominant term in the above expression is
$\frac{k}{u^3}$.
This gives the bound on $\overline{H}$ in (\ref{eq::B-and-lambda}).
The bound on $\underline{\sigma}$ given in \eqref{eq:sigmamin} follows
from \eqref{eq:Gj}. Indeed, for every $P \in \mathcal{P}^{\mathrm{OLS}}$
\[
 \min_j \sqrt{ G_j V G_j^\top}
 \geq   \sqrt{v} \min_j \| G_j \|.
\]
 Then, using \eqref{eq:Gj},
\[ 
    \min_j \| G_j \| \geq  \min_j \| \Omega_j \| \geq \lambda_{\min}(\Omega) =  \frac{1 }{
    	U },
\]
where $\Omega_j$ denotes the $j^{\mathrm{th}}$ row of $\Omega$.

The final value of the constant $C$ depends only on $A$ and $U$, and since $U \leq A$, we can
reduce the dependence of such constant on $A$ only.

$\Box$

\section{Appendix 6: Anti-concentration and comparison 
bounds for maxima of Gaussian random vectors and Berry-Esseen bounds for
polyhedral sets}
\label{app:high.dim.clt}

Now
we collect some results that can be are derived from
\cite{chernozhukov2015comparison}, \cite{cherno2} and \cite{nazarov1807maximal}.
However, our statement of the results is slightly different
than in the original papers.
The reason for this is that we need to keep track
of some constants in the proofs that affect our rates.

The following anti-concentration result for the maxima of Gaussian vectors follows from Lemma A.1 in
\cite{cherno2} and relies on a deep result in \cite{nazarov1807maximal}. 

\vspace{11pt}

\begin{theorem}[Anti-concentration of Gaussian maxima]
    \label{thm:anti.concentration}
    Let $(X_1\ldots,X_p)$ be a centered Gaussian vector in $\mathbb{R}^p$
    with $\sigma_j^2 = \mathbb{E}[X_j^2] > 0$ for all $j=1,\ldots,p$. Moreover,
	let $\underline{\sigma} = \min_{1 \leq j \leq p} \sigma_j$. Then, for
	any $y = (y_1,\ldots,y_p) \in \mathbb{R}^p$ and $a > 0$ 
	\[
	    \mathbb{P}(  X_j  \leq y_j + a, \forall j) - \mathbb{P}( X_j
		\leq y_j, \forall j) \leq \frac{a}{\underline{\sigma}} \left( \sqrt{2 \log p} + 2 \right).
	\]
\end{theorem}

The previous result implies that, for any $a > 0$ and $y = (y_1,\ldots,y_p) \in
\mathbb{R}^p_+$, 
	\[
	    \mathbb{P}(  |X_j|  \leq y_j + a, \forall j) - \mathbb{P}( |X_j|
		\leq y_j, \forall j) \leq \frac{a}{\underline{\sigma}} \left(
		\sqrt{2 \log 2p} + 2 \right)
	\]
and that, for any $y > 0$,
	\[
	    \mathbb{P}( \max_j |X_j | \leq y + a) - \mathbb{P}(\max_j |X_j|
		\leq y) \leq \frac{a}{\underline{\sigma}} \left( \sqrt{2 \log 2p} + 2 \right).
	\]

The following high-dimensional central limit theorem follows from Proposition
2.1 in \cite{cherno2} and \Cref{thm:anti.concentration}. Notice that we have
kept the dependence on the minimal variance explicit. 

\vspace{11pt}

\begin{theorem}[Berry-Esseen bound for simple convex sets]
\label{thm:high.dim.clt}
Let $X_1,\ldots,X_n$ be independent centered random vectors in $\mathbb{R}^p$.
Let $S^X_n = \frac{1}{\sqrt{n}} \sum_{i=1}^n X_i$ and, similarly, let $S^Y_n
= \frac{1}{n} \sum_{i=1}^n Y_i$, where $Y_1,\ldots, Y_n$ are independent
vectors with $Y_i \sim N_p(0,\mathbb{E}[X_i X_i^\top])$.
Let $\mathcal{A}$ be the collection of polyhedra $A$ in $\mathbb{R}^p$ of the
form
$$
A =\left\{ x \in \mathbb{R}^d \colon v^\top x \leq t_v , v \in
\mathcal{V}(\mathcal{A}) \right\}
$$
where $ \mathcal{V}(\mathcal{A}) \subset \mathbb{R}^p$ is a set of $m$
points of unit norm, with $m \leq (n p)^d$ for some constant $d>0$,
and $( t_v \colon v \in \mathcal{V}(\mathcal{A}) )$ is a set of $m$
positive numbers. For each $i=1,\ldots,n$ let 
$$
\tilde{X}_i =
(\tilde{X}_{i1},\ldots,\tilde{X}_{im})^\top = \left( v^\top X_i, v \in
\mathcal{V}(\mathcal{A}) \right).
$$
Assume that the following conditions are satisfied, for some $B_n \geq 1$ and
$\underline{\sigma}>0$:
    \begin{enumerate}
	\item[(M1')] $n^{-1} \sum_{i=1}^n \mathbb{E}\left[ \tilde{X}_{ij}^2
	    \right] \geq
	    \underline{\sigma}^2$, for all $j=1,\ldots, m$;
	\item[(M2')] $n^{-1} \sum_{i=1}^n \mathbb{E}\left[ |
	    \tilde{X}_{ij}|^{2+k} \right] \leq B^{k}_n$, for all
	    $j=1,\ldots,m$ and $k=1,2$;
	\item[(E1')] $\mathbb{E}\left[ \exp\left( | \tilde{X}_{i,j} | / B_n
	    \right) \right] \leq 2$, for $i=1,\ldots,n$ and $k=1,2$.
    \end{enumerate}
Then, there exists a constant $C>0$ depending only on $d$ such that 
\[
	\sup_{A \in \mathcal{A}} \left|\mathbb{P}(S^X_n \in A) -
	\mathbb{P}(S^Y_n \in A)  \right| \leq \frac{C}{\underline{\sigma}}
	\left( \frac{B_n^2 \log^7(pn) }{n } \right)^{1/6}.
\]
\end{theorem}

Finally, we make frequent use the following comparison theorem for the maxima of Gaussian
vectors.  Its proof can be established using arguments from the proof of Theorem
4.1 in \cite{cherno2} -- which itself relies on a modification of Theorem 1 from \cite{chernozhukov2015comparison} -- along with the above anti-concentration bound of
    \Cref{thm:anti.concentration}.
As usual,  we have
kept the dependence on the minimal variance explicit. 

\vspace{11pt}

\begin{theorem}[Gaussian comparison]
    \label{thm:comparisons}
Let $X \sim N_p(0,\Sigma_X)$ and $Y \sim N_p(0,\Sigma_Y)$ with
\[
    \Delta = \max_{i,j} | \Sigma_X(j,k) - \Sigma_Y(j,k)|
\]
Let $\underline{\sigma}^2 = \max\{  \min_j \Sigma_X(j,j) ,  
\min_j \Sigma_Y(j,j) \}$. Then, there exists a universal constant $C>0$ such that
\[
    \sup_{t \in \mathbb{R}^p} \left| \mathbb{P}( X \leq t) -
    \mathbb{P}( Y \leq t)    \right| \leq C \frac{\Delta^{1/3} (2 \log
    p)^{1/3}}{\underline{\sigma}^{2/3}}.
    \]
\end{theorem}

{\bf Remark.}  The above result further implies that
\[
\sup_{t >0 } \left| \mathbb{P}( \| X \|_\infty \leq t) -
    \mathbb{P}( \| Y \|_\infty \leq t)    \right| \leq 2 C \frac{\Delta^{1/3} (2 \log
    p)^{1/3}}{\underline{\sigma}^{2/3}},
\]
which corresponds to the original formulation of the Gaussian comparison theorem of \cite{chernozhukov2015comparison}.

\section{Appendix 7: The Procedures}

\begin{figure}
\rule{7in}{.5mm}
\begin{center}
{\sf Boot-Split}
\end{center}
{\sc Input}: Data ${\cal D} = \{(X_1,Y_1),\ldots, (X_{2n},Y_{2n})\}$.
Confidence parameter $\alpha$.  Constant $\epsilon$ (Section
\ref{sec:loco.parameters}).\\
{\sc Output}: Confidence set $\hat{C}^*_{\wS}$ for $\beta_{\wS}$ and
$\hat{D}^*_{\wS}$ for $\gamma_{\wS}$.

\begin{enum}
\item Randomly split the data into two halves
  ${\cal D}_{1,n}$ and ${\cal D}_{2,n}$.
\item Use ${\cal D}_{1,n}$ to select a subset of variables $\wS$.
This can be forward stepwise, the lasso, or any other method.
Let $k= |\wS|$.
\item Write ${\cal D}_{2,n}=\{(X_1,Y_1),\ldots, (X_n,Y_n)\}$.
  Let $P_n$ be the empirical distribution of ${\cal D}_{2,n}$.
\item For $\beta_{\wS}$:
\begin{enum}
\item Get $\hat\beta_{\wS}$ from ${\cal D}_{2,n}$ by least squares.
\begin{enum}
\item Draw $(X_1^*,Y_1^*),\ldots, (X_m^*,Y_m^*) \sim P_n$.
  Let $\hat\beta^*_{\wS}$ be the estimator constructed from the bootstrap sample.
\item Repeat $B$ times to get $\hat\beta_{\wS,1}^*, \ldots,\hat\beta_{\wS,B}^*$.
\item Define $\hat{t}_\alpha$ by
$$
\frac{1}{B}\sum_{b=1}^B I\Bigl(\sqrt{n}||\hat\beta_{\wS,b}^* - \hat\beta_{\wS}||_\infty > \hat{t}_\alpha\Bigr) = \alpha.
$$
\end{enum}
\item Output: $\hat{C}^*_{\wS} = \{ \beta\in\mathbb{R}^k:\ ||\beta-\hat\beta_{\wS}||_\infty\leq \hat{t}_\alpha/\sqrt{n}\}$.
\end{enum}
\item For $\gamma_{\wS}$:
\begin{enum}
\item Get $\hat\beta_{\wS}$ from ${\cal D}_{1,n}$.
This can be any estimator.
For $j\in \hat{S}$ let
$\hat\gamma_{\hat{S}}(j) = \frac{1}{n}\sum_{i=1}^n r_i$
where
$r_i = (\delta_i(j) + \epsilon \xi_i(j))$,
$\delta_i(j) =  |Y_i - \hat\beta_{\wS,j}^\top X_i|  - |Y_i - \hat\beta_{\wS}^\top X_i|$ and
$\xi_i(j)\sim {\rm Unif}(-1,1)$.
Let $\hat\gamma_{\wS} = (\hat\gamma_{\wS}(j):\ j\in \wS)$.
\begin{enum}
\item Draw $(X_1^*,Y_1^*),\ldots, (X_n^*,Y_n^*) \sim P_n$.
\item Let
$\hat\gamma_{\wS}(j) = \frac{1}{n}\sum_{i=1}^n r_i^*$.
Let $\hat\gamma_{\wS}^* = (\hat\gamma_{\wS}^*(j):\ j\in \wS)$.
\item Repeat $B$ times to get
$\hat\gamma_{\wS,1}^*, \ldots, \hat\gamma_{\wS,B}^*$.
\item Define $\hat{u}_\alpha$ by
$$
\frac{1}{B}\sum_{b=1}^B I\Bigl(\sqrt{n}||\hat\gamma_{\wS,b}^* - \hat\gamma_{\wS}||_\infty > \hat{u}_\alpha\Bigr) = \alpha.
$$
\end{enum}
\item Output: $\hat{D}^*_{\wS} = \{ \gamma_{\wS} \in\mathbb{R}^k:\ ||\gamma_{\wS}-\hat\gamma_{\wS}||_\infty\leq \hat{u}_\alpha/\sqrt{n}\}$.
\end{enum}
\end{enum}
\caption{\emph{The Boot-Split Algorithm.}}
\label{fig::boot-split}
\rule{7in}{.5mm}
\end{figure}

\begin{figure}
\rule{7in}{.5mm}
\begin{center}
{\sf Normal-Split}
\end{center}
{\sc Input}: Data ${\cal D} = \{(X_1,Y_1),\ldots, (X_{2n},Y_{2n})\}$.
Confidence parameter $\alpha$. Threshold and variance parameters $\tau$ and $\epsilon$ (only for $\gamma_{\wS}$).\\
{\sc Output}: Confidence set $\hat{C}_{\wS}$ for $\beta_{\wS}$ and
$\hat{D}_{\wS}$ for $\gamma_{\wS}$.

\begin{enum}
\item Randomly split the data into two halves
  ${\cal D}_{1,n}$ and ${\cal D}_{2,n}$.
\item Use ${\cal D}_{1,n}$ to select a subset of variables $\wS$.
This can be forward stepwise, the lasso, or any other method.
Let $k= |\wS|$.
\item For $\beta_{\wS}$:
\begin{enum}
\item Get $\hat\beta_{\wS}$ from ${\cal D}_{2,n}$ by least squares.
\item Output $\hat{C}_{\wS} = \bigotimes_{j\in \wS} C(j)$ where
$C(j) = \hat\beta_{\wS}(j) \pm z_{\alpha/(2k)} \sqrt{\hat\Gamma_n(j,j)}$
where $\hat\Gamma$ is given by (\ref{eq::Ga}).
\end{enum}
\item For $\gamma_{\wS}$:
\begin{enum}
\item Get $\hat\beta_{\wS}$ from ${\cal D}_{1,n}$.
This can be any estimator.
For $j\in \wS$ let
$\hat\gamma_{\wS}(j) = \frac{1}{n}\sum_{i=1}^n r_i$
where
$r_i = (\delta_i(j) + \epsilon \xi_i(j))$,
$\delta_i(j) =  \left| Y_i - t_{\tau} \left( \hat\beta_{\wS,j}^\top X_i \right) \right|  - \left|Y_i - t_{\tau} \left( \hat\beta_{\wS}^\top X_i \right)\right|$ and
$\xi_i(j)\sim {\rm Unif}(-1,1)$.
Let $\hat\gamma_{\wS} = (\hat\gamma_{\wS}(j):\ j\in \wS)$.
\item Output $\hat{D}_{\wS} = \bigotimes_{j\in \wS} D(j)$ where
$D(j) = \hat\gamma_{\wS}(j) \pm z_{\alpha/(2k)}\hat{\Sigma}(j,j)$,
with $\hat{\Sigma}(j,j)$ given by \eqref{eq:Sigma.loco}.
\end{enum}
\end{enum}
\caption{\emph{The Normal-Split Algorithm.}}
\label{fig::normal-split}
\rule{7in}{.5mm}
\end{figure}

\begin{figure}
\rule{7in}{.5mm}
\begin{center}
{\sf Median-Split}
\end{center}
{\sc Input}: Data ${\cal D} = \{(X_1,Y_1),\ldots, (X_{2n},Y_{2n})\}$.
Confidence parameter $\alpha$.\\
{\sc Output}: Confidence set $\hat{E}_{\wS}$.

\begin{enum}
\item Randomly split the data into two halves
  ${\cal D}_{1,n}$ and ${\cal D}_{2,n}$.
\item Use ${\cal D}_{1,n}$ to select a subset of variables $\wS$.
This can be forward stepwise, the lasso, or any other method.
Let $k= |\wS|$.
\item Write ${\cal D}_{2,n}=\{(X_1,Y_1),\ldots, (X_n,Y_n)\}$.
  For $(X_i,Y_i)\in {\cal D}_{2,n}$ let
$$
W_i(j) = 
|Y_i-\hat\beta_{\wS,j}^\top X_i| - |Y_i-\hat\beta_{\wS}^\top X_i|,
$$
\item Let $W_{(1)}(j) \leq \cdots \leq W_{(n)}(j)$ be the order statistics
and let
$E(j) = [W_{(n-k_2)}(j),W_{(n-k_1+1)}(j)]$
where
$$
k_1 =  \frac{n}{2} + \sqrt{n \log\left( \frac{2k}{\alpha}\right)},\ \ \ 
k_2 =  \frac{n}{2} - \sqrt{n \log\left( \frac{2k}{\alpha}\right)}.
$$
\item Let $\hat{E}_{\wS} = \bigotimes_{j\in S} E(j)$.
\end{enum}
\caption{\emph{The Median-Split Algorithm.}}
\label{fig::median-split}
\rule{7in}{.5mm}
\end{figure}

\end{document}